\documentclass{amsart}

\usepackage[utf8]{inputenc}

\everymath{\displaystyle}
\usepackage{geometry}
\usepackage{graphicx} 
\usepackage{amssymb}
\usepackage{amsmath}
\usepackage{dsfont}
\usepackage{ulem}

\usepackage{xcolor}
\usepackage{tcolorbox}

\theoremstyle{definition}
\newtheorem{definition}{Definition}[section]

\theoremstyle{theorem}
\newtheorem{theorem}{Theorem}[section]
\newtheorem{corollary}[theorem]{Corollary}
\newtheorem{lemma}[theorem]{Lemma}

\newtheorem{assumption}[theorem]{Assumption}

\theoremstyle{remark}
\newtheorem{remark}{Remark}[section]
\usepackage{bm}
\let\<\langle
\let\>\rangle
\usepackage{hyperref}

\usepackage{enumerate}
\numberwithin{equation}{section}
\definecolor{bostonuniversityred}{rgb}{0.8, 0.0, 0.0}
\newcommand{\amit}[1]{{\color{black}#1}}
\definecolor{blue(ryb)}{rgb}{0.01, 0.28, 1.0}

\newcommand{\norm}[1]{\left\lVert#1\right\rVert}
\newcommand{\abs}[1]{\left\lvert#1\right\rvert}
\newcommand{\pa}[1]{\left( #1 \right)}
\newcommand{\rpa}[1]{\left[ #1 \right]}
\newcommand{\br}[1]{\left\lbrace #1\right\rbrace}

\newcommand{\R}{\mathbb{R}}

\newcommand{\N}{\mathbb{N}}
\newcommand{\PP}{\mathcal{P}}

\newcommand{\cP}{\mathcal{P}}

\newcommand{\adm}{\mathrm{adm}}
\newcommand{\cadm}{\mathrm{cADM}}
\newcommand{\pcadm}{\mathrm{pcADM}}
\newcommand{\A}{\mathcal{A}}
\newcommand{\D}{\mathcal{D}}
\newcommand{\CC}{\mathcal{C}}
\renewcommand{\div}{\mathrm{div}}

\newcommand{\wt}{\widetilde}
\newcommand{\jj}{\mathfrak{j}}
\newcommand{\sign}{\mathrm{sign}}
\newcommand{\X}{\mathcal{X}}
\newcommand{\EE}{\mathcal{E}}
\newcommand{\e}{\mathfrak{e}}

\newcommand{\dd}{{\mathrm{d}}}

\title[Static and dynamic optimal transport governed by linear control systems]{On the equivalence between static and dynamic optimal transport governed by linear control systems}
\author[A. Einav]{Amit Einav}
\address{Amit Einav \hfill\break 
	Department of Mathematical Sciences, Durham University, Upper Mountjoy Campus, Stockton Road DH1 3LE Durham, United Kingdom}
\email{amit.einav@durham.ac.uk}
\author[Y. Jiang]{Yue Jiang}
\address{Yue Jiang \hfill\break 
	Department of Mathematical Sciences, Durham University, Upper Mountjoy Campus, Stockton Road DH1 3LE Durham, United Kingdom}
\email{yue.jiang@durham.ac.uk}
\author[A.R. M\'esz\'aros]{Alp\'ar R. M\'esz\'aros}
\address{Alp\'ar R. M\'esz\'aros \hfill\break 
	Department of Mathematical Sciences, Durham University, Upper Mountjoy Campus, Stockton Road DH1 3LE Durham, United Kingdom}
\email{alpar.r.meszaros@durham.ac.uk}

\begin{document}
\maketitle

\begin{abstract}
	In this paper we revisit a class of optimal transport problems associated to non-autonomous linear control systems. Building on properties of the cost functions on $\R^{d}\times\R^{d}$ derived from suitable variational problems, we show the equivalence between the static and dynamic versions of the corresponding transport problems. Our analysis is constructive in nature and relies on functional analytic properties of the {\it end-point map} and the fine properties of the optimal control functions. These lead to some new quantitative estimates which play a crucial role in our investigation. 
\end{abstract}
\vspace{1cm}

	

\vspace{1cm}

\section{Introduction}\label{sec:intro}

The theory of optimal transport has witnessed a great success in the past three decades due to its far reaching applications and unexpected connections to multiple mathematical fields. We refer to the monographs \cite{Villani03, AGS08, Villani08, Filippo15} for a  thorough description of the theory.

One such connection between PDEs, geometry and mathematical physics was initiated by Benamou--Brenier in \cite{BB00}, giving a dynamic, fluid mechanical description of the classical Monge--Kantorovich optimal transport problem. Let $p>1$ be given and denote by $\cP_p(\R^d)$ the set of Borel probability measures supported on $\R^d$ having finite $p^{th}$-moments. The classical $p$-Wasserstein distance $W_p:\cP_p(\R^d)\times\cP_p(\R^d)\to[0,+\infty)$ is defined as 
\begin{equation}\label{def:W_p}
W_p(\mu, \nu):=\inf\left\{\int_{\R^d\times\R^d}|x-y|^p\dd\pi(x,y):\ \ \pi\in\Pi(\mu,\nu)\right\}^{\frac1p},
\end{equation}
where $\Pi(\mu,\nu):=\left\{\pi\in\cP_p(\R^d\times\R^d): (\pi^x)_\sharp\pi = \mu; (\pi^y)_\sharp\pi = \nu\right\}$ stands for the set of transport or transference plans, and $\pi^x,\pi^y:\R^d\times\R^d\to\R^d$ denote the canonical projections, i.e. $\pi^x(x,y) = x$ and $\pi^y(x,y) = y$.

Proven originally in \cite{BB00} for $p=2$, and later extended for general $p>1$ (cf. \cite{AGS08}; see also \cite{Jimenez08}), we have that $W_p$ can be equivalently characterised by
\begin{equation}
\label{bb formula introduction}
    W_p(\mu,\nu)=\inf_{(\rho,v)\in \adm(\mu,\nu)}\left\{\int_0^1\int_{\mathbb{R}^d}|v(t,x)|^p\,\dd\rho_t(x)\,\dd t\right\}^{\frac1p},
\end{equation}
where $\adm(\mu,\nu)$ denotes the set of pairs $(\rho_{t},v_{t})_{t\in[0,1]}$ with $\rho:[0,1]\to \mathcal{P}_p(\mathbb{R}^d)$ a narrowly continuous curve and $v:[0,1]\times\mathbb{R}^d\to \mathbb{R}^d$ is a time-dependent Borel vector field with 
\begin{equation}\nonumber
			\int_0^1\int_{\mathbb{R}^d}|v_t(x)|^p\,\dd\rho_t(x)\,\dd t<\infty,
\end{equation}
such that the continuity equation 
\begin{equation}
	\label{eq:continuity_equation}
	\left\{
	\begin{array}{ll}
	\partial_t\rho_t +  \div (\rho_t v_t)=0, &  {\rm{in}}\ \ (0,1)\times\R^d,\\
	\rho_0=\mu,\quad \rho_1=\nu,
	\end{array}
	\right.
\end{equation}
is satisfied in the sense of distributions on $\pa{0,1}\times \R^d$.
While the classical Monge--Kantorovich problem described in \eqref{def:W_p} is a `static' problem, its counterpart in \eqref{bb formula introduction} is a `dynamic' one (having an Eulerian perspective). Beyond the connection to fluid mechanics, \eqref{bb formula introduction} reveals a deep geometric feature of the metric space $(\cP_p(\R^d),W_p)$, as this problem is nothing but a geodesic problem, which can be equivalently written as 
$$
W_p(\mu,\nu) = \inf\left\{\int_0^1 |\rho'|_{W_p}(t)\dd t: \rho:[0,1]\to\cP_p(\R^d), \rho_0 = \mu, \rho_1=\nu\right\},
$$
where $|\rho'|_{W_p}(t)$ stands for the metric derivative of the curve $(\rho_s)_{s\in[0,1]}$ at $t$. Therefore, \eqref{bb formula introduction} precisely selects constant speed geodesics connecting $\mu$ to $\nu$ in $(\cP_p(\R^d),W_p)$, and in fact if $(\rho_t,v_t)_{t\in[0,1]}$ is optimal then $\|v_t\|_{L^p_{\rho_t}} = |\rho'|_{W_p}(t)$ for almost every $t\in[0,1]$.

\medskip

Starting from this connection between \eqref{def:W_p} and \eqref{bb formula introduction}, one can naturally ask the question whether this extends to more complex settings, such as in the case of curved reference space instead of $\R^{d}$ or more general cost functions \amit{$c(x,y)=\abs{x-y}^{p}$ for $p\in (1,\infty)$}. 

For a given lower semicontinuous and bounded below cost function $c:\R^{d}\times\R^{d}\to\R$, and measures $\mu,\nu\in\cP(\R^d)$ (having suitable moment bounds) the direct method of the calculus of variations immediately gives the existence of an optimiser for the problem 
$$
\inf\left\{\int_{\R^d\times\R^d}c(x,y)\dd\pi(x,y):\ \ \pi\in\Pi(\mu,\nu)\right\}.
$$
However, it seems to be much more challenging, in general, to find a dynamic equivalent for this for generic cost functions $c$.

When the cost function is derived from a Lagrangian action, i.e.
$$
c(x,y):=\inf\left\{\int_0^1L(s,\gamma(s),\dot \gamma(s))\dd s:\ \gamma(0) = x, \gamma(1)=y\right\},
$$
where $L:[0,1]\times TM\to \R$ is a given Lagrangian, defined on the product of the time interval $[0,1]$ and the tangent bundle of a manifold $M$, breakthrough results were obtained in \cite{Bernard06, Bernard07} and \cite{DePascale06}. These results propose not only an abstract dynamic transport problem equivalent of the static problem via the concept of Young measures, but made interesting connections with the Mather theory in Hamiltonian mechanics. We also mention the work \cite{Pratelli05}, in a similar context.

\medskip

A common feature in all the aforementioned models is that at the microscopic level individual particles are optimising their velocities, and are trying to minimise a global kinetic-type energy together. One can imagine situations, however, when because of particles being constrained in particular geometric settings, they have  `restricted' directions. An interesting geometric framework for this is the case of sub-Riemannian geometry, when the allowable directions are given by a subset of the possible directions from the tangent space of the underlying manifold. Using control theoretic language, in such situations the dimension of the space where the controls live is often strictly smaller than the dimension of the state space. Optimal transport problems, therefore, must take such restrictions into account (see for instance \cite{Ambrosio04, Figalli08, AgrLee:09, FigRif:10, HPR11, Rifford14, CGP17, ELO23, EL24, ElaJac:25}).

\medskip

\subsection*{The setting of the problem considered in this paper and our main results}

In this work, we consider the optimal transport problem associated to linear control systems of the form 
\begin{align}\label{eq:intro_cont}
\gamma'(t) = M(t)\gamma(t)+N(t)\alpha(t),\ \  t\in(0,T),
\end{align}
where $T>0$ is a given time horizon and $M:[0,T]\to\mathbb{R}^{d\times d},\ \ N:[0,T]\to\mathbb{R}^{d\times n}$ are two given matrix-valued curves with $d,n\in\N$ such that $1\leq n\leq d$. Here $\gamma:[0,T]\to\R^d$ represents the state variable, while $\alpha:[0,T]\to\R^n$ stands for the control. For $p>1$, we define the cost function $c_p:\R^d\times\R^d\to[0,+\infty)$ via
\begin{equation}\label{intro:c_p}
c_p(x,y):=\inf\left\{\int_0^T |\alpha(t)|^p\dd t:\ (\gamma,\alpha)\in\pcadm(x,y) \right\},
\end{equation}
where
$$
\pcadm(x,y):= \left\{\ (\gamma,\alpha)\in W^{1,p}\pa{0,T ;\R^d}\times L^p\pa{0,T;\R^n}\ {\rm{solves}}\ \eqref{eq:intro_cont}\ {\rm{and}}\ \gamma(0) = x,\ \gamma(T) = y\right\}.
$$ 
Throughout the paper we impose classical sufficient assumptions (that we detail later; cf. \cite{Sontag98}) which will allow the controllability of the above system. These in particular imply that for any $x,y\in\R^d$ we have $\pcadm(x,y)\neq\emptyset$.

\medskip

We are concerned with the equivalence between the `static' and `dynamic' (or `Benamou--Brenier type') optimal transport problems associated to the cost function $c_p$. Let $\mu,\nu\in\cP_p(\R^d)$ be given. The {\it static} problem is defined as 
\begin{equation}\label{intro:static}
\CC_p\pa{\mu,\nu}:=\inf_{\pi\in \Pi(\mu,\nu)}\int_{\mathbb{R}^d\times\mathbb{R}^d} c_p(x,y)\dd\pi(x,y).
\end{equation}
The set $\pcadm(x,y)$ of admissible path and control pairs naturally generate pairs of measure flows and associated controls connecting two probability measures. Indeed, for $\mu,\nu\in\cP_p(\R^d)$ define
 \begin{align*}
\cadm(\mu,\nu):=\left\{(\rho_t,u_t)_{t\in[0,T]}\ \ {\rm{solves}}\ \ \eqref{eq:continuity_extended_explicit}\ \ {\rm{and}}\ \rho_0 = \mu,\rho_T = \nu\right\}, 
\end{align*} 
where $\rho:[0,T]\to \PP_p\pa{\R^d}$ is a narrowly continuous curve such that $\displaystyle \int_0^{\amit{T}}\int_{\R^d}|x|^p\dd\rho_t(x)\dd t<+\infty$ and $u:[0,T]\times\R^d\to\R^n$ is a Borel vector field such that $\displaystyle \int_0^1\int_{\R^d}|u_t(x)|^p\dd\rho_t(x)\dd t<+\infty$, and
\begin{equation}\label{eq:continuity_extended_explicit}
		\left\{
		\begin{array}{ll}
			\partial_t\rho_t +  \div (\rho_t b(t,\cdot,u_t(\cdot)))=0, & {\rm{in}}\ (0,T)\times\R^d,\\[5pt]
			b(t,x,u_t(x)):=M(t)x + N(t)u_t(x), & {\rm{a.e.}}\ (t,x)\in (0,T)\times\R^d
		\end{array}
		\right.
	\end{equation}
is understood in the sense of distributions on $(0,T)\times\R^d$. With this definition at hand we introduce the {\it dynamic} optimal transport problem for $\mu,\nu\in\cP_p(\R^d)$ as 
\begin{equation}\label{controlled p action eq}
		\mathcal{D}_p(\mu,\nu):= \inf_{(\rho,u)\in \cadm(\mu,\nu)}\int_0^T \int_{\mathbb{R}^d}|u(t,x)|^p \,\dd\rho_t(x)\dd t.
\end{equation}

The fact that $\cadm\pa{\mu,\nu}$ is non-empty is nontrivial, and this is the subject of the first result of this paper.

\begin{theorem}\label{thm:not_empty}
 Under our standing controllability assumptions (cf. Assumption \ref{asmp} below) the set $\cadm\pa{\mu,\nu}$ is not empty for any $\mu,\nu\in \PP_p\pa{\R^d}$. Consequently, $\D_p\pa{\mu,\nu}<+\infty$.
\end{theorem}

The main result of our paper is the equivalence between the static and dynamic problems, and this is the subject of the following theorem.

\begin{theorem}\label{thm:general_bb}
Let $\mu,\nu\in \mathcal{P}_p(\mathbb{R}^d)$ with $p>1$ and let $M:[0,T]\to \mathbb{R}^{d\times d}$ and $ N:[0,T]\to \mathbb{R}^{d\times n}$ satisfy the standing controllability assumptions (cf. Assumption \ref{asmp} below). Then, the following equivalence holds:
		\begin{equation*}
			\label{eq:generalised_bb}
			\min_{\pi\in \Pi(\mu,\nu)}\int_{\mathbb{R}^d\times\mathbb{R}^d} c_p(x,y)\dd\pi(x,y)=\CC_p\pa{\mu,\nu} =
			\D_p\pa{\mu,\nu}=\min_{(\rho,u)\in \cadm(\mu,\nu)}\int_0^T \int_{\mathbb{R}^d}|u(t,x)|^p \,\dd\rho_t(x)\dd t.
		\end{equation*}
In particular, the minimisers for both variational problems exist.
\end{theorem}

In the proof of this theorem we rely on certain fine properties of the optimal controls in the definition of $c_p$, and these are collected in the following theorem.

\begin{theorem}
\label{thm:optimal_control_cost}
	\leavevmode
	Let $M:[0,T]\to \mathbb{R}^{d\times d}$, $ N:[0,T]\to \mathbb{R}^{d\times n}$ satisfy our standing assumptions (cf. Assumption \ref{asmp} below). 
	\begin{enumerate}[(i)]
		\item\label{item:pcADM} For any given $x,y\in\mathbb{R}^d$ and $T>0$, we have that  
		$\pcadm(x,y)\not=\emptyset$. 
		\item 
		There exists a unique $\alpha^*_p(\cdot;x,y)\in L^p\pa{0,T;\R^n}$ such that
		\begin{equation}\nonumber
			c_p(x,y)=\min_{\pa{\gamma,\alpha}\in \pcadm(x,y)}\int_0^T |\alpha(t)|^p \dd t=\int_0^T \abs{\alpha^\ast(t;x,y)}^p \dd t.
		\end{equation}
		Moreover, $\alpha^\ast_p\pa{\cdot;\cdot,\cdot}\in C\pa{[0,T]\times \R^d\times \R^d ;\R^n}$ and $c_p^{\frac{1}{p}}$ is globally Lipschitz continuous on $\R^d\times \R^d$. Consequently, $c_p$ is continuous on $\R^d\times \R^d$. 
	\end{enumerate}
\end{theorem}

\begin{remark}\label{rem:left_hand_side_not_a_distance}
	It is important to note that left hand side of \eqref{controlled p action eq} is in general \textit{not} the $p-$power of a metric, as in general $c_p(x,x)\neq 0$. We will see, however, that $c_p(x,y)$ is intimately connected to the Euclidean distance between $x$ and the \emph{end point} of the flow of the homogeneous version of \eqref{eq:intro_cont} starting at $y$.
\end{remark}

\subsection*{The existing literature in connection with our main results}

Beyond the deep results in the framework of sub-Riemannian geometry, optimal transport associated to general control systems have received a huge attention in the last decade or so in the applied mathematics community. Indeed, these models can encapsulate various phenomena linked to multi-agents systems, mean field type optimal control problems, and mean field games. We refer to the non-exhaustive list of recent works \cite{CGP17, CM18, FLOS2019, BivQui:20, JimMarQui:20, BonFra:21, BonRos:21, BonFra:22, CLOS22, GriMes:22, ELO23, EL24, ElaJac:25, CraElaLee:25, AmbGriMes:25} and to the references therein.

The closest works studying problems similar to the ones considered in our paper are \cite{CGP17, ELO23} and \cite{EL24}. We now summarise the main results from these works, as follows.

\begin{itemize}
\item The work \cite{CGP17} considers the equivalence between the static and dynamic problems \eqref{intro:static} and \eqref{controlled p action eq} in the purely quadratic case $p=2$. Here the linear-quadratic setting gives simplifications, and allows to compute optimal controls explicitly. The authors have also introduced a `stochastic version' of the transport problems, and the equivalence result between the static and dynamic problems remains conditional to the fact that this stochastic problem is equivalent to the dynamic problem.
\item The authors in \cite{ELO23} address the question of the equivalence between \eqref{intro:static} and \eqref{controlled p action eq} in the case of general Lagrangian actions on compact manifolds, with arbitrary growth of order $p>1$ in the control variable at infinity. They also allow general control affine dynamics (the affinity only goes in the control variable while the dependence on the state variable can be nonlinear) in the control system of the form 
\begin{equation}
	\label{time-independent dynamics model 2023}
	x'(t)=F_0(x(t))+\sum_{i=1}^n F_i(x(t))\cdot \alpha_i(t),
\end{equation}
where $F_0,F_1,...,F_n$ are smooth time-independent ambient vector fields satisfying linear growth conditions at infinity. The compactness of the manifold $M$ plays a crucial role in their analysis, and the main tool used in their analysis is a relaxation technique based on Young measures (cf. \cite{Bernard07, Bernard08}).
\item Via a similar approach of relaxation and Young measures, in \cite{EL24} the author revisits the problems considered in \cite{ELO23} and makes significant improvements. In the setting of general Lagrangians and control affine dynamics as in \eqref{time-independent dynamics model 2023} the assumption on the compactness of the supports of the source and target measures is removed.
\end{itemize}

\subsection*{The novelties of our approach and description of our main results}
Instead of the nonlinear control systems \eqref{time-independent dynamics model 2023} used in \cite{ELO23} and \cite{EL24}, here we consider linear control systems of the form \eqref{eq:intro_cont}.
The Lagrangians are precisely $p$-powers of the norm of the control function, with arbitrary range of $p>1$. This form will allow us to obtain some fine properties of the cost functions $(x,y)\mapsto c_{p}(x,y)$ as well as the associated optimal control functions. It is worth noting that we allow the coefficient matrices $M$ and $N$ in the control system to be time-dependent (compared to the non-autonomous control systems in \cite{ELO23} and \cite{EL24}). Our approach in this paper relies on precise {\it functional analytic properties} of suitable operators that we describe below. In what follows we describe the guiding ideas in the proofs Theorems \ref{thm:not_empty}, \ref{thm:general_bb} and \ref{thm:optimal_control_cost}.

\medskip

As a preliminary analysis leading to the proof of our main Theorem \ref{thm:general_bb}, we spend a considerable amount of time studying the control problem \eqref{intro:c_p} in the definition of the cost function $(x,y)\mapsto c_{p}(x,y)$. In this analysis we rely of two crucial tools: the first is the {\it end-point mapping} (cf. Definition \ref{end-point mapping definition}) 
$$E_{s,t}^x:L^p(s,t)\to \mathbb{R}^d$$ 
(which associates the state of the system \eqref{eq:intro_cont} at time $t$ to a starting position $x$ at time $s$ and a given control function) and the second is the {\it state transition map} (defined in \eqref{eq:Phi})
$$
\Phi:[0,T]\times[0,T]\times\R^{d}\to\R^{d},
$$
which is the flow map associated to the homogeneous system, i.e. when $N\equiv 0$ in \eqref{eq:intro_cont}. We show that under our standing assumptions, $E_{s,t}^{x}$ is surjective (cf. Theorem \ref{thm:controllability}), weakly continuous and continuously differentiable in the Fr\'echet sense (cf. Lemma \ref{properties of endpoint mapping}). These properties readily imply that for any $x,y\in \mathbb{R}^d$ there exists a unique optimiser in \eqref{intro:c_p}, which we will denote by $\alpha^*_p(\cdot;x,y)\in L^p\pa{0,T;\R^n}$ (cf. Theorem \ref{thm:optimal_control_cost} above and Theorem \ref{p-action minimiser-theorem}). 

These will then lead to an important `comparison result' for $c_{p}(x,y)$ (see Theorem \ref{thm:optimal_control_cost} above and Corollary \ref{cor:upper_and_lower_bounds_on_c_p}), namely that there exists $K_{1},K_{2}>0$ such that for all $x,y\in\R^{d}$ we have 
$$
K_1 \abs{y-\Phi(0,T)x}^p \leq c_p\pa{x,y} \leq K_2\abs{y-\Phi(0,T)x}^p.
$$
Relying on properties of the state transition map $\Phi$, this comparison further implies that $c_{p}^{1/p}$ is globally Lipschitz continuous, and so $(x,y)\mapsto c_{p}(x,y)$ is continuous on $\R^{d}\times\R^{d}.$ This is an improvement on the corresponding results from \cite{EL24}, which showed that the cost function $(x,y)\mapsto c_{p}(x,y)$ is lower semi-continuous.

Using precise characterisations of the optimal control $\alpha^*_p:[0,T]\times\R^{d}\times\R^{d}\to\R^{n}$ via Lagrange multipliers, we show that it is continuous in all of its variables.

\medskip

Building upon these properties of $E_{s,t}^{x}$, $\alpha^{*}$, and $c_{p}$ we are equipped to investigate the optimal transport problems. The next tool in our analysis will be a superposition principle and suitably defined Borel probability measures on the space of continuous paths $C([0,T];\R^{d})$. The actual subclass of paths that we choose are precisely the ones that are {\it generated by optimal controls} $\alpha^{*}$ in \eqref{eq:intro_cont}. Approximating $\mu,\nu$ by empirical measures and utilising the controllability of the system \eqref{eq:intro_cont}, we show that we can build measures $\eta$ that are supported on $C([0,T];\R^{d})$ and connect $\mu$ to $\nu$, i.e. $(e_{0})_\sharp\eta = \mu$ and $(e_{T})_{\sharp}\eta = \nu$ (where  $e_{t}:C([0,T];\R^{d})\to\R^{d}$ stands for the evaluation map $e_{t}(\gamma) = \gamma(t)$). These measures are attained as limits of measures which are concentrated on finitely many trajectories of \eqref{eq:intro_cont}.

Having such measures $\eta$ and their suitable disintegrated families $(\eta_{t,x})_{(t,x)\in [0,T]\times\R^{d}}$ at hand, we define $(\rho_{t},u_{t})_{t\in[0,T]}$ in a natural way as
$$
\rho_{t} := (e_{t})_{\sharp}\eta,\ \ {\rm{and}}\ \ u(t,x) := \int_{e_t^{-1}\br{x}}\alpha^\ast_p \pa{t;\gamma(0),\gamma(T)}\dd\eta_{t,x}\pa{\gamma},
$$
and these will precisely be the weak solutions to the continuity equation \eqref{eq:continuity_extended_explicit}, which are candidates in the study of the dynamic problem \eqref{controlled p action eq} in the definition of $\mathcal{D}_{p}(\mu,\nu).$

\medskip

The last tool in our analysis is the very important observation that there exists a bijection between the set of measures supported on optimal trajectories and the set of transference plans $\Pi(\mu,\nu)$ (see Lemma \ref{lem:connection_between_optimisation_sets}). This is inspired by \cite[Proposition 6]{Bernard07} where a similar surjection in a different context is investigated. With this map, and all other preliminary results, the proof of the main theorem of this paper, i.e. Theorem \ref{thm:general_bb} is easily deduced. It is worth to note that a consequence of our analysis is a prescribed way to create a minimiser in one problem from a minimiser in the other. Consequently, once we have an explicit minimiser in one problem (the static one, for instance), we can generate a minimiser for the other (see Remark \ref{rem:recipe} for more details).

\medskip

To summarise this description, let us emphasise that while we obtain similar results to the ones presented in \cite{ELO23} and \cite{EL24},  the novelty of our paper lies in the methodology that we propose. By precisely analysing the variational problem associated to the cost function $c_{p}$, the optimal control function $\alpha^{*}$ and the end-point map $E_{s,t}^{x}$, then introducing measures defined on optimal trajectories, we give constructive proofs and quantitative estimates. Our hope is that our approach could lead to further new investigations also related to `kinetic type' optimal transport problems and distances, such as the ones recently studied in \cite{Iac, IacJun, BriMaaQua, Par}.

\medskip

\subsection*{Organisation of the rest of the paper} In \S\ref{sec:controllability} we will investigate the controllability question associated to the dynamics \eqref{eq:intro_cont}, and prove Theorem \ref{thm:optimal_control_cost}. We will use the tools from \S\ref{sec:controllability} to study $\cadm\pa{\mu,\nu}$ and the associated functional $(\mu,\nu)\mapsto\D_p\pa{\mu,\nu}$ in \S\ref{sec:dynamics}, culminating in the proof of the Benamou--Brenier type theorem, Theorem \ref{thm:general_bb}, in \S\ref{sec:general_bb}. We conclude the work with 
an Appendix where we show technical results that would have hindered the flow of work. 

\subsection*{Acknowledgements} YJ has been supported by the China Scholarship Council (with Scholarship Number 202308060310). ARM has been supported by the EPSRC New Investigator Award ``Mean Field Games and Master equations'' under award no. EP/X020320/1. The authors would like to thank the anonymous \amit{referees} for the helpful and constructive comments which helped improve our manuscript.

\section{Controllability of the driving ODEs and properties of the cost function $(x,y)\mapsto c_p(x,y)$}\label{sec:controllability}
Let $T>0$, $d,n\in\N$ be such that $1\leq n\leq d$, and let $M\in C([0,T];\mathbb{R}^{d\times d}),\ N\in C([0,T];\mathbb{R}^{d\times n})$. Let furthermore $x,y\in\R^d$ be given. The question of {\it controllability} of the system
\begin{equation}\label{eq:controlled}
\left\{
	\begin{split}
		&\gamma'(t)=M(t)\gamma(t)+N(t)\alpha(t),\quad 0< t<T, \\
		&\gamma(0)=x,\;\gamma(T)=y,
	\end{split}
\right.
\end{equation}
is well documented in the literature (cf. \cite{Sontag98}). Indeed, for $p\ge 1$ we have precise conditions on $M,N$ developed in the literature ensuring the existence of  $\alpha=\alpha(\cdot,x,y)\in L^p\pa{0,T;\mathbb{R}^n}$ such that we have $\gamma\in W^{1,p}\left(0,T;\R^d\right)$ which solves this system, connecting $x$ to $y$ in the prescribed way. For the convenience of the reader we discuss below this controllability question and further properties of the control system.

Throughout this paper we impose the following conditions on $M,N$.

\begin{assumption}\label{asmp}
Let $d,n\in\N$ be such that $1\leq n\leq d$. We assume the following. 
	\begin{enumerate}[(i)]
		\item\label{item:continuity_of_matrices} $M\in C([0,T];\mathbb{R}^{d\times d})$ and $N\in C([0,T];\mathbb{R}^{d\times n})$.
		\item\label{item:differentiability_of_matrices} There exists $t'\in [0,T)$ such that $M\in C^{\beta-1} \pa{t',T},N\in C^{\beta} \pa{t',T}$ for $\beta := \lfloor d/n\rfloor$. Moreover, for any $k\in \br{0,\dots, \beta-1},\widetilde{k}\in \br{0,\dots, \beta}$ 
		$$M^{(k)}\pa{T_-}=\lim_{t\to T^{-}}M^{(k)}(t)\quad \text{and} \quad N^{(\widetilde{k})}\pa{T_-}=\lim_{t\to T^{-}}N^{(\widetilde{k})}(t)\quad \text{exist and are finite}.$$
		\item\label{item:controllability} $\mathrm{rank}\pa{\mathcal{R}}=d$, where $\mathcal{R}:=\pa{\pmb{r}_0,\pmb{r}_1,...,\pmb{r}_{\beta}}\in \mathbb{R}^{d\times (\beta+1)n}$ is the $d\times \pa{\beta+1}n$ matrix which is built from the $d\times n $ block matrices
		\begin{equation}
			\label{Gamma matrix formula general M(t)}
			\pmb{r}_k:=P_k\pa{T_-}
		\end{equation}
		where $P_k$ are the family of matrix polynomial defined by the recursive condition
		\begin{equation}
			\label{formula of matrix polynomials p_i}
			P_k(s):=\begin{cases}
				N(s),\ & k=0,\\
				-M(s)P_{k-1}(s) + \frac{d}{ds}P_{k-1}(s), & 1\leq k \leq \beta.
			\end{cases}
		\end{equation}
	\end{enumerate}
\end{assumption}

\begin{remark}\label{rem:rank_condition_name}
	Based on \cite[Proposition 3.5.16]{Sontag98}, the rank condition given in \eqref{item:controllability} is known as a {\it generalised Kalman-rank condition}. 
	
\end{remark}

\begin{remark}\label{rem:simple_rank_condition}
	In the particular case where $M(s)=M$ and $N(s)=N$ for all $s\in[0,T]$, for some $M\in\R^{d\times d},N\in\R^{d\times n}$, we find that 
	\begin{equation}\nonumber
		P_k(s):=\pa{-1}^k M^kN
	\end{equation}
	and the rank condition \eqref{item:controllability} is reduced to 
	$$\mathrm{rank}\pa{N,-MN,M^2N,\dots,\pa{-1}^\beta M^\beta N}=d$$ 
    which is equivalent to the well-known classical Kalman-rank condition
	$$\mathrm{rank}\pa{N,MN,M^2N,\dots, M^\beta N}=d,$$
	since $\mathrm{rank}\pa{\pmb{r}_0,\pmb{r}_1,...,\pmb{r}_{\beta}} = d \quad \Leftrightarrow \quad \mathrm{rank}\pa{c_0\pmb{r}_0,c_1\pmb{r}_1,...,c_\beta\pmb{r}_{\beta}} = d,$
	where $c_0,\dots,c_\beta$ are non-zero real numbers. 
\end{remark}

\subsection{The controllability of the ODEs}

As Assumption \ref{asmp} is slightly weaker  than the ones imposed in \cite{Sontag98} (in the sense that we impose differentiability of order $\beta-1$  and $\beta$ on $M$ and $N$, respectively, instead of smoothness), we have elected to provide proofs to most of the stated results for the sake of completeness (though some technical proofs have been postponed to Appendix \ref{app:additional}).

For $s\in[0,T]$ fixed we will rely on the homogeneous system 
\begin{equation}
	\label{homogeneous particle system}
	\begin{cases}
		&\gamma_H'(\tau)=M(\tau)\gamma_H(\tau),\ 0\leq s\leq \tau\leq t \leq T \\
		&\gamma_H(s)=x
	\end{cases}
\end{equation}
Since $M\in C([0,T];\R^{d\times d})$, the standard Cauchy--Lipschitz theory guarantees that we can find a unique solution $\gamma_H^{s,x}\in \amit{C^1((s,t))}$ to \eqref{homogeneous particle system}. For any $s,t\in [0,T]$ we define the {\it state transition map} (a two-parameter semigroup) $\Phi(s,\tau):\mathbb{R}^d\to \mathbb{R}^d$ by 
\begin{equation}\label{eq:Phi}
\Phi(s,t)x := \gamma_H^{s,x}(t).
\end{equation}
$\Phi$ has the following properties.
\begin{lemma}\label{lem:2_parameter_semigroup} 
	Assume that $M\in C([0,T];\R^{d\times d})$ and let $\Phi\pa{\cdot,\cdot}:[0,T]\times [0,T]\times \R^d \to \R^d$ defined in \eqref{eq:Phi}. Then
	\begin{enumerate}[(i)]
		\item\label{item:semigroup_property} $\Phi(\tau,t)\circ \Phi(s,\tau)=\Phi(s,t)$ for any $s,t,\tau\geq 0$.
		\item\label{item:P(t,t)} $\Phi(t,t)=\mathbf{id}$ for any $t\geq 0$.
		\item\label{item:semigroup_is_matrix} Each $\Phi(s,t)$ is a linear operator and as such has a matrix representation. 
		\item\label{item:semigroup_bound} $\norm{\Phi(s,t)}\leq e^{M_1\abs{t-s}}$ where $M_1:=\sup_{t\in[0,T]}\norm{M(t)}$ and $\|\cdot\|$ denotes the operator norm. 
		\item\label{item:differentiability_of_semigroup} $\Phi(\cdot,\cdot)\in C^{1}\pa{\rpa{0,T}\times \rpa{0,T}}$. Identifying $\Phi$ with its matrix representation we find that
		$$\frac{\dd}{\dd t}\Phi(s,t) = M(t)\Phi(s,t),\quad\text{and}\quad \frac{\dd}{\dd s}\Phi(s,t) = -\Phi(s,t)M(s).$$
		If in addition we assume that for $t'\in [0,T)$, $M\in C^{\beta}(t',T)$ for some $\beta\in\N$, then we obtain $\Phi(\cdot,\cdot)\in C^{\beta+1}\pa{\pa{t',T}\times \pa{t',T}}$.
	\end{enumerate}
\end{lemma}
    The content of this lemma is well-known results for experts, but for the sake of completeness, we provide its proof in Appendix \ref{app:additional}. It is worth to mention that one can realise the transition map $\Phi(s,t)$ via the map $\Psi(t):=\Psi(0,t)$ which generates the fundamental solution of $x'(t)=M(t)x(t)$ with $x(0)=x$. Indeed, one finds that 
    	$$\Phi(s,t)=\Phi(0,t)\circ \Phi(s,0)=\Psi(t)\Psi(s)^{-1}.$$

 Another essential tool in our analysis is the so-called {\it end-point mapping} that we define as follows.

\begin{definition}
	\label{end-point mapping definition}
	For a given $x\in \mathbb{R}^d$ and $s,t\in [0,T]$, the \emph{end-point map} $E_{s,t}^x:L^p(s,t)\to \mathbb{R}^d$ is defined by
	\begin{equation}\nonumber
		E_{s,t}^x(\alpha):=\gamma^{s,x}_\alpha(t)
	\end{equation}
	where $\gamma^{s,x}_\alpha$ is the solution of \eqref{eq:controlled} with initial condition $\gamma^{s,x}_\alpha(s)=x$. That is to say, 
	\begin{equation}\label{eq:def_of_end_point}
		\gamma^{s,x}_{\alpha}(t)=E_{s,t}^x(\alpha):=\Phi(s,t)x+\int_s^t \Phi(\tau,t)N(\tau)\alpha(\tau)\dd\tau.
	\end{equation}
\end{definition}
\begin{remark}\label{rem:gamma_s_x_alpha_is_a_solution}
	By Lemma \ref{lem:2_parameter_semigroup}\eqref{item:differentiability_of_semigroup}, we have that $\Phi$ is continuously differentiable on $[0,T]\times [0,T]$ and $N\in C([0,T];\R^{d\times n})$ by assumption. This implies that $\Phi(\cdot,t)N(\cdot)\alpha(\cdot)\in L^1\pa{s,t}$, for all $t\in (s,T)$ and as such $\gamma^{s,x}_\alpha$ is absolutely continuous and satisfies
\begin{align*}	
	\frac{\dd}{\dd t}\gamma_{\alpha}^{s,x}(t)&=M(t)\Phi(s,t)x+ \Phi(t,t)N(t)\alpha(t)+\int_s^t \frac{\dd}{\dd t}\Phi(\tau,t)N(\tau)\alpha(\tau)\dd\tau\\
	&=M(t)\pa{\Phi(s,t)x+\int_s^t \Phi(\tau,t)N(\tau)\alpha(\tau)\dd\tau}+N(t)\alpha(t)=M(t)\gamma_{\alpha}^{s,x}(t)+N(t)\alpha(t).
\end{align*}
	In particular, we see that $\gamma_{\alpha}^{s,x} \in W^{1,p}\pa{s,t}$.
\end{remark}
Recall that our goal is to find for any $x,y\in\R^d$ a pair $(\gamma,\alpha)\in W^{1,p}(0,T)\times L^{p}((0,T);\R^{n})$ such that \eqref{eq:controlled} holds. This can be formulated using the end-point mapping.
\begin{definition}
	\label{definition of controllability particle}
	We say that the system \eqref{eq:controlled} is {\it controllable} if for any $x,y\in \mathbb{R}^d\times \mathbb{R}^d$ there exists $\alpha\in L^p((0,T);\mathbb{R}^n)$ such that $E_{0,T}^{x}(\alpha)=y$. Equivalently, the system is controllable if $E_{0,T}^{x}$ is {\it surjective} for any $x\in \mathbb{R}^d$.
\end{definition}

Our main result for this short subsection is the following.

\begin{theorem}
\label{thm:controllability}
	Under Assumption \ref{asmp} the system \eqref{eq:controlled} is controllable.
\end{theorem}

\begin{proof}
	The proof follows ideas presented in \cite{Sontag98}. As the conditions outlined in Assumption \ref{asmp} are slightly weaker than those given in \cite{Sontag98}, we give the details here.
	
	We start by noticing that since $E^x_{0,T}(\alpha) = \Phi(0,T)x + E^0_{0,T}(\alpha)$ (where we used the linearity of $\Phi$), it is enough to show that $E^0_{0,T}:L^p\pa{(0,T);\R^n}\to\R^{d}$ is surjective to conclude the controllability of the system \eqref{eq:controlled}.  
	
	Assume by contradiction that $E^0_{0,T}$ is not surjective. As $E^0_{0,T}$ is a linear operator, its image is a subspace of $\R^d$ which is of dimension less than $d$. Consequently, we can find  $0\neq b\in \mathbb{R}^d$ such that $b^\top E^0_{0,T}(\alpha)= 0$ for any $\alpha\in L^p\pa{(0,T);\R^n}$. In other words, 
	\begin{equation}
		\nonumber 
		\int_0^{T}[b^\top\Phi(\tau,T)N(s)]\cdot \alpha(\tau)\dd\tau = 0
	\end{equation}
	for any $\alpha\in L^p((0,T);\R^n)$. Plugging $\wt{\alpha}(t):=[b^\top\Phi(t,T)N(t)]^\top$ (which is continuous and as such in $L^p\pa{0,T;\R^n}$) in the above we conclude that 
	\begin{equation}\nonumber
		\int_0^T \abs{\wt{\alpha}(\tau)}^2\dd\tau=\int_0^{T}|b^\top\Phi(\tau,T)N(\tau)|^2\dd\tau = 0.
	\end{equation}
	This implies, due to the continuity of $\wt{\alpha}$, that 
	\begin{equation}
		\label{commutativity condition proposition - eq 2}
		b^\top\Phi(s,T)N(s)=0,\qquad \forall s\in [0,T].
	\end{equation}
	We will use this identity to show that
	$$b^\top\mathcal{R}=0,$$
	where $\mathcal{R}$ is defined in \eqref{item:controllability} of Assumption \ref{asmp}. This will contradict the assumption that $\mathrm{rank}\pa{\mathcal{R}}=d$.
	
	As \eqref{commutativity condition proposition - eq 2} holds for all $s\in (0,T)$ we find that
	$$b^\top\frac{\dd^k}{\dd s^k}\pa{\Phi(s,T)N(s)}=0,\qquad \forall s\in (t',T),\;0\leq k \leq \beta,$$
	where we have used the fact that $\Phi(s,T)$ and $N(s)$ are differentiable $\beta$ times on $(t',T)$.
	
	We know that for any $s\in (t',T)$
	$$\frac{\dd}{\dd s}\pa{\Phi(s,T)N(s)} = -\Phi(s,T)M(s)N(s)+\Phi(s,t)N'(s) = \Phi(s,T)P_1(s),$$
	 as $\frac{\dd}{\dd s}\Phi(s,T)=-\Phi(s,T)M(s)$.
	Assuming that 
	$$\frac{\dd^j}{\dd s^j}\pa{\Phi(s,T)N(s) } = \Phi(s,T)P_j(s)$$ 
	for some $1\leq j \leq  \beta-1$ we find that
	$$\frac{\dd^{j+1}}{\dd s^{j+1}}\pa{\Phi(s,T)N(s)} = \frac{\dd}{\dd s}\pa{\Phi(s,T)P_j(s) }= -\Phi(s,T)M(s)P_j(s) + \Phi(s,T)\frac{\dd}{\dd s}P_j(s) = \Phi(s,T)P_{j+1}(s).$$
	Consequently, we conclude that for any $0\leq k\leq \beta$
	$$b^\top \pmb{r}_k = \lim_{s\to T^-}b^\top\Phi(s,T)P_{k}(s) = \lim_{s\to T^-}b^\top\frac{\dd^k}{\dd s^k}\pa{\Phi(s,T)N(s)}=0,$$
	from which we conclude that $b^\top\mathcal{R}=0$ and with it the desired contradiction.
\end{proof}

An immediate corollary of this theorem is the first part of Theorem \ref{thm:optimal_control_cost}.
\begin{corollary}\label{cor:pcadm_is_not_empty}
	Let $M:[0,T]\to \mathbb{R}^{d\times d}$, $ N:[0,T]\to \mathbb{R}^{d\times n}$ be such that Assumption \ref{asmp} holds. Then for any $x,y\in \R^d$ we have that $\pcadm(x,y)\not=\emptyset$.
\end{corollary}

\begin{remark}
	It is worth to mention that the rank condition written in Assumption \ref{asmp}\eqref{item:controllability} is a sufficient but not necessary condition for the controllability of our system. A necessary and sufficient condition is expressed via the \textit{controllability Gramian}, $G:=\int_0^T\Phi(s,T)N(s)N(s)^\top\Phi(s,T)^\top\dd s$. We refer the reader to \cite{Sontag98} for more information on this matter.
\end{remark}

Before moving to a more detailed investigation of the controllability which will result in the second part of Theorem \ref{thm:optimal_control_cost}, we mention a simpler setting for the system \eqref{eq:controlled} in which the rank condition is more tractable.

\begin{lemma}\label{lem:simpler_rank_condition}
 	Let $M:[0,T]\to \mathbb{R}^{d\times d}$, $ N:[0,T]\to \mathbb{R}^{d\times n}$ be such that \eqref{item:continuity_of_matrices} and \eqref{item:differentiability_of_matrices} in Assumption \ref{asmp} hold. Assume in addition that for all $t_1,t_2\in [0,T]$ we have that 
 	\begin{equation}\label{eq:commutativity_of_M}
 		M\pa{t_1}M\pa{t_2} = M(t_2)M(t_1).
 	\end{equation}
 	Then in the rank condition \eqref{item:controllability} we have that  
 	$$\pmb{r}_k=
        \sum_{m=0}^k \begin{pmatrix}
 		k \\ m
 	\end{pmatrix}B_m\pa{-M(T_-),\dots,-M^{(m-1)}(T_-)}N^{(k-m)}(T_-)
     $$
 	where
 	$$B_k(x_1,\dots,x_k) =\begin{cases}
 		1,& k=0,\\
 		k! \sum_{\amit{j_1+2j_2+\dots+kj_k=k}} \prod_{i=1}^k \frac{x_i^{j_i}}{\pa{i!}^{j_i}j_i! }, & k\geq 1,
 	\end{cases} $$
 	are the $k^{th}$ complete Bell polynomials. 
\end{lemma}

\begin{proof}
	The commutation relation \eqref{eq:commutativity_of_M} implies that the solution to the homogenous system \eqref{homogeneous particle system} is given by
	$$\gamma_H(t) = e^{\int_s^t M(\theta)\dd\theta}x$$
	from which we conclude that the matrix representation of $\Phi(s,t)$ is 
	$$S(s,t):=e^{\int_s^t M(\theta)\dd\theta}.$$
	Moreover, \eqref{eq:commutativity_of_M} implies that for any $k,j\in \br{0,\dots,\beta-1}$
		\begin{equation}\nonumber
		M^{(k)}\pa{t_1}M^{(j)}\pa{t_2} = M^{(j)}(t_2)M^{(k)}(t_1),\qquad \forall t_1,t_2\in [0,T],
	\end{equation}
	which in turn implies that $M^{(k)}(s)S\pa{s,t} =S\pa{s,t} M^{(k)}(s) $ for any $t,s\in [0,T]$ and any $0\leq k \leq \beta-1$. Much like the proof of Fa\`a di Bruno's formula for $g(x)=e^{f(x)}$ we find that 
	$$ \frac{\dd^k}{\dd s^k}S(s,t)=S(s,t)B_k\pa{-M(s),\dots,-M^{(k-1)}(s)}$$
	for any $0\leq k \leq \beta$. As we saw in the proof of Theorem \ref{thm:controllability}
\begin{align*}	
	\pmb{r}_k &= \lim_{s\to T_-}\frac{d^k}{ds^k}\pa{\Phi(s,T)N(s)}= \lim_{s\to T_-}\frac{d^k}{ds^k}\pa{S(s,T)N(s)}
	=\sum_{m=0}^k \begin{pmatrix}
		k \\ m
	\end{pmatrix}\lim_{s\to T_-}\frac{\dd^m}{\dd s^m}S(s,T)N^{(k-m)}(T_-)\\
	&=\sum_{m=0}^k \begin{pmatrix}
		k \\ m
	\end{pmatrix}B_m\pa{-M(s),\dots,-M^{(m-1)}(s)}N^{(k-m)}(T_-),
\end{align*}
	which is the desired result. 
\end{proof}

\subsection{The cost function $(x,y)\mapsto c_p(x,y)$}
As was mentioned in \S\ref{sec:intro}, the controllability of the system \eqref{eq:controlled} will not only help us show that $\cadm\pa{\mu,\nu}$ is not empty, but is also vital to the emergence of our proposed cost function, $\R^{d}\times\R^{d}\ni(x,y)\mapsto c_p(x,y)$. This section is devoted to the study of this function, expressed in the second part of Theorem \ref{thm:optimal_control_cost}. As before, we will assume that the conditions in Assumption \ref{asmp} are satisfied throughout this subsection.

We start with a few simple properties of the end-point mapping.
\begin{lemma}
	\label{properties of endpoint mapping}
	Let $x\in \mathbb{R}^{d}$ and let $0\leq s<t<\infty$. Recall the definition of the end-point mapping from Definition \ref{end-point mapping definition}.
	\begin{enumerate}[(i)]
		\item\label{item:end_point_is_weakly_cont} $E_{s,t}^x$ is weakly continuous, i.e. if $\alpha_m \underset{m\to\infty}{\rightharpoonup} \alpha$ in $L^p(s,t;\mathbb{R}^n)$  then $E_{s,t}^x(\alpha_m)\underset{m\to\infty}{\longrightarrow} E_{s,t}^x(\alpha)$. 
		\item\label{item:end_point_is_diff} $E_{s,t}^x$ continuously differentiable on $L^p(s,t;\mathbb{R}^n)$. Moreover, its Fr\'echet derivative at any $\alpha\in L^p(s,t;\mathbb{R}^n)$ is given by
		\begin{equation}
			\label{Gateaux differential of E^X}
			D_\alpha E_{s,t}^x(u)=E^0_{s,t}(u)=\int_s^t \Phi(\tau,t)N(\tau)u(\tau)\dd\tau.
		\end{equation}
	\end{enumerate}
\end{lemma}
\begin{proof}
	\leavevmode
	\begin{enumerate}[(i)]
		\item For $ 1\leq k \leq n$ we denote by $\alpha_{k,m}$ and $\alpha_k$ the $k^{th}$ components of $\alpha_m$ and $\alpha$, respectively. Since $(\alpha_m)_{m\in\N}$ converges weakly to $\alpha$ in $L^p(s,t;\mathbb{R}^n)$ we have that $(\alpha_{m,k})_{m\in\N}$ converges weakly to $\alpha_k$ in $L^p\pa{s,t}$ for any $1\leq k\leq n$. Since $\Phi(\cdot,\cdot)$ and $N$ are continuous, we find that for any $1\leq k\leq n$
	\begin{align*}	
		\pa{E_{s,t}^x(\alpha_m)}_k & = \pa{\Phi(s,t)x}_k+\sum_{j=1}^d \sum_{l=1}^n\int_s^t \Phi(\tau,t)_{kj}N(\tau)_{jl}\alpha_{m,l}(\tau)\dd\tau\\ 
		&\underset{m\to\infty}{\longrightarrow }\pa{\Phi(s,t)x}_k+\sum_{j=1}^d \sum_{l=1}^n\int_s^t \Phi(\tau,t)_{kj}N(\tau)_{jl}\alpha_{l}(\tau)\dd\tau =\pa{E_{s,t}^x(\alpha)}_k,
	\end{align*}
		which shows the desired result.
		\item We notice that for any $\alpha \in L^p(s,t;\mathbb{R}^n)$ 
		$$E^{x}_{s,t}(\alpha) = \Phi(s,t)x + E^{0}_{s,t}(\alpha)$$
		and that for any scalar $a$ and $\alpha,\beta \in L^p(s,t;\mathbb{R}^n)$
		$$E^{0}_{s,t}(a\alpha+\beta) = aE^{0}_{s,t}(\alpha)+E^{0}_{s,t}(\beta).$$
		In other words, $E^{x}_{s,t}$ is an affine mapping and consequently its Fr\'echet derivative at any $\alpha$ is given by
		$$D_{\alpha}E^x_{s,t}[u] =E^{0}_{s,t}(u). $$
		Since $\Phi$ and $N$ are bounded (by Assumption \ref{asmp} and Lemma \ref{lem:2_parameter_semigroup}) we find that 
		$$\abs{E^{0}_{s,t}(u)}\leq \norm{\Phi}_{L^\infty\pa{[0,T]\times [0,T]}}\norm{N}_{L^\infty\pa{[0,T]}}\norm{u}_{L^p\pa{s,t;\R^n}}\abs{t-s}^{\frac{1}{q}},$$
		where $q=p/(p-1)$.
		This shows that $E^{0}_{s,t}$ is continuous. The proof is thus completed. 
	\end{enumerate}
\end{proof}
An immediate consequence of the above is the following:

\begin{theorem} \label{p-action minimiser-theorem}
	For any $x,y\in \mathbb{R}^d$ there exists a unique $\alpha^*_p(\cdot;x,y)\in L^p\pa{0,T;\R^n}$ such that
	\begin{equation}\label{eq:c_p_and_alpha_ast}
		c_p(x,y)=\int_0^T \abs{\alpha_p^\ast(t;x,y)}^p \dd t
	\end{equation}
\end{theorem}

\begin{proof}
	We start by mentioning that as $\pcadm\pa{x,y}$ is not empty and any $L^p$ norm is bounded from below by $0$, the function $(x,y)\mapsto c_p(x,y)$ is well defined and is finite. 
	
	By definition, we can find a minimising sequence $\pa{\gamma_m,\alpha_m}_{m\in\N}$ in $\pcadm\pa{x,y}$ such that 
	\begin{equation}\nonumber
			\norm{\alpha_m}^p_{L^p\pa{0,T;\R^n}} \underset{m\to\infty}{\longrightarrow} c_p(x,y)<\infty.
	\end{equation}
	As this implies the boundedness of $(\alpha_m)_{m\in\N}$ in $L^p\pa{0,T;\R^n}$, we conclude that there exists a subsequence of $(\alpha_m)_{m\in\N}$, $(\alpha_{m_k})_{k\in\N}$, that converges to some $\alpha^\ast\in L^p\pa{0,T;\R^n}$ weakly, as $k\to\infty$. Consequently,
	\begin{equation}\nonumber
	\norm{\alpha^*}^p_{L^p\pa{0,T;\R^n}}\leq \liminf_{k\to \infty}\norm{\alpha_{m_k}}^p_{L^p\pa{0,T;\R^n}}=c_p(x,y)
	\end{equation}
	as the norm is lower-semi continuous with respect to weak convergence.
	
	As $E_{s,t}^x$ is weakly continuous we conclude that 
	$$E_{0,T}^x\pa{\alpha^\ast} = \lim_{k\to\infty}E_{0,T}^x\pa{\alpha_{m,k}}=y,$$
	showing that $\pa{\gamma_{\alpha^\ast}^{0,x},\alpha^\ast}\in \pcadm\pa{x,y}$ where $\gamma_{\alpha^\ast}^{0,x}$ is defined via \eqref{end-point mapping definition}. Consequently
	$$c_p(x,y) \leq \norm{\alpha^*}^p_{L^p\pa{0,T;\R^n}},$$
	which shows the existence of a minimiser to the definition of $c_p(x,y)$.
	
	To show that $\alpha^\ast$ is unique we assume that there exists $\pa{\gamma_{\beta^\ast},\beta^\ast} \in \pcadm\pa{x,y}$ such that $c_p(x,y) = \norm{\alpha^*}^p_{L^p\pa{0,T;\R^n}}=\norm{\beta^*}^p_{L^p\pa{0,T;\R^n}}$. As $E_{0,T}^x$ is an affine operator we see that for any $\lambda \in [0,1]$ we have that 
	$$\pa{\lambda \gamma_{\alpha^\ast}^{0,x} + \pa{1-\lambda}\gamma_{\beta^\ast}, \lambda \alpha^\ast+\pa{1-\lambda}\beta^\ast}\in \pcadm\pa{x,y}$$
	and as such
	$$c_p\pa{x,y} \leq \norm{\lambda \alpha^\ast+\pa{1-\lambda}\beta^\ast}^p_{L^p\pa{0,T;\R^n}} \leq \pa{\lambda\norm{\alpha^*}_{L^p\pa{0,T;\R^n}} +\pa{1-\lambda}\norm{\beta^*}_{L^p\pa{0,T;\R^n}}}^p=c_p\pa{x,y},$$
	The above implies that there is equality in each inequality and in particular we have equality in our triangle inequality for any $\lambda\in [0,1]$. Choosing $\lambda=\frac{1}{2}$ we conclude that there exists $a\geq 0$ such that $\alpha^\ast=a \beta^\ast$ and since $\norm{\alpha^*}_{L^p\pa{0,T;\R^n}}=\norm{\beta^*}_{L^p\pa{0,T;\R^n}}$ we must have that $\beta^\ast=\alpha^\ast$. The proof is thus complete.
\end{proof}

Surprisingly, the information we gathered so far is enough to study the continuity of $(x,y)\mapsto c_p(x,y)$ without knowing further regularity on the unique minimise $(x,y)\mapsto\alpha_p^\ast\pa{\cdot ;x,y}$.

As a starting point we notice that $c_p\pa{x,y}^{\frac{1}{p}}$ does not represent a distance between $x$ and $y$. Indeed, we see that $c_p\pa{x,y}=0$ implies that we can find $\pa{\gamma_\ast,\alpha_p^\ast}\in \pcadm\pa{x,y}$ such that
$$0=c_p\pa{x,y} = \norm{\alpha_p^\ast}^p_{L^p\pa{0,T;\R^n}}$$ 
which implies that $\alpha_p^\ast=0$. However, that means that 
$$y=\gamma_\ast(T) = \Phi\pa{0,T}x + \int_{0}^T\Phi(\tau,T)N(\tau)\alpha_p^\ast(\tau)\dd\tau = \Phi(0,T)x.$$
Since $\Phi(0,T)x\not=x$ in general (unless $\Phi(0,T)$ is the identity matrix) 
we see that (as mentioned in Remark \ref{rem:left_hand_side_not_a_distance}) $c_p(x,x)\not=0$ -- i.e. we need to spend energy to move from a point to itself. 

This simple observation leads us to consider a modification of $c_p$ which \textit{will} act as a distance between points in $\R^d$.

\begin{lemma}\label{lem:modification_of_c_p}
	For any $x,y\in\R^d$ define
	$$d_p(x,y) := c_p\pa{x,\Phi(0,T)y}^{\frac{1}{p}},$$
	where $\Phi(\cdot,\cdot)$ is the state transition map associated to \eqref{homogeneous particle system}. Then $d_p:\R^d\times \R^d\to[0,+\infty)$ is a metric. Moreover, $d_p$ is induced by a norm which we will denote by $\norm{\cdot}_{d_p}$.
\end{lemma}

\begin{proof}
	The non-negativity of $d_p$ is clear from its definition. We have also seen that $d_p\pa{x,y}=0$ implies that 
	$$\Phi(0,T)y = \Phi(0,T)x$$
	from which we find that $y=x$. Conversely we know that \amit{the curve $\gamma(t)=\Phi(0,t)x$} satisfies $\pa{\gamma,0}\in\pcadm\pa{x,\Phi(0,T)x}$ and as such
	$$0 \leq d_p\pa{x,x} = c_p\pa{x,\Phi(0,T)x}^{\frac{1}{p}} \leq \norm{0}_{L^p\pa{0,T;\R^n}}=0$$
	showing that $d_p(x,x)=0$.\\
	Next we show the symmetry of $d_p$. We start by noticing that if $\pa{\gamma,\alpha}\in \pcadm\pa{x,\Phi(0,T)y}$ then
	$$\Phi(0,T)y=\Phi(0,T)x + \int_{0}^T\Phi(\tau,T)N\pa{\tau}\alpha\pa{\tau}\dd\tau$$
	which implies that 
	$$\Phi(0,T)x=\Phi(0,T)y + \int_{0}^T\Phi(\tau,T)N\pa{\tau}\pa{-\alpha\pa{\tau}}\dd\tau,$$
	i.e. $\pa{\gamma_{-\alpha}^{0,y},-\alpha}\in \pcadm\pa{y,\Phi(0,T)x}$, where $\gamma_{\alpha}^{0,x}$ is defined via \eqref{end-point mapping definition}. Consequently
	$$c_p\pa{x,\Phi(0,T)y} = \norm{\alpha_p^\ast\pa{\cdot;x,\Phi(0,T)y}}^p_{L^p\pa{0,T;\R^n}}=\norm{-\alpha_p^\ast\pa{\cdot;x,\Phi(0,T)y}}^p_{L^p\pa{0,T;\R^n}} \geq c_p\pa{y,\Phi(0,T)x}.$$
	As $x$ and $y$ were arbitrary we can interchange them to conclude that 
	$$d_p(x,y) = c_p\pa{x,\Phi(0,T)y}^{\frac{1}{p}}=c_p\pa{y,\Phi(0,T)x}^{\frac{1}{p}}=d_p\pa{y,x}.$$
	To conclude the fact that $d_p$ is a metric, we will now show that it  satisfies the triangle inequality. Let $x,y,z\in\R^{d}$ be given. We have that for any $\pa{\gamma_1,\alpha_1}\in\pcadm\pa{x,\Phi(0,T)z}$ and $\pa{\gamma_1,\alpha_2}\in \pcadm\pa{y,\Phi(0,T)z}$
	$$\Phi(0,T)z =\gamma_1(T)= \Phi(0,T)x + \int_{0}^T \Phi(\tau,T)N(\tau)\alpha_1(\tau)\dd\tau,$$
	and
	$$\Phi(0,T)z =\gamma_2(T) =\Phi(0,T)y + \int_{0}^T \Phi(\tau,T)N(\tau)\alpha_2(\tau)\dd\tau,$$
	which implies that 
	$$\Phi(0,T)y = \Phi(0,T)x + \int_{0}^T \Phi(\tau,T)N(\tau)\pa{\alpha_1(\tau)-\alpha_2(\tau)}\dd\tau.$$
	In other words, $\pa{\gamma_{\alpha_1-\alpha_2}^{0,x},\alpha_1-\alpha_2}\in \pcadm\pa{x,\Phi(0,T)y}$. We conclude that for any such $\pa{\gamma_1,\alpha_1}$ and $\pa{\gamma_2,\alpha_2}$
	$$d_p\pa{x,y} = c_p\pa{x,\Phi(0,T)y}^{\frac{1}{p}} \leq \norm{\alpha_1-\alpha_2}_{L^p\pa{0,T;\R^n}} \leq \norm{\alpha_1}_{L^p\pa{0,T;\R^n}}+\norm{\alpha_2}_{L^p\pa{0,T;\R^n}}.$$
	Taking the infimum over the right hand side gives us
	$$d_p\pa{x,y} \leq c_p\pa{x,\Phi(0,T)z}^{\frac{1}{p}}+c_p\pa{y,\Phi(0,T)z}^{\frac{1}{p}}=d_p\pa{x,z}+d_p\pa{y,z},$$
	which is the desired inequality.
	
	To show the second part of the lemma, i.e. that $d_p$ is induced by a norm, we only need to show that 
	$$d_p\pa{a x,ay} = \abs{a}d_p\pa{x,y}$$
	for any scalar $a\not=0$ and any $x,y\in\R^d$, and that 
	$$d_p\pa{x+z,y+z} = d_p\pa{x,y}$$
	for any $x,y,z\in\R^d$.	
	
	We notice that it is enough to show that for any scalar $a\not=0$ and any $x,y\in\R^d$ 
	\begin{equation}\label{eq:scaling_homo}
		d_p\pa{ax,ay}^p \leq \abs{a}^p d_p\pa{x,y}^p,
	\end{equation}
	to show the scaling homogeneity.
	Indeed, if \eqref{eq:scaling_homo} holds then since  $z=\frac{1}{a}\pa{az}$ for every $a\not=0$ and $z\in \R^d$ we find that 
	$$d\pa{x,y}^p \leq \frac{1}{\abs{a}^p}d_p\pa{ax,ay}^p,$$
	which together with \eqref{eq:scaling_homo} gives us the desired identity. To show \eqref{eq:scaling_homo} we notice that if $\pa{\gamma,\alpha}\in \pcadm\pa{x,\Phi(0,T)y}$ then $\pa{a\gamma,a\alpha}\in \pcadm\pa{ax,\Phi(0,T)\pa{ay}}$, where we have used the fact that $\Phi$ is a linear map. Consequently
	$$d_{p}\pa{ax,ay}^p =c_p\pa{ax,\Phi(0,T)\pa{ay}}\leq \inf_{\pa{\gamma,\alpha}\in \pcadm\pa{x,\Phi(0,T)y}}\norm{a\alpha}^p_{L^p\pa{0,T;\R^n}}=\abs{a}^pc_p\pa{x,\Phi(0,T)y} =\abs{a}^pd_p\pa{x,y}^p. $$
	
	We are only left with showing the invariance of $d_p$ under transition. Much like with the scaling homogeneity, it would be enough to show that for any $x,y,z\in\R^d$
	\begin{equation}\label{eq:transition_invariance}
		d_p\pa{x+z,y+z} \leq d_p\pa{x,y}.
	\end{equation}
	Given $x,y\in\R^d$ we consider the pair $\pa{\gamma^\ast,\alpha^\ast}\in \pcadm\pa{x,\Phi(0,T)y}$ with
	$$c_p\pa{x,\Phi(0,T)y} = \norm{\alpha^\ast}_{L^p\pa{0,R;\R^n}}.$$
	We see that for any $z\in \R^d$ the curve 
	$$\delta(t):=\gamma^\ast(t)+\Phi(0,t)z = \Phi(0,t)z + \Phi(0,t)x+\int_{0}^t \Phi(\tau,T)N(\tau)\alpha^{*}(\tau)\dd\tau $$
	satisfies
	$$\delta(0) = x+z,\quad \delta(T)=\Phi(0,T)z + \Phi(0,T)y=\Phi(0,T)\pa{z+y},$$
	which implies that $\pa{\delta,\alpha^{*}}\in \pcadm\pa{x+z,\Phi(0,T)\pa{y+z}}$. Consequently,
	$$d_p\pa{x+z,y+z} \leq \norm{\alpha^\ast}_{L^p\pa{0,T;\R^n}} =c_p\pa{x,\Phi(0,T)y}=d_p\pa{x,y},$$
	which concludes our proof.
\end{proof}

An immediate corollary of Lemma \ref{lem:modification_of_c_p} is the following:

\begin{corollary}\label{cor:upper_and_lower_bounds_on_c_p}
	\leavevmode
	\begin{enumerate}[(i)]
		\item\label{item:upper_bound_on_c_p} There exists $K_1,K_2>0$ such that for any $x,y\in\R^d$
		\begin{equation}\label{eq:upper_and_lower_bounds_on_c_p}
			K_1 \abs{y-\Phi(0,T)x}^p \leq c_p\pa{x,y} \leq K_2\abs{y-\Phi(0,T)x}^p.
		\end{equation}
		\item\label{item:c_p_1_over_p_is_lip} $c_p^{\frac{1}{p}}$ is Lipschitz continuous on $\R^d\times\R^d$ and consequently $c_p$ is continuous.
	\end{enumerate}
\end{corollary}

\begin{proof}
	As all norms in finite dimension are equivalent, we can find $M_1,M_2>0$ such that for all $x,y\in\R^{d}$
	$$M_1\abs{x-y} \leq \norm{x-y}_{d_p}\leq M_2\abs{x-y}.$$
	Since $\Phi(0,T)\Phi(T,0)=\textbf{id}$ we have that 
	$$c_p\pa{x,y}^{\frac{1}{p}} = d_p\pa{x,\Phi(T,0)y} = \norm{x-\Phi(T,0)y}_{d_p}$$
	from which we get
	$$M_1^p \abs{x-\Phi(T,0)y}^p \leq c_p\pa{x,y} \leq M_2^p\abs{x-\Phi(T,0)y}^p.$$
	As
	$$\abs{x-\Phi(T,0)y} = \abs{\Phi(T,0)\pa{y-\Phi(0,T)x}}$$
	and as $\norm{x}_{\Phi(T,0)} := \abs{\Phi(T,0)x}$ is a norm due to the linearity and invertibility of $\Phi(T,0)$, we can find $N_1,N_2>0$ such that for all $x\in\R^{d}$
	$$N_1 \abs{x} \leq \norm{x}_{\Phi(T,0)}\leq N_2\abs{x}$$
	and conclude that 
		$$\pa{M_1N_1}^p \abs{y-\Phi(0,T)x}^p \leq c_p\pa{x,y} \leq \pa{M_2N_2}^p\abs{y-\Phi(0,T)x}^p,$$
	showing \eqref{eq:upper_and_lower_bounds_on_c_p}.	
	
	The fact that $c_p^{\frac{1}{p}}$ is Lipschitz continuous follows almost immediately from the fact that $d_p$ is a metric induced by a norm. Indeed, as was noted,
\begin{align*}	
	\abs{c_p(x,y)^{\frac{1}{p}}- c_p(z,w)^{\frac{1}{p}}} &= \abs{\norm{x-\Phi(T,0)y}_{d_p}-\norm{z-\Phi(T,0)w}_{d_p}}\\
	&\leq \norm{\pa{x-z}-\Phi(T,0)\pa{y-w}}_{d_p} \leq M_2\abs{\pa{x-z}+\Phi(T,0)\pa{y-w}}\\
	&\leq M_2\pa{1+\norm{\Phi(T,0)}}\pa{\abs{x-z}+\abs{y-w}}\leq \sqrt{2}M_2\pa{1+\norm{\Phi(T,0)}}\abs{\pa{x,y}-\pa{z,w}},
\end{align*}
	where $\norm{\Phi(T,0)}$ is the operator norm of $\Phi(T,0)$. 
	
	Lastly, the continuity of $c_p$ follows from the Lipschitz property of $c_p^{\frac{1}{p}}$, the non-negativity of $c_p$, and the fact that
	$$c_p(x,y) = \pa{c_p\pa{x,y}^{\frac{1}{p}}}^p.$$
	The proof is thus complete. 
 \end{proof}

\begin{remark}
	One can take a different approach to the study of the cost function $c_p$. For any $\pa{\gamma,\alpha}\in \pcadm\pa{x,y}$ one can consider the curve
	$$\widetilde{\gamma}(t) = \Phi\pa{t,0}\gamma(t),$$
	and using part \eqref{item:differentiability_of_semigroup} of Lemma \ref{lem:2_parameter_semigroup} we find that
	$$\widetilde{\gamma}'(t) = -\Phi(t,0)M(t)\gamma(t) + \Phi\pa{t,0}\pa{M(t)\gamma(t) +N(t)\alpha(t)}=\Phi\pa{t,0}N(t)\alpha(t):=\widetilde{N}(t)\alpha(t).$$
	and that $\widetilde{\gamma}(0)=x$ and $\widetilde{\gamma}(T)=\Phi\pa{T,0}y$. Defining the function 
	$$\widetilde{c}_p(x,y):=\min\{\norm{\alpha}_p^p\;:\;\widetilde{\gamma}'(t)=\widetilde{N}(t)\alpha(t),\; \widetilde{\gamma}(0)=x,\;\widetilde{\gamma}(T)=y\}$$
	one would find it simpler to show that $\widetilde{c}^{\frac{1}{p}}_p\pa{x,y}$ gives a metric. The fact that
	$$\widetilde{c}_p\pa{x,y} = c_p\pa{x,\Phi(0,T)y}$$
	will reproduce Lemma \ref{lem:modification_of_c_p} and Corollary \ref{cor:upper_and_lower_bounds_on_c_p} immediately. 
\end{remark}
The last ingredient in the proof of Theorem \ref{thm:optimal_control_cost}, and a study that will serve us well later, is the study of the regularity of the optimal control associated to $c_p(x,y)$, $\alpha_p^\ast\pa{\cdot;x,y}$. 

\subsection{The optimal control in the definition of $c_{p}(x,y)$} The study of the optimal control $(x,y)\mapsto\alpha_p^\ast\pa{\cdot;x,y}$ is much more nuanced and delicate than its cost function counterpart. As $\alpha_p^\ast\pa{\cdot;x,y}$ is attained as a minimum of a real valued functional over a certain set, which can be represented as a level set of a continuously differentiable function, an appropriate way to approach the study of $\alpha_p^\ast\pa{\cdot;x,y}$ is via Lagrange multiplers. We will use the following theorem, whose proof can be found in Appendix B of \cite{Rifford14}:
\begin{theorem}
	\label{lagrange multiplier theorem rifford}
	Let $(X,\norm{\cdot}_X)$ be a normed vector space, $U$ be an open subset of $X$, and $E:U\to \mathbb{R}^d$ and $J:U\to \mathbb{R}$ be two continuously differentiable mappings on $U$. Assume in addition that there exists some $u^*\in U$ that satisfies $J(u^*)\leq J(u)$ for all $u\in U$ such that $E(u^*)=E(u)$. 
	Then, there exist $\lambda\in \mathbb{R}$ and $p\in \mathbb{R}^d$ with $(\lambda,p)\neq 0$ such that $p^\top D_{u^*}E(v)=\lambda D_{u^*}J(v)$ for any $v\in U$.
\end{theorem}

To utilise this theorem we notice that in our setting $X=U=L^p\pa{0,T;\R^n}$, $E=E_{0,T}^x$ and $J = \norm{\cdot}_{L^p\pa{0,T;\R^n}}^p$. While we have shown the continuous differentiability of $E$ in Lemma \ref{properties of endpoint mapping}, we still need to investigate the differentiability of $J$. This is expressed in the next lemma.

\begin{lemma}
	\label{frechet differentiability of J lemma}
	Let $p>1$. Then the functional $J:L^p\pa{0,T;\R^n}\to \mathbb{R}$ defined by 
	\begin{equation}\label{eq:def_of_J}
		J(\alpha):=\int_0^T |\alpha(t)|^p\dd t
	\end{equation}
	 is continuously differentiable. Moreover, its Fr\'echet derivative at $\alpha\in L^p\pa{0,T;\R^n}$ acting on $u\in L^p\pa{0,T;\R^n}$ is given by 
	\begin{equation}
		\label{frechet derivative of J at alpha-equation}
		D_\alpha J(u)=p \int_0^T \mathfrak{j}_p(\alpha(t))^\top u(t)\dd t
	\end{equation}
	where $\mathfrak{j}_p:L^p\pa{0,T;\R^n}\to L^q\pa{0,T;\R^n}$ is defined to be
	\begin{equation}\label{eq:def_of_j_p_in_L_p}
		\mathfrak{j}_p(\alpha):=
		\begin{cases}
			\abs{\alpha}^{p-2}\alpha,& \alpha\neq 0,\\
			0,& \alpha=0.
		\end{cases}
	\end{equation}
\end{lemma}

This result would be well known for experts. For the sake of completion we have decided to include it in Appendix \ref{app:additional}.

\begin{remark}
	It worth to mention that in \cite[Subsection 8.3]{AGS08} the authors have shown that $\jj_p$ is in the subdifferential of the functional in question, while the above lemma has a slightly stronger conclusion, as here we characterise continuous Fr\'echet differentiability.
\end{remark}

With $E^{x}_{0,T}$ and $\norm{\cdot}^p_{L^p\pa{0,T;\R^n}}$ fully explored we are now ready to turn our attention to $\alpha^\ast\pa{\cdot;x,y}$ with the help of the Lagrange multiplier theorem, Theorem \ref{lagrange multiplier theorem rifford}:

\begin{theorem}
	\label{thm:formula of alpha^*}
	For $x,y\in \mathbb{R}^d$, let $\alpha^*_p(\cdot;x,y)\in L^{p}(0,T;\R^{n})$ be the optimal control obtained in Theorem \ref{p-action minimiser-theorem}. Then there exists a unique $\xi_p(x,y)\in \mathbb{R}^d$ such that 
	\begin{equation}
		\label{alpha^*_p formula implicit}
		\mathfrak{j}_p\big(\alpha^*_p(t;x,y)\big)=\frac{1}{p}N(t)^\top\Phi(t,T)^\top  \xi_p(x,y),
	\end{equation}
	or equivalently,
		\begin{equation}
		\label{alpha^*_p formula implicit v2}
		\alpha^*_p(t;x,y)=\frac{1}{p^{q-1}}\jj_q\pa{N(t)^\top\Phi(t,T)^\top  \xi_p(x,y)},
	\end{equation}
	a.e. with respect to the Lebesgue measure on $[0,T]$, where $\jj_p$ is defined in \amit{Lemma} \ref{frechet differentiability of J lemma}. Moreover, the following are equivalent
	\begin{enumerate}[(i)]
		\item\label{item:xi_is_zero} $\xi_p(x,y)=0$.
		\item\label{item:alpha_ast_is_zero} $\alpha^*_p(\cdot;x,y)=0$.
		\item\label{item:y_is_phi(0,T)x} $y=\Phi(0,T)x$.
	\end{enumerate}
\end{theorem}
\begin{proof}
	For a given $x,y\in \R^d$ using Lemma \ref{properties of endpoint mapping}, Theorem \ref{frechet differentiability of J lemma}, Theorem \ref{lagrange multiplier theorem rifford} with $X=U=L^p\pa{0,T;\R^n}$, $E=E_{0,T}^x$ and $J = \norm{\cdot}_{L^p\pa{0,T;\R^n}}^p$, and the fact that 
	$$J\pa{\alpha_p^\ast} \leq J\pa{\alpha},\qquad \forall \alpha\in L^p\pa{0,T; \R^n}\text{ with }E^x_{0,T}\pa{\alpha_p^\ast}=E^x_{0,T}\pa{\alpha},$$
	we find that there exist $\lambda\in \mathbb{R}$ and $\wt{\xi}\pa{x,y}\in \mathbb{R}^d$ with $(\lambda,\wt{\xi}(x,y))\neq 0$ such that 
	\begin{equation}\label{eq:xi_pre_condition}
		\wt{\xi}(x,y)^\top E^0_{0,T}(u)=p\lambda \int_{0}^T \jj_p\pa{\alpha_p^\ast(t;x,y)}^\top u(t)\dd t,\quad \forall u\in L^p\pa{0,T;\R^n}.
	\end{equation}
	We claim that $\lambda$ given in \eqref{eq:xi_pre_condition} cannot be zero. Indeed, had $\lambda=0$ then we would have concluded that 
	$$\wt{\xi}(x,y)^\top E^0_{0,T}(u)=0,\quad \forall u\in L^p\pa{0,T;\R^n}.$$
	We have seen in the proof of Theorem \ref{thm:controllability} that under Assumption \ref{asmp} $E^0_{0,T}:L^p\pa{0,T;\R^n}\to\R^d$ is surjective and as such we must have that $\wt{\xi}(x,y)=0$. This contradicts the fact that $(\lambda,\wt{\xi}(x,y))\neq 0$.
	
	Denoting by $\xi(x,y):=\frac{\wt{\xi}(x,y)}{\lambda}$
	we rewrite \eqref{eq:xi_pre_condition} as 
		\begin{equation}\nonumber
		\frac{1}{p}\int_{0}^T \xi(x,y)^\top \Phi(\tau,T)N(\tau)u(\tau)\dd\tau=\int_{0}^T \jj_p\pa{\alpha_p^\ast(\tau;x,y)}^\top u(\tau)\dd\tau,\quad \forall u\in L^p\pa{0,T;\R^n},
	\end{equation}
	from which, by the arbitrariness of $u$ we conclude the implicit expression for $\alpha_p^\ast\pa{\cdot;x,y}$, \eqref{alpha^*_p formula implicit} (since we have that $\xi(x,y)^\top \Phi(\cdot,T)N(\cdot)$ and $\jj_p\pa{\alpha^\ast\pa{\cdot;x,y}}$ are both in $ L^q\pa{0,T;\R^n}$).
	
	The uniqueness of $\xi(x,y)$ also follows from the surjectivity of $E^0_{0,T}$. Indeed, if $\xi_1(x,y)$ and $\xi_2(x,y)$ satisfy \eqref{alpha^*_p formula implicit} then 
	$$\pa{\xi_1(x,y)-\xi_2(x,y)}^\top E^0_{0,T}(u) = D_{\alpha_p^\ast}J(u)-D_{\alpha_p^\ast}J(u)=0,\quad \forall u\in L^p\pa{0,T;\R^n},$$
	which will imply that $\xi_1(x,y)=\xi_2(x,y)$.
   
   To show that \eqref{alpha^*_p formula implicit} and \eqref{alpha^*_p formula implicit v2} are equivalent we notice that for any H\"older conjugates $p,q>1$ we have that 
   $$\jj_q\pa{\jj_q(x)}= \begin{cases}
   	\abs{\jj_p(x)}^{q-2}\jj_p(x), & \jj_p(x)\not=0,\\
   	0,& \jj_p(x)=0,
   \end{cases}=\begin{cases}
   \abs{x}^{\pa{p-1}\pa{q-2}}\abs{x}^{p-2}x, & x\not=0,\\
   0,&x=0,
   \end{cases}=x,$$
   where we have used that facts that $\jj_p(x)=0$ if and only if $x=0$, $\abs{\jj_p(x)}=\abs{x}^{p-1}$, and that $$\pa{p-1}\pa{q-2}+\pa{p-2}=0.$$
   Consequently if \eqref{alpha^*_p formula implicit} holds then 
   $$	\alpha^*_p(t;x,y)=\jj_q\pa{\frac{1}{p}N(t)^\top\Phi(t,T)^\top  \xi(x,y)} =\frac{1}{p^{q-1}}\jj_q\pa{N(t)^\top\Phi(t,T)^\top  \xi(x,y)} ,$$
   which gives us \eqref{alpha^*_p formula implicit v2} and the converse holds by applying $\jj_p$ to \eqref{alpha^*_p formula implicit v2}. 
   
   Lastly, we will consider the equivalence of \eqref{item:xi_is_zero}-\eqref{item:y_is_phi(0,T)x}. 
   
   Since, by definition,
   $$y = E^x_{0,T}\pa{\alpha_p^\ast\pa{\cdot;x,y}} = \Phi\pa{0,T}x + \int_{0}^T \Phi\pa{\tau,T}N(\tau)\alpha_p^\ast\pa{\tau;x,y}\dd\tau$$
   we have that \eqref{item:alpha_ast_is_zero} implies \eqref{item:y_is_phi(0,T)x}. Conversely, if $y=\Phi\pa{0,T}x$ we see that $\pa{\Phi\pa{0,\cdot}x,0}\in \pcadm\pa{x,y}$ and as such $c_p\pa{x,y}=0 = \norm{\alpha_p^\ast\pa{\cdot;x,y}}_{L^p\pa{0,T;\R^n}}^p$, showing that $\alpha_p^\ast\pa{\cdot;x,y}=0$. 
   
   To conclude the proof we will show that \eqref{item:xi_is_zero} is equivalent to \eqref{item:alpha_ast_is_zero}. 
   
   Using \eqref{alpha^*_p formula implicit v2} we see that if $\xi_p\pa{x,y}=0$ then $\alpha_p^\ast\pa{\cdot;x,y}=0$ as $\jj_q(0)=0$.
   
   Conversely, let us assume that $\alpha_p^\ast\pa{\cdot;x,y}=0$. Using \eqref{alpha^*_p formula implicit} we see that $N(t)^\top\Phi(t,T)^\top  \xi_p(x,y)=0 $ for a.e. in $t$, and in fact since $\Phi$ and $N$ are continuous
   $$\xi_p(x,y)^\top \Phi(t,T)N(t)=0,\qquad \forall t\in [0,T].$$
  The above is nothing but equation \eqref{commutativity condition proposition - eq 2} in the proof of the controllability of the system \eqref{eq:controlled}, Theorem \ref{thm:controllability}, and as such we can copy the same proof to conclude that due to the rank condition given in Assumption \ref{asmp}, we must have that $\xi_p(x,y)=0$. The proof is now complete.
\end{proof}

\begin{remark}\label{rem:special_case_p=2}
	Theorem \ref{thm:formula of alpha^*} is particularly revealing in the special case where $p=2$. In that case, since $\jj_2(x)=x$, we find that 
	\begin{equation}\label{eq:special_form_p=2}
		\alpha^*_2(t;x,y)=\frac{1}{2}N(t)^\top \Phi(t,T)^\top  \xi_2(x,y)
	\end{equation}
	giving us an explicit connection between the minimiser $\alpha_2^\ast\pa{\cdot;x,y}$ and the Lagrange multiplier $\xi_2(x,y)$. As
	$$y=\Phi(0,T)x + \int_{0}^T \Phi(\tau,T)N(\tau)\alpha_2^\ast\pa{\tau;x,y}\dd\tau$$
	\eqref{eq:special_form_p=2} we find that 
	$$y-\Phi(0,T)x =  \frac{1}{2}\pa{\int_{0}^T \Phi(\tau,T)N(\tau)N(\tau)^\top\Phi(\tau,T)^\top \dd\tau}\xi_2(x,y).$$
	The $d\times d$ matrix $\bm{M} := \int_{0}^T \Phi(\tau,T)N(\tau)N(\tau)^\top\Phi(\tau,T)^\top \dd\tau$ is a symmetric matrix. Moreover, we see that for $v\in \R^d$
	$$v^\top \bm{M}v=0\qquad \Leftrightarrow \qquad \int_{0}^T \abs{v^\top \Phi(\tau,T)N(\tau) }^2 \dd\tau=0.$$
	The continuity of $\Phi$ and $N$ imply that the right hand side of the above is equivalent to 
	$$ v^\top \Phi(\tau,T)N(\tau) =0\qquad \forall t\in [0,T].$$
	This condition, which we encountered twice before -- predominantly in the proof of Theorem \ref{thm:controllability}, implies that under the Assumption \ref{asmp} we must have that $v=0$. Consequently, we conclude that $\bm{M}$ is invertible and
	$$\xi_2(x,y) = 2\bm{M}^{-1}\pa{y-\Phi(0,T)x}.$$
	Using \eqref{eq:special_form_p=2} we conclude that 
		\begin{equation}\label{eq:alpha_ast_2_explicit}
		\alpha^*_2(t;x,y)=N(t)^\top \Phi(t,T)^\top \bm{M}^{-1}\pa{y-\Phi(0,T)x}.
	\end{equation}
	Not only does \eqref{eq:alpha_ast_2_explicit} provides us with an explicit expression for $\alpha_2^\ast$ -- it shows its exact regularity in both $t$ and $\pa{x,y}$. Moreover, \eqref{eq:alpha_ast_2_explicit} gives us an \textit{explicit upper bound} for $c_p\pa{x,y}$ in terms of $\abs{y-\Phi(0,T)x}^p$ (in contrast to the less explicit one given in Corollary \ref{cor:upper_and_lower_bounds_on_c_p}). Indeed, $\alpha_2^\ast$ is clearly continuous in all its variables and as such $\alpha_2^\ast\pa{\cdot;x,y}\in L^p\pa{0,T;\R^n}$. As $\pa{\gamma^{0,x}_{\alpha_2^\ast\pa{\cdot;x.y}},\alpha_2^\ast\pa{\cdot;x,y}}\in \pcadm(x,y)$, where $\gamma_{\alpha_2^\ast\pa{\cdot;x,y}}^{0,x}$ is defined via \eqref{end-point mapping definition}, we conclude that
	\begin{equation}\label{eq:explicit_upper_bound_for_c_p}
		\begin{split}
			c_p\pa{x,y} \leq&\norm{\alpha_2^\ast\pa{\cdot;x,y}}^p_{L^p\pa{0,T;\R^n}} \\
			&\leq  T\norm{\Phi}^p_{L^\infty\pa{[0,T]\times [0,T]}}\norm{N}^p_{L^\infty\pa{[0,T]}}\norm{\bm{M}^{-1}}^p\abs{y-\Phi(0,T)x}^p=:C_p\abs{y-\Phi(0,T)x}^p,
		\end{split}
	\end{equation}
	where $\norm{\bm{M}^{-1}}$ is the operator norm of $\bm{M}^{-1}$. 
\end{remark}

\begin{remark}
The discussion in the previous remark shows the main difference between the cases $p=2$ (provided in \cite{CGP17}) and $p\neq 2$. The case $p=2$ gives explicit formulas, which is not the case for $p\neq 2$.

Theorem \ref{thm:formula of alpha^*} does not give us an explicit formula for $\alpha_p^\ast\pa{\cdot;x,y}$ with which we can show its continuity, at least when $p\not=2$. However, it does show that the minimiser for our cost function \textit{separates the time and space variables} and that, due to the fact that we are under Assumption \ref{asmp}, the continuity of $\alpha_p^\ast$ in all its variables is equivalent to the continuity of the Lagrange multiplier $\xi_p(x,y)  $ in the spatial variables.
\end{remark}

The last ingredient to fully prove Theorem \ref{thm:optimal_control_cost} is the following one, which is of interest in its own right.

\begin{theorem}\label{thm:continuity_of_xi}
	The Lagrange multiplier function, $\xi_p:\R^d\times \R^d\to \R^d$, defined in Theorem \ref{thm:formula of alpha^*} is continuous. Consequently, the minimiser $\alpha_p^\ast$ is continuous in all its variables.
\end{theorem}

\begin{proof}
	  
	The fact that $\alpha^\ast_p$ is continuous in all its variables when $\xi_p$ is follows from \eqref{alpha^*_p formula implicit v2}, and the continuity of $\Phi$, $N$, and $\jj_p$ for $p>1$. 
	
	We turn our attention, thus, to the continuity of $\xi_p(x,y)$. Since $\R^{2d}$ is a metric space, to show the continuity of $\xi_p$ it is enough to show that if $\pa{x_m,y_m}\underset{m\to\infty}{\longrightarrow}\pa{x,y}$, then for any subsequence of $\pa{x_m,y_m}_{m\in\N}$, $\pa{x_{m_k},y_{m_k}}_{k\in\N}$, there exists a subsequence $\pa{x_{m_{k_j}},y_{m_{k_j}}}_{j\in\N}$ such that 
	$$\xi_p\pa{x_{m_{k_j}},y_{m_{k_j}}}\underset{j\to\infty}{\longrightarrow}\xi_p(x,y).$$
	We start by showing that $\xi_p\pa{x,y}$ is controlled by $c_p(x,y)^{\frac{1}{q}}$. Recall that we have shown in the proof of Theorem \ref{thm:formula of alpha^*} that $\xi_p(x,y)$ is the unique vector in $\R^d$ such that 
	$$\xi_p(x,y)^\top E^0_{0,T}(u) = \int_{0}^T \jj_p\pa{\alpha_p^\ast(t;x,y)}^{\top}u(t)\dd t\quad \forall u\in L^p\pa{0,T;\R^n},$$
	and consequently
\begin{align*}	
	\abs{\xi(x,y)^\top E^0_{0,T}(u)} &\leq \norm{\jj_p\pa{\alpha_p^\ast\pa{\cdot;x,y}}}_{L^q\pa{0,T;\R^n}}\norm{u}_{L^p\pa{0,T;\R^n}}\\
	&=\norm{\alpha_p^\ast\pa{\cdot;x,y}}_{L^p\pa{0,T;\R^n}}^{\frac{p}{q}}\norm{u}_{L^p\pa{0,T;\R^n}}=c_p\pa{x,y}^{\frac{1}{q}}\norm{u}_{L^p\pa{0,T;\R^n}},
\end{align*}
	where we have used the fact that $\abs{\jj_p\pa{\alpha}}=\abs{\alpha}^{p-1}$ and the definition of $c_p\pa{x,y}$. By Theorem \ref{thm:controllability} we know that under Assumption \ref{asmp} $E^0_{0,T}$ is surjective and as such for every $i\in \br{1,\dots, d}$ we can find $u_i\in L^p\pa{0,T;\R^n}$ such that $E^0_{0,T}(u_i)=e_i$, where $\br{e_i}_{i=1,\dots,d}$ is the standard basis for $\R^d$. We conclude that 
	\begin{equation}\label{eq:upper_bound_on_xi_with_c_p}
		\abs{\xi_p(x,y)} = \sqrt{\sum_{i=1}^d(\xi(x,y)^\top e_i)^2} \leq c_p\pa{x,y}^{\frac{1}{q}}\sqrt{\sum_{i=1}^d\norm{u_i}^2_{L^p\pa{0,T;\R^n}}}:=C_{E}c_p\pa{x,y}^{\frac{1}{q}}.
	\end{equation}
	From Corollary \ref{cor:upper_and_lower_bounds_on_c_p} we know that $c_p$ is continuous and as such so is $c_p^{\frac{1}{q}}$. Since $\pa{x_{m_{k}},y_{m_{k}}}_{k\in\N}$ converges to $\pa{x,y}$ we have that $\left(c_p\pa{x_{m_{k}},y_{m_k}}^{\frac{1}{q}}\right)_{k\in\N}$ converges to $c_p\pa{x,y}^{\frac{1}{q}}$ and in particular
	$$\sup_{k\in\N}\abs{\xi_p\pa{x_{m_k},y_{m_k}}} \leq C_E\sup_{k\in\N}c_p\pa{x_{m_{k}},y_{m_k}}^{\frac{1}{q}}<\infty.$$
	Using the Heine--Borel theorem we extract a subsequence of $\pa{x_{m_{k}},y_{m_{k}}}_{k\in\N}$, $\pa{x_{m_{k_j}},y_{m_{k_j}}}_{j\in\N}$, such that 
	$$\xi_p\pa{x_{m_{k_j}},y_{m_{k_j}}}\underset{j\to\infty}{\longrightarrow }\xi,$$
	for some $\xi\in \R^d$. If we show that $\xi=\xi_p(x,y)$ we will conclude the proof. 
	
	Using \eqref{alpha^*_p formula implicit v2} we find that 
	$$\alpha_p^\ast(t,x_{m_{k_j}},y_{m_{k_j}}) =\frac{1}{p^{q-1}}\jj_q\pa{N(t)^\top\Phi(t,T)^\top  \xi_p\pa{x_{m_{k_j}},y_{m_{k_j}}}} \underset{j\to\infty}{\longrightarrow} \frac{1}{p^{q-1}}\jj_q\pa{N(t)^\top\Phi(t,T)^\top  \xi}$$
	pointwise a.e. in $t$, due to the continuity of all the functions involved. Moreover, the above also shows that
\begin{align*}	
	\sup_{j\in\N}\abs{\alpha_p^\ast(t,x_{m_{k_j}},y_{m_{k_j}}) } &=\frac{1}{p^{q-1}}\sup_{j\in\N}\abs{N(t)^\top\Phi(t,T)^\top  \xi_p\pa{x_{m_{k_j}},y_{m_{k_j}}}}^{q-1}\\
	&\leq \frac{1}{p^{q-1}} \norm{\Phi}_{L^\infty\pa{[0,T]\times [0,T]}}^{q-1}\norm{N}_{L^\infty\pa{[0,T]}}^{q-1}\pa{\sup_{j\in\N}\abs{\xi_p\pa{x_{m_{k_j}},y_{m_{k_j}}}}}^{q-1}<\infty,
\end{align*}
	from which we conclude that as
	$$y_{m_{k_j}} = \Phi(0,T)x_{m_{k_j}}+\int_{0}^T\Phi(\tau,T)N(\tau)\alpha_p^\ast\pa{\tau;x_{m_{k_j}},y_{m_{k_j}}}\dd\tau,$$
	taking $j$ to infinity and using the convergence of $\pa{x_{m_{k_j}},y_{m_{k_j}}}_{j\in\N}$, the continuity of $\Phi(0,T)$, and the dominated convergence theorem we get that 
	$$y=\Phi(0,T)x + \int_{0}^T\Phi(\tau,T)N(\tau)\frac{1}{p^{q-1}}\jj_q\pa{N(t)^\top\Phi(t,T)^\top  \xi}\dd\tau.$$
	Denoting by $\alpha(t) := \frac{1}{p^{q-1}}\jj_q\pa{N(t)^\top\Phi(t,T)^\top  \xi}$ we see that $\alpha\in L^p\pa{0,T;\R^n}$ and that $\pa{\gamma_{\alpha}^{0,x},\alpha}\in \pcadm\pa{x,y}$, where $\gamma_{\alpha}^{0,x}$ is defined via \eqref{end-point mapping definition}. This implies that 
\begin{align*}	
	c_p\pa{x,y} & \leq \int_{0}^T \abs{\alpha(t)}^p\dd t \leq \liminf_{j\to\infty}\int_{0}^T \abs{\alpha_p^\ast\pa{t;x_{m_{k_j}},y_{m_{k_j}}}}^p\dd t\\
	&=\liminf_{j\to\infty}c_p\pa{x_{m_{k_j}},y_{m_{k_j}}}=c_{p}\pa{x,y},
\end{align*}
	where we have used Fatou's lemma and the continuity of $c_p$. We conclude that 
	$$c_p\pa{x,y} = \int_{0}^T \abs{\alpha(t)}^p\dd t $$
	and due to the uniqueness of the minimiser
	$$\frac{1}{p^{q-1}}\jj_q\pa{N(t)^\top\Phi(t,T)^\top  \xi_p(x,y)}=\alpha_p^\ast\pa{t;x,y}=\alpha(t) = \frac{1}{p^{q-1}}\jj_q\pa{N(t)^\top\Phi(t,T)^\top  \xi}.$$
	As Theorem \ref{thm:formula of alpha^*} guarantees that the Lagrange multiplier is unique we find that $\xi=\xi_p(x,y)$, which is what we wanted to show. The proof is thus complete.
\end{proof}

We conclude this section by gathering all the results we've shown to prove Theorem \ref{thm:optimal_control_cost}:

\begin{proof}[Proof of Theorem \ref{thm:optimal_control_cost}]
	The proof is an immediate consequence of corollaries \ref{cor:pcadm_is_not_empty} and \ref{cor:upper_and_lower_bounds_on_c_p}, and theorems \ref{p-action minimiser-theorem} and \ref{thm:continuity_of_xi}.
\end{proof}

Now that our study of the system of controlled ODEs is complete, we turn our attention to the study of the generalised continuity equation.

\section{The continuity equation and superposition principles}\label{sec:dynamics}

Let $p>1$, $\mu,\nu \in \mathcal{P}_p(\R^d)$ and $T>0$ be given. In this section we will focus our attention on the generalised continuity equation

\begin{equation}\label{eq:cnt_1}
\left\{	
	\begin{array}{ll}
		\partial_t\rho_t(x) +  \div (\rho_t(x) b(t,x,u_{t}(x)))=0,& \pa{t,x}\in (0,T)\times\R^d,\\
		\rho_0=\mu,\quad \rho_T=\nu,
	\end{array}
\right.
\end{equation}
where the vector field $b:[0,T]\times\R^d\times\R^{n}\to\R^d$ has the special form 
$$
b(t,x,u_{t}(x)):=M(t)x + N(t)u_{t}(x).
$$
Here we recall that $M:[0,T]\to \R^{d\times d}$ and $N:[0,T]\to \R^{d\times n}$ are the matrices from the previous sections.

The unknown of \eqref{eq:cnt_1} is a pair, $\pa{\rho,u}_{t\in[0,T]}$, where $\rho:[0,T]\to \PP_p\pa{\R^d}$ is a narrowly continuous curve such that $\displaystyle\int_0^T\int_{\R^d}\abs{x}^p\dd\rho_t(x)\dd t<\infty $, $u$ is a Borel vector field such that $u\in L^p_tL^p_{\rho_t}([0,T]\times \mathbb{R}^d;\mathbb{R}^n)$, and \eqref{eq:cnt_1} holds in the sense of distribution on $\pa{0,T}\times\R^d$.

We will study \eqref{eq:cnt_1}  by considering measures that are concentrated on paths generated by the ODEs from \eqref{eq:controlled}. The approach and techniques we will use in this section are standard for experts, see, for instance, \cite{AGS08}, yet our detailed study of the cost function $c_p$ and its associated minimiser $\alpha^\ast_p$ will simplify many of our proofs. 

Much like in the previous section, we will assume that Assumption \ref{asmp} holds throughout this section. We will work in the Banach space
\begin{equation}\label{def:X}
\X:=\pa{C(0,T;\mathbb{R}^d),\norm{\cdot}_{L^\infty\pa{[0,T]; \R^d}}} 
\end{equation}
and consider the set 
\begin{equation}\nonumber
	\Gamma_p:=\left\{\gamma\in \mathcal{X}:\gamma'(t)=M(t)\gamma(t)+N(t)\alpha_p^*(t;\gamma(0),\gamma(T))\right\}
\end{equation}

\amit{Notice that any element of $\Gamma_p$ is absolutely continuous. Furthermore,} the fact that $\pa{\gamma^{0,x}_{\alpha^\ast_p\pa{\cdot;x,y}},\alpha^\ast_p\pa{\cdot;x,y}}\in\pcadm\pa{x,y}$ and Remark \ref{rem:gamma_s_x_alpha_is_a_solution} guarantee that $\Gamma_p\neq\emptyset$. In fact, the collection of curves $\br{\gamma^{0,x}_{\alpha^\ast_p\pa{\cdot;x,y}}}_{\pa{x,y}\in \R^d\times\R^d}$ represents $\Gamma_p$ uniquely. This is expressed in the following lemma.
\begin{lemma}
    \label{e_0,T is bijective on Gamma_p}
    Consider the continuous map $\mathfrak{e}_{0,T}:\X\to \R^d\times \R^d $ defined by
    \begin{equation}\label{eq:def_of_mathfrak_e}
    	\mathfrak{e}_{0,T}\pa{\gamma} = \pa{\gamma(0),\gamma(T)}
    \end{equation}
    Then, its restriction on $\Gamma_p$, $\mathfrak{e}_{0,T}|_{\Gamma_p}:\Gamma_p\to \R^d\times \R^d$, is a homeomorphism. In particular, for any $x,y\in\R^d$ we have $\mathfrak{e}_{0,T}|_{\Gamma_p}^{-1}(x,y)=\gamma^{0,x}_{\alpha^*_p(\cdot;x,y)}$.
\end{lemma}
\begin{proof}
    We start by noticing that for any $x,y\in\R^d$ we have that $\gamma^{0,x}_{\alpha^*_p(\cdot;x,y)}\in \Gamma_p$ and $\e_{0,T}(\gamma^{0,x}_{\alpha^*_p(\cdot;x,y)})=(x,y)$, and consequently $\e_{0,T}\vert_{\Gamma_p}$ is surjective. To show injectivity we notice that if $\gamma_1,\gamma_2\in \Gamma_p$, $\gamma_1(0)=\gamma_2(0)$ and $\gamma_1(T)=\gamma_2(T)$ then since $\gamma_1$ and $\gamma_2$ satisfy the same classical ODE (recall that $M$, $N$, and $\alpha_p^\ast\pa{\cdot;x,y}$ are continuous) and have the same initial conditions we must have that $\gamma_1= \gamma_2$. In particular, if $\gamma\in \Gamma_p$ are such that $\e_{0,T}(\gamma)=\pa{x,y}$ then we must have that 
    $$\gamma= \gamma^{0,x}_{\alpha_p^\ast\pa{\cdot;x,y}}.$$
    
    To conclude our lemma we are only left with showing that $\mathfrak{e}_{0,T}|_{\Gamma_p}^{-1}$ is continuous. 
    Indeed, let $\pa{x_n,y_n}_{n\in\N}$ be a sequence in $\R^d\times \R^d$ which converges to $\pa{x,y}\in\R^d\times\R^{d}$. Then 
	\begin{equation}\nonumber
	\begin{split}
		&\abs{\gamma^{0,x_n}_{\alpha_p^\ast\pa{\cdot;x_n,y_n}}(t)-\gamma^{0,x}_{\alpha_p^\ast\pa{\cdot;x,y}}(t)}
		\leq \abs{\Phi(0,t)\pa{x_n-x}} + \int_{0}^t \abs{\Phi\pa{\tau,t}N(\tau)\pa{\alpha_p^\ast\pa{\tau;x_n,y_n}-\alpha_p^\ast\pa{\tau;x,y}}\dd\tau}\\
		&\leq \norm{\Phi}_{L^\infty\pa{[0,T]\times [0,T]}}\pa{1+\norm{N}_{L^\infty\pa{[0,T]}}}\pa{\abs{x_n-x}+ T^{\frac{1}{q}}\pa{\int_{0}^T\abs{\alpha_p^\ast\pa{\tau;x_n,y_n}-\alpha_p^\ast\pa{\tau;x,y}}^p\dd\tau}^{\frac{1}{p}}}.
	\end{split}	
	\end{equation}
	As the right-hand side is independent of $t$ we conclude that
	\begin{equation}\label{eq:cont_invertible_e_0T}
		\begin{split}
			&	\norm{\e_{0,T}\vert_{\Gamma_p}^{-1}\pa{x_n,y_n}-\e_{0,T}\vert_{\Gamma_p}^{-1}\pa{x,y}}_{L^\infty\pa{[0,T]}} \leq \norm{\Phi}_{L^\infty\pa{[0,T]\times [0,T]}}\pa{1+\norm{N}_{L^\infty\pa{[0,T]}}}\\
			&\qquad\qquad\qquad \times\pa{\abs{x_n-x}+ T^{\frac{1}{q}}\pa{\int_{0}^T\abs{\alpha_p^\ast\pa{\tau;x_n,y_n}-\alpha_p^\ast\pa{\tau;x,y}}^p\dd\tau}^{\frac{1}{p}}}.
		\end{split}
	\end{equation} 
	Recall that $\alpha_p^\ast\in C\pa{[0,T]\times\R^d\times\R^d;\R^n}$ according to Theorem \ref{thm:continuity_of_xi}. Defining $f_{n},g_{n},g:[0,T]\to[0,\infty)$ as 
\begin{align*}	
	f_n(t)&:=\abs{\alpha_p^\ast\pa{t;x_n,y_n}-\alpha_p^\ast\pa{t;x,y}}^p,\\
	g_n(t)&:=2^{p-1}\pa{\abs{\alpha_p^\ast\pa{t;x_n,y_n}}^p+\abs{\alpha_p^\ast\pa{t;x,y}}^p},\\
	g(t)&:=2^p\abs{\alpha_p^\ast\pa{t;x,y}}^p,
\end{align*}
	we find that
	$$f_n(t)=\abs{f_n(t)} \leq g_n(t),\qquad f_n(t)\underset{n\to\infty}{\longrightarrow}0,\qquad g_n(t)\underset{n\to\infty}{\longrightarrow}g(t),$$ 
	pointwise, and
	$$\int_{0}^T g_n(t)\dd t = 2^{p-1}\pa{c_p\pa{x_n,y_n}+c_p(x,y)} \underset{n\to\infty}{\longrightarrow}2^p c_p(x,y) = \int_{0}^T g(t)\dd t,$$
	where we have used the definition and continuity of $(x,y)\mapsto c_p(x,y)$ (Corollary \ref{cor:upper_and_lower_bounds_on_c_p} for the latter). Consequently, by the generalised dominated convergence theorem (see Lemma \ref{generalised dct lemma}) and \eqref{eq:cont_invertible_e_0T} we have that 
	$$\lim_{n\to\infty}\norm{\e_{0,T}\vert_{\Gamma_p}^{-1}\pa{x_n,y_n}-\e_{0,T}\vert_{\Gamma_p}^{-1}\pa{x,y}}_{L^\infty\pa{[0,T]}} =0,$$
	which shows the continuity of the inverse map $\e_{0,T}\vert_{\Gamma_p}^{-1}$.
\end{proof}

We can say more:

\begin{lemma}
	\label{Gamma_p is closed lemma}
	$\Gamma_p$ is closed in $\mathcal{X}$ and as such is a Polish space. 
\end{lemma}
\begin{proof}
	Assume that we have a sequence $(\gamma_n)_{n\in\N}$ in  $\Gamma_p$ that converges to some $\gamma\in \X$. We find that 
	$$\gamma(0) = \lim_{n\to\infty}\gamma_n(0),\quad \gamma(T) = \lim_{n\to\infty}\gamma_n(T),\quad \text{and}\quad \alpha_p^\ast\pa{t;\gamma(0),\gamma(T)}=\lim_{n\to\infty}\alpha_{p}^{\ast}\pa{t;\gamma_n(0),\gamma_n(T)}$$
	where we have used the continuity of $\alpha_p^\ast$, guaranteed by Theorem \ref{thm:continuity_of_xi}.
	Moreover, using \eqref{alpha^*_p formula implicit v2} together with \eqref{eq:upper_bound_on_xi_with_c_p} we find that
$$\abs{\alpha^\ast_p\pa{t;\gamma_n(0),\gamma_n(T)}} \leq C_E^{q-1}\frac{1}{p^{q-1}} \norm{\Phi}_{L^\infty\pa{[0,T]\times [0,T]}}^{q-1}\norm{N}_{L^\infty\pa{[0,T]}}^{q-1}\sup_{j\in\N}c_p\pa{\gamma_n(0),\gamma_n(T)}^{\frac{q-1}{q}}<\infty,$$
	where $C_E$ is a constant which was defined in the proof of Theorem \ref{thm:continuity_of_xi}, and where we have used the continuity of $c_p$ (Corollary \ref{cor:upper_and_lower_bounds_on_c_p}) and the convergence of $(\gamma_n(0))_{n\in\N}$ and $(\gamma_n(T))_{n\in\N}$.
	
	Combining the above with the fact that $(\gamma_n)_{n\in\N}$ is a sequence in $\Gamma_p$ we find that for any $t\in [0,T]$
	$$\gamma(t) = \lim_{n\to\infty}\gamma_n(t) = \lim_{n\to\infty}\pa{\Phi(0,t)\gamma_n(0)+ \int_{0}^T \Phi(\tau,T)N(\tau)\alpha^\ast_p\pa{\tau;\gamma_n(0),\gamma_n(T)}\dd\tau}$$
	$$=\Phi(0,t)\gamma(0)+ \int_{0}^T \Phi(\tau,T)N(\tau)\alpha^\ast_p\pa{\tau;\gamma(0),\gamma(T)}\dd\tau.$$
	where we have used the continuity of $\Phi$ and $N$, as well as the dominated convergence theorem. As the above implies that $\gamma\in \Gamma_p$ following Remark \ref{rem:gamma_s_x_alpha_is_a_solution}, the proof is now complete.
\end{proof}

\begin{remark}\label{rem:less_conditions_for_Gamma_p}
	Looking at the proof of Lemma \ref{Gamma_p is closed lemma} we notice that due to the properties of $\alpha^\ast_p$ we did not really need to assume the uniform convergence of $(\gamma_n)_{n\in\N}$ and only needed pointwise convergence.
 \end{remark}

The set $\Gamma_p$ will give us the ability to connect between two measures $\mu$ and $\nu$ on a path of characteristics. 

As the evaluation map $e_t:\X\to \R^d$ defined by
\begin{equation}\label{eq:evoluation_map_at_t}
	e_t\pa{\gamma}:=\gamma(t)
\end{equation}
is continuous on $\X$, we can define a path of probability measures on $\R^d$ from any probability measure $\eta\in \PP\pa{\X} $ by 
$$\eta_t = {e_t}_{\sharp}\eta.$$
If our chosen $\eta$ is concentrated on $\Gamma_p$, we will be able to ``extract'' from the path of measures $\left({e_t}_{\sharp}\eta\right)_{t\in[0,T]}$ a pair $\pa{\rho,u}\in \cadm\pa{{e_0}_{\sharp}\eta,{e_T}_{\sharp}\eta}$.

We will require the next result to bring this intuition to light. 

\begin{theorem}
	\label{thm:relaxed controllability}
	Let $\mu,\nu\in\mathcal{P}_p(\mathbb{R}^d)$. Then the set
	\begin{equation}\label{eq:def_of_A_p}
		\A_p(\mu,\nu)
		:=\{\eta\in\mathcal{P}(\mathcal{X})\;|\; \eta(\Gamma_p)=1,{e_0}_\sharp\eta=\mu, {e_T}_\sharp\eta=\nu\}
	\end{equation}
	is not empty.  
\end{theorem}

\begin{proof}
	Due to the density of empirical measure in $\PP_p\pa{\R^d}$ with respect to the $p-$Wasserstein distance (see, for instance, the discussion in  \cite[Subsection 5.1.2]{CDF2018} with the requirement of a finite second moment in the law Law of Large Numbers replaced by the $p^{th}-$moment version\footnote{This can be shown by using Etemadi's theorem which can be found \cite[Theorem 5.17]{Klenke2020}}) we can find sequences $\pa{x_m}_{m\in\N},\pa{y_m}_{m\in\N}$ in $\R^d$ such that 
	$$\lim_{N\to \infty}W_p(\mu_N,\mu)+W_p(\nu_N,\nu)=0 $$
	where
	$$\mu_N=\frac{1}{N}\sum_{i=1}^N\delta_{x_i},\qquad\nu_N=\frac{1}{N}\sum_{i=1}^N\delta_{y_i}.$$

	We consider the finite set 
	$$\Gamma_p^N = \br{\gamma^{0,x_i}_{\alpha^\ast_p\pa{\cdot;x_i,y_j}}}_{i,j=1,\dots,N}\subset \Gamma_p.$$
	and associate to it the measure $\eta_N\in \PP\pa{\X}$
	$$\eta_N = \frac{1}{N^2}\sum_{i,j=1}^{N}\delta_{\gamma^{0,x_i}_{\alpha^\ast_p\pa{\cdot;x_i,y_j}}}. $$

	We claim that $\pa{\eta_N}_{N\in\N}$ is a tight sequence in $\PP\pa{\X}$. To show this we identify suitable compact sets in $\X$. For a given $R>0$ we define the set 
	$$E_R := \left\{\gamma\in\X\;|\; \abs{\gamma(0)}+\abs{\gamma(T)} \leq R\right\} \cap \Gamma_p. $$
 We find that $E_R=\e_{0,T}\vert_{\Gamma_p}^{-1}(B_R)$ where
    $$B_R:=\{(x,y)\in\R^d\times\R^d: |x|+|y|\leq R\},$$
    and where $\e_{0,T}\vert_{\Gamma_p}$ was defined in Lemma \ref{e_0,T is bijective on Gamma_p}. Since $B_R$ is compact and $\e_{0,T}\vert_{\Gamma_p}$ is a homeomorphism, we conclude that $E_R$ is compact in $\Gamma_p$.
	
	To show the tightness of $\pa{\eta_N}_{N\in\N}$ we need to use the connection between this sequence and $\pa{\mu_N}_{N\in\N}$ and $\pa{\nu_N}_{N\in\N}$. As
	$${e_t}_{\sharp}\eta_N = \frac{1}{N^2}\sum_{i,j=1}^{N}\delta_{\gamma^{0,x_i}_{\alpha^\ast_p\pa{\cdot;x_i,y_j}}(t)}$$
	we see that 
	\begin{equation}\label{eq:end_points_of_paths}
		\begin{split}
			&{e_0}_{\sharp}\eta_N =\frac{1}{N^2}\sum_{i,j=1}^{N}\delta_{x_i} = \mu_N,\\
			 &{e_T}_{\sharp}\eta_N =\frac{1}{N^2}\sum_{i,j=1}^{N}\delta_{y_j} = \nu_N.
		\end{split}
	\end{equation}
	In addition, 
	since for any $\mu_1,\mu_2\in \PP_p\pa{\R^d}$
	$$\int_{\R^d}\abs{x}^p \dd\mu_1 =\int_{\R^d\times\R^d}\abs{x}^p\dd\pi\pa{x,y} \leq 2^{p-1}\int_{\R^d\times\R^d}\abs{x-y}^p\dd\pi\pa{x,y} + 2^{p-1}\int_{\R^d}\abs{y}^p \dd\mu_2(y), $$
	where $\pi\in\Pi\pa{\mu,\nu}$, we conclude that 
	\begin{equation}\nonumber
		\begin{split}
			\sup_{N} &\max\br{\int_{\R^d}\abs{x}^p \dd\mu_N(x),\int_{\R^d}\abs{x}^p \dd\nu_N(x)} \\
			&\leq \sup_{N}\pa{W_p^p(\mu_N,\mu)+W_p^p(\nu_N,\nu)}+\int_{\R^d}\abs{x}^p \dd\mu(x)+\int_{\R^d}\abs{x}^p \dd\nu(x):=\mathcal{R}_p<\infty
		\end{split}	
		\end{equation}
	Utilising these two observations we find that for all $N\in\N$
\begin{align*}
	R\eta_N\pa{\X\setminus E_R} &\leq \int_{\X\setminus E_R}\pa{\abs{\gamma(0)}+\abs{\gamma(T)}}\dd\eta_N(\gamma)\leq  \int_{\X }\pa{\abs{\gamma(0)}+\abs{\gamma(T)}}\dd\eta_N(\gamma)\\
	&=\int_{\X }\pa{\abs{e_0(\gamma)}+\abs{e_T\pa{\gamma}}}\dd\eta_N(\gamma)= \int_{\R^d} \abs{x} \dd{e_0}_{\sharp}\eta_N(x) + \int_{\R^d} \abs{x} \dd{e_T}_{\sharp}\eta_N(x)\\ 
	&=\int_{\R^d}\abs{x}\dd\mu_N(x)+\int_{\R^d}\abs{x}\dd\nu_N(x) \leq 2\mathcal{R}_p^{\frac{1}{p}},
\end{align*}	
	from which we sees that 
	\begin{equation}\label{eq:tightness}
		\sup_{N\in\N}\eta_N\pa{\X\setminus E_R} \leq \frac{2\mathcal{R}_p^{\frac{1}{p}}}{R}.
	\end{equation}
	As $R>0$ was arbitrary and $E_R$ are compact we conclude the desired tightness of $\pa{\eta_N}_{N\in\N}$.
	
	With the tightness of $\pa{\eta_N}_{N\in\N}$ established we invoke Prokhorov's theorem and find a subsequence of $\pa{\eta_N}_{N\in\N}$, $\pa{\eta_{N_k}}_{k\in\N}$, that converges narrowly to some $\eta\in \PP\pa{\X}$. It remains to show that $\eta\in \A_p\pa{\mu,\nu}$ to conclude the proof.

	 Using the fact that $ \Gamma_p$ is closed and the Portemanteau theorem (see, for instance, \cite[Theorem 13.16]{Klenke2020}) we find that 
	$$1\geq \eta\pa{\Gamma_p} \geq  \limsup_{k\to\infty}\eta_{N_k}\pa{\Gamma_p}=1,$$
	where we have used the fact that $\eta_N$ is supported in $\Gamma_p$ for all $N\in\N$. We conclude that
	$$\eta\pa{\Gamma_p}=1,$$
	i.e., $\eta$ is concentrated on $\Gamma_p$. 
	
	Lastly, the continuity of the map $e_t$ for any $t\in[0,T]$ and the fact that $\pa{\eta_{N_k}}_{k\in\N}$ converges narrowly to $\eta$ implies that 
	$${e_0}_{\sharp}\eta = \lim_{k\to\infty}{e_0}_{\sharp}\eta_{N_k} = \lim_{k\to\infty}\mu_{N_k}=\mu,$$
	$${e_T}_{\sharp}\eta = \lim_{k\to\infty}{e_T}_{\sharp}\eta_{N_k} = \lim_{k\to\infty}\nu_{N_k}=\nu,$$
	where we have used \eqref{eq:end_points_of_paths} and the narrow convergence of $\pa{\mu_N}_{N\in\N}$ and $\pa{\nu_N}_{N\in\N}$ to $\mu$ and $\nu$ respectively. In other words
	$$\eta\pa{\Gamma_p}=1,\quad {e_0}_{\sharp}\eta =\mu, \quad {e_T}_{\sharp}\eta =\nu$$
	showing that $\eta\in \A_p\pa{\mu,\nu}$ which is consequently not empty. 
\end{proof}

\amit{\begin{remark}\label{rem:A_p_is_more}
	The set $A_p\pa{\mu,\nu}$ is not only empty -- it is, in fact, homeomorphic to $\Pi\pa{\mu,\nu}$. This will play an important role our study of the generalised Benamou--Brenier type formula, which is the subject of the next section (see Lemma \ref{lem:connection_between_optimisation_sets}).
	\end{remark}
}
The last ingredient we will need to show Theorem \ref{thm:not_empty} is the notion of disintegration of measures. In particular we will use the following theorem, which can be found in \cite[Theorem 5.3.1]{AGS08}.
 \begin{theorem}
 \label{thm:disintegration}
Let $X$, $Y$ be Radon separable metric spaces and let $\mu\in\ \PP\pa{X}$ be given. Let $T:X\to Y$ be a Borel map and define $\nu=T_{\sharp}\mu\in \PP\pa{Y}$. Then there exists a $\nu-$a.e. uniquely determined Borel family of probability measures $\pa{\mu_y}_{y\in Y}\subset \PP\pa{X}$, i.e., $y\mapsto \mu_{y}(B)$ is a Borel map for any Borel set $B$ in $X$, such that 
$$\mu_y\pa{X\setminus T^{-1}(y)}=0,\quad \text{ for }\nu\text{-a.e. }y\in Y$$
and 
\begin{equation}\label{eq:disintegration_integral_consequence}
	\int_{X}f(x)\dd\mu(x) = \int_{Y}\pa{\int_{T^{-1}(y)}f(x)\dd\mu_y(x)}\dd\nu(y)
\end{equation}
for every Borel map $f:X\to [0,+\infty]$. We sometimes use the notation 
$$\dd\mu(x) = \dd\mu_y(x) \dd (T_{\sharp}\mu) (y)$$
to express \eqref{eq:disintegration_integral_consequence}.
\end{theorem}
 As all the spaces we consider, including $\X$, are Polish spaces we will be able to use the disintegration theorem. 

\medskip

Before we \amit{begin} with its proof, let us sketch the main steps we will take to show Theorem \ref{thm:not_empty}. 

The main idea we will use is to create a path of measures that `evolves' on the controlled ODEs \eqref{eq:controlled}. Consequently, a candidate for the family of probability measures $\pa{\rho_t}_{t\in[0,T]}$ such that $\rho_0=\mu$ and $\rho_T=\nu$ will be given by 
$$\rho_t = {e_t}_{\sharp} \eta,$$
where $\eta \in \mathcal{A}_p\pa{\mu,\nu}$. 
 
Motivated by the need for the pair $\pa{\rho,u}$ to solve \eqref{eq:continuity_extended_explicit} in the sense of distributions we see that, at least formally, for any $\phi \in C_c^1\pa{\R^d}$
\begin{align*}\nonumber 
		\frac{\dd}{\dd t}\int_{\R^d} \phi\pa{x}\dd\rho_t(x) &= \frac{\dd}{\dd t}\int_{\Gamma_p}\phi\pa{e_t\pa{\gamma}}\dd\eta\pa{\gamma} = \int_{\Gamma_p}\nabla \phi \pa{\gamma(t)}^\top \gamma'(t)\dd\eta(\gamma)\\
		&=\int_{\Gamma_p}\nabla \phi \pa{\gamma(t)}^\top \pa{M(t)\gamma(t)+N(t)\alpha^\ast_p \pa{t;\gamma(0),\gamma(T)}}\dd\eta(\gamma)\\
		& = \int_{\R^d}\nabla \phi \pa{x}^\top M(t)x\dd\rho_t(x)
		+\int_{\Gamma_p}\nabla \phi\pa{\gamma(t)}^\top N(t)\alpha^\ast_p \pa{t;\gamma(0),\gamma(T)}\dd\eta(\gamma).
\end{align*}

As we want to go back to $\dd\rho_t$ in the last expression above we will need to use the full connection between $\rho_t$ and $\eta(\gamma)$ which is given by the disintegration of $\eta$:
$$\dd\eta(\gamma) = \dd\eta_{t,x}\pa{\gamma}\dd {e_t}_{\sharp}\eta (x) =  \dd\eta_{t,x}\pa{\gamma}\dd \rho_t (x) $$
which gives us (formally)
\begin{align}\nonumber
		\frac{\dd}{\dd t}\int_{\R^d} \phi\pa{x}\dd\rho_t(x) &=  \int_{\R^d}\nabla \phi \pa{x}^\top M(t)x\dd\rho_t(x)\\
		\label{eq:reason_for_u}&+\int_{\Gamma_p}\nabla\phi \pa{x}^\top N(t)\pa{\int_{\bm{e}_t^{-1}\br{x}}\alpha^\ast_p \pa{t;\gamma(0),\gamma(T)}d\eta_{t,x}\pa{\gamma}}\dd\rho_t(x)\\
		\nonumber&= \int_{\R^d}\nabla \phi \pa{x}^\top \pa{M(t)x + N(t)u(t,x)}\dd\rho_t(x),
\end{align}
with 
$$u(t,x) := \int_{\bm{e}_t^{-1}\br{x}}\alpha^\ast_p \pa{t;\gamma(0),\gamma(T)}\dd\eta_{t,x}\pa{\gamma}.$$
We will now make the above intuitive idea more rigorous. To be able to deal with measurability issues with $u$ we will need to consider a `lift' of our measure $\eta$ to include the time variable before we disintegrate.

\begin{proof}[Proof of Theorem \ref{thm:not_empty}]
	We start by noticing that if we show that $\cadm\pa{\mu,\nu}$ is non-empty we will immediately conclude that 
	$$\D_p\pa{\mu,\nu}=\inf_{(\rho,u)\in \cadm(\mu,\nu)}\int_0^T \int_{\mathbb{R}^d}|u(t,x)|^p \,\dd\rho_t(x)\dd t<\infty,$$
	so we can focus on that part alone.
	
	Given $\mu,\nu\in\PP_p\pa{\R^d}$ we choose $\eta\in \A_p\pa{\mu,\nu}$, which we now know to be non-empty, and define
	$$\rho(t,x)=\rho_t(x) := \pa{e_t}_{\sharp}\eta(x).$$
	
	Next, we consider the probability measure $\bm{\eta}\in\PP\pa{[0,T]\times \Gamma_p}$ defined by
	$$\bm{\eta}\pa{t,\gamma} = \frac{\dd t}{T}\otimes \eta\pa{\gamma},$$
	where $\dd t$ is the standard Lebesgue measure on $[0,T]$ with the Borel $\sigma-$algebra. Using the continuous map $\bm{e}:[0,T]\times \Gamma_p \to  [0,T]\times \R^d$
	defined by 
	$$\bm{e}(t,\gamma) = \pa{t,\gamma(t)} = \pa{t,e_t(\gamma)},$$
	and the disintegration theorem we find a $\bm{e}_{\sharp}\bm{\eta}-$a.e. unique Borel family of probability measures $\pa{\bm{\eta}_{t,x}}_{x\in \R^d}\subset \PP\pa{[0,T]\times \Gamma_p}$ such that 
	$$\dd\bm{\eta}(t,\gamma) = \dd\bm{\eta}_{s,x}(t,\gamma) \dd(\bm{e}_{\sharp}\bm{\eta})(s,x).$$
	Using this family we define
	$$u(t,x) := \int_{\bm{e}^{-1}\pa{t,x}} \alpha_p^\ast \pa{s;\gamma(0),\gamma(T)}\dd\bm{\eta}_{t,x}(s,\gamma)=\int_{\bm{e}^{-1}\pa{t,x}} \alpha_p^\ast \pa{s;{e_0}\pa{\gamma},{e_T}\pa{\gamma}}\dd\bm{\eta}_{t,x}(s,\gamma),$$
	where we note that as $\alpha_p^\ast$ is uniquely determined and is continuous in its variables according to Theorem \ref{thm:continuity_of_xi} and as $e_0$ and $e_T$ are continuous maps, the function $u(t,x)$ is well defined and is Borel measurable. 
	
	With the pair $\pa{\rho,u}$ defined, we move towards showing that $\rho_t\in \PP_p\pa{\R^d}$ for any $t\in [0,T]$ and is a narrowly continuous curve such that 
	$\displaystyle\int_0^T\left(\int_{\R^d}\abs{x}^p\dd\rho_t(x)\right)^{\frac{1}{p}}\dd t<\infty $ as well as showing that
	 $u\in L^p_tL^p_{\rho_t}([0,T]\times \mathbb{R}^d;\mathbb{R}^n)$. 
	
	 Starting with the former we find that for any $\phi\in C_b\pa{\R^d}$ and any sequence $\pa{t_n}_{n\in\N}$ that converges to $t\in[0,T]$
	 $$\int_{\R^d}\phi(x)\dd\rho_{t_n}(x) = \int_{\X}\phi\pa{\gamma(t_n)}\dd\eta(\gamma)\underset{n\to\infty}{\longrightarrow}\int_{\X}\phi\pa{\gamma(t)}\dd\eta(\gamma) =\int_{\R^d}\phi(x)\dd\rho_{t}(x), $$
	 where we have used the continuity of $\phi\circ \gamma$ for any $\gamma\in \X$ and the dominated convergence theory (since $\phi$ is bounded). In addition, 
		\begin{equation}\nonumber
		\begin{split}
			&\int_{\R^d}\abs{x}^p \dd\rho_t(x)=\int_{\R^d}\abs{x}^p \dd{e_t}_{\sharp}\eta(x) = \int_{\Gamma_p}\abs{\gamma(t)}^p\dd\eta(\gamma).
		\end{split}
	\end{equation}
	As any $\gamma\in \Gamma_p$ satisfies
	\begin{align*}
	\abs{\gamma(t)} &=\left|\Phi(0,t)\gamma(0) + \int_{0}^t \Phi(\tau,t)N(\tau)\alpha_p^\ast\pa{\tau;\gamma(0),\gamma(T)}\dd\tau\right|\\
	&\leq  \norm{\Phi}_{L^\infty\pa{[0,T]\times[0,T]}}\pa{1+\norm{N}_{L^\infty\pa{[0,T]}}}\pa{\abs{\gamma(0)}+ \int_{0}^T \abs{\alpha_p^\ast\pa{\tau;\gamma(0),\gamma(T)}\dd\tau}},
	\end{align*}
	we see that
		\begin{align*}
			&\int_{\R^d}\abs{x}^p \dd\rho_t(x)\\
			& \leq 2^{p-1}\norm{\Phi}_{L^\infty\pa{[0,T]\times[0,T]}}^p\pa{1+\norm{N}_{L^\infty\pa{[0,T]}}}^p
			\int_{\Gamma_p}\pa{\abs{\gamma(0)}^p+ T^{\frac{p}{q}}\int_{0}^T \abs{\alpha_p^\ast\pa{\tau;\gamma(0),\gamma(T)}}^p\dd\tau}\dd\eta(\gamma) \\
			&= 2^{p-1}\norm{\Phi}_{L^\infty\pa{[0,T]\times[0,T]}}^p\pa{1+\norm{N}_{L^\infty\pa{[0,T]}}}^p\pa{1+T^{\frac{p}{q}}}\int_{\Gamma_p}\pa{\abs{\gamma(0)}^p+ c_p\pa{\gamma(0),\gamma(T)}}\dd\eta(\gamma). 
	\end{align*}
	Using \eqref{eq:explicit_upper_bound_for_c_p} from Remark \ref{rem:special_case_p=2} we find that
	\begin{equation} \label{eq:auxiliary_upper_bound}
		\begin{split}
			&c_p(x,y)  \leq C_p\abs{y-\Phi(0,T)x}^p 
			\leq 2^{p-1}C_p\pa{1+\norm{\Phi(0,T)}}\pa{\abs{x}^p+\abs{y}^p}
		\end{split}
	\end{equation}
	and consequently
	\begin{align*}
			\int_{\R^d}\abs{x}^p \dd\rho_t(x) &
			\leq 2^{p-1}\norm{\Phi}_{L^\infty\pa{[0,T]\times[0,T]}^p}\pa{1+\norm{N}_{L^\infty\pa{[0,T]}}}^p\pa{1+T^{\frac{p}{q}}}\\
			 &\times\pa{1+2^{p-1}C_p\pa{1+\norm{\Phi(0,T)}}}\int_{\Gamma_p}\pa{\abs{\gamma(0)}^p+ \abs{\gamma(T)}^p}\dd\eta(\gamma).
	\end{align*}
	Noticing that 
	\begin{equation}\label{eq:estimating_p_moments}
		\begin{split}
			&	\int_{\Gamma_p}\pa{\abs{\gamma(0)}^p+ \abs{\gamma(T)}^p}\dd\eta(\gamma) = \int_{\R^d}\abs{x}^p \dd\pa{e_0}_{\sharp}\eta(x)+\int_{\R^d}\abs{x}^p \dd{e_T}_{\sharp}\eta(x)\\
			&\qquad\qquad	= \int_{\R^d}\abs{x}^p \dd\mu(x) + \int_{\R^d}\abs{x}^p \dd\nu(x),
		\end{split}
	\end{equation}
	where we have used the fact that $\eta\in \A_p\pa{\mu,\nu}$, we conclude that 
	\begin{align} \label{eq:p_moment_of_rho_t_v1}
			\sup_{t\in [0,T]}\int_{\R^d}\abs{x}^p \dd\rho_t(x) &\leq 2^{p-1}\norm{\Phi}_{L^\infty\pa{[0,T]\times[0,T]}^p}\pa{1+\norm{N}_{L^\infty\pa{[0,T]}}}^p\pa{1+T^{\frac{p}{q}}}\\
			\nonumber & \times  \pa{1+2^{p-1}C_p\pa{1+\norm{\Phi(0,T)}}}\pa{\int_{\R^d}\abs{x}^p \dd\mu(x) + \int_{\R^d}\abs{x}^p \dd\nu(x)}<\infty,
	\end{align}
	showing the needed requirements on $\pa{\rho_t}_{t\in [0,T]}$.
	
	To estimate the $L^p_t L^p_{\rho_t}([0,T]\times \mathbb{R}^d)$ norm of $u$ we start by noticing that for any Borel function $\phi$ on $[0,T]$ and any bounded Borel function $\psi$ on $\R^d$ we have that 
\begin{align*}	
	 \int_{[0,T]\times \R^d}\phi(s)\psi(x) \dd(\bm{e}_{\sharp}\bm{\eta})(s,x)&= \int_{[0,T]\times \Gamma_p}\phi(s)\psi\pa{\gamma(s)} \dd\bm{\eta}(s,\gamma)\\
	&=\frac{1}{T}\int_{0}^T \phi(s) \pa{\int_{\Gamma_p}\psi\pa{\gamma(s)}\dd\eta(\gamma)}\dd s=\frac{1}{T}\int_{0}^T \phi(s) \pa{\int_{\R^d}\psi\pa{x}\dd\pa{e_s}_{\sharp}\eta(x)}\dd s\\
	&=\frac{1}{T}\int_{0}^T \pa{\int_{\R^d}\phi(s)\psi\pa{x}\dd{e_s}_{\sharp}\eta(x)}\dd s,
\end{align*}
	where the Borel measurability of the map $s\mapsto \int_{\R^d}\phi(s)\psi\pa{x}\dd{e_s}_{\sharp}\eta(x)$ is guaranteed by the Fubini--Tonelli theorem. As the subalgebra generated by functions of the form $\phi(s)\psi(x)$ with $\phi\in C\pa{[0,T]}$ and $\psi\in C\pa{B_n(0)}$ is dense in $C\pa{[0,T]\times B_n(0)}$ with respect to the supremum norm according to the Stone--\amit{Weierstrass} theorem, we conclude that for any $\Psi\in C\pa{[0,T]\times \R^n}$ we can find \amit{a sequence $\pa{k_m}_{m\in\N}\subset \N$ and a sequence of functions 
	$$\Psi_m(s,x):=\sum_{i=1}^{k_m}\phi_{m,i}(s)\psi_{m,i}(x)$$ 
	with $\phi_{m,i}\in C\pa{[0,T]}$ and $\psi_{m,i}\in C\pa{B_n(0)}$ for $i\in \br{1,\dots,k_m}$ such that
	\begin{equation}\label{eq:uniform_convergence_approximation}
		\sup_{\pa{s,x}\in[0,T]\times \R^d} \abs{\Psi_m(s,x) \xi_{n}(x) - \Psi(s,x)\xi_{n}(x)}\underset{m\to\infty}{\longrightarrow}0,
	\end{equation}
	where $\xi_n\in C\pa{\R^d}$ satisfies $0\leq \xi_n\leq 1$, $\xi_n \vert_{B_n(0)}=1$, and $\xi_n\vert_{B_{n+1}(0)^c}=0$.
	This implies that 
	$$\sup_{s\in [0,T]}\abs{\int_{\R^d}\Psi(s,x)\xi_n(x)\dd{e_s}_{\sharp}\eta(x) - \int_{\R^d}\Psi_m(s,x)\xi_n(x)\dd{e_s}_{\sharp}\eta(x)}\underset{m\to\infty}{\longrightarrow}0. $$
	Since 
	$$s\mapsto \int_{\R^d}\Psi_m(s,x)\xi_n(x)\dd{e_s}_{\sharp}\eta(x) = \sum_{i=1}^{k_m}\int_{\R^d}\phi_{m,i}(s)\psi_{m,i}(x)\xi_n(x)\dd{e_s}_{\sharp}\eta(x)$$
	
	is Borel measurable for any $m\in\N$, we conclude the Borel measurability of $s\mapsto \int_{\R^d}\Psi(s,x)\xi_n(x)\dd{e_s}_{\sharp}\eta(x)$ and that
\begin{align*}	
	&\frac{1}{T}\int_{0}^T \pa{\int_{\R^d}\Psi(s,x)\xi_n(x)\dd{e_s}_{\sharp}\eta(x)}\dd s 
    = \lim_{m\to\infty}\frac{1}{T}\int_{0}^T \pa{\int_{\R^d}\Psi_m(s,x)\xi_n(x)\dd{e_s}_{\sharp}\eta(x)}\dd s\\
	&=\lim_{m\to\infty}\sum_{i=1}^{k_m}\frac{1}{T}\int_{0}^T \int_{\R^d}\phi_{m,i}(s)\psi_{m,i}(x)\xi_n(x)\dd{e_s}_{\sharp}\eta(x)\dd s = \lim_{m\to\infty}\sum_{i=1}^{k_m}\int_{[0,T]\times \R^d}\phi_{m,i}(s)\psi_{m,i}(x)\xi_n(x)\dd(\bm{e}_{\sharp}\bm{\eta})(s,x)\\
	&=\lim_{m\to\infty}\int_{[0,T]\times \R^d}\Psi_m(s,x)\xi_n(x) \dd(\bm{e}_{\sharp}\bm{\eta})(s,x) =\int_{[0,T]\times \R^d}\Psi(s,t)\xi_n(x) \dd(\bm{e}_{\sharp}\bm{\eta})(s,x), 
\end{align*}
}
	where we have used \eqref{eq:uniform_convergence_approximation} again. Using the dominated convergence theorem we conclude that for any $\Psi\in C_b\pa{[0,T]\times \R^d}$ we have that 
	$$\int_{\R^d}\Psi(s,x)\dd{e_s}_{\sharp}\eta(x)=\lim_{n\to\infty}\int_{\R^d}\Psi(s,x)\xi_n(x)\dd{e_s}_{\sharp}\eta(x),$$
	which shows the Borel measurability of $[0,T]\ni s\mapsto \int_{\R^d}\Psi(s,x)\dd{e_s}_{\sharp}\eta(x)$ and
	\begin{equation}\label{eq:L_p_for_u_v1}
		\frac{1}{T}\int_{0}^T \pa{\int_{\R^d}\Psi(s,x)\dd{e_s}_{\sharp}\eta(x)}\dd s =\int_{[0,T]\times \R^d}\Psi(s,x) \dd(\bm{e}_{\sharp}\bm{\eta})(s,x).
	\end{equation} 
	The above is enough to show that $[0,T]\ni s\mapsto\int_{\R^d}\Psi(s,x)\dd{e_s}_{\sharp}\eta(x)$ is Borel measurable and that \eqref{eq:L_p_for_u_v1} holds for any non-negative measurable function $\Psi$. \amit{The measurability of the above function} is a bit more delicate than it seems for merely Borel function $\Psi$, and we postpone its proof to Appendix \S\ref{app:additional} (see Lemma \ref{applem:missing_measurability}).

	Utilising \eqref{eq:L_p_for_u_v1} we find that
	\begin{align}\label{eq:bound_on_L_p_of_u_wrt_c_p}
		\int_0^T\int_{\R^d} &\abs{u(t,x)}^p \dd\rho_t(x)\dd t = \int_0^T\int_{\R^d} \abs{u(t,x)}^p \dd{e_t}_{\sharp}\eta(x)\dd t =T\int_0^T\int_{\R^d} |u(t,x)|^p \dd\bm{e}_{\sharp}\bm{\eta}(t,x) \\
			\nonumber& \leq T\int_{0}^T\int_{\R^d}\pa{\int_{ \bm{e}^{-1}\pa{t,x}} \abs{\alpha_p^\ast \pa{s;\gamma(0),\gamma(T)}}^p\dd\bm{\eta}_{t,x}(s,\gamma)}\dd\bm{e}_{\sharp}\bm{\eta}(t,x) \\
			\nonumber &= T\int_{[0,T]\times \Gamma_p} \abs{\alpha_p^\ast \pa{s;\gamma(0),\gamma(T)}}^p \dd\bm{\eta}(s,\gamma) = \int_{\Gamma_p} \pa{\int_{0}^T \abs{\alpha_p^\ast \pa{s;\gamma(0),\gamma(T)}}^pds}\dd\eta(\gamma)\\
			\nonumber &=\int_{\Gamma_p}c_p\pa{\gamma(0),\gamma(T)}\dd\eta(\gamma)
	\end{align}

Using \eqref{eq:auxiliary_upper_bound} and the fact that $\eta\in \A_p\pa{\mu,\nu}$ we find that

$$\int_0^T\int_{\R^d} \abs{u(t,x)}^p \dd\rho_t(x)\dd t \leq 2^{p-1}C_p\pa{1+\norm{\Phi(0,T)}} \pa{\int_{\R^d} \abs{x}^p\dd\mu(x) +\int_{\R^d}\abs{x}^p\dd\nu(x)}<\infty,$$
and as such $u\in L^p_t L^p_{\rho_t}\pa{[0,T]\times \R^d}$ as claimed. 

To conclude our proof we are left with showing that the pair $\pa{\rho,u}$ solves our generalised continuity equation \eqref{eq:cnt_1} in the sense of distributions.

Given $\phi\in C_c^1\pa{(0,T)\times \R^d}$ we find that 
\begin{align*}
\int_{0}^T \int_{\R^d}\abs{N(t)u(t,x)\cdot \nabla_x\phi(t,x)}d\rho_t(x)\dd t
&\leq \norm{\phi}_{W^{1,\infty}\pa{[0,T]\times\R^d}}\norm{N}_{L^\infty\pa{[0,T]}}\int_{0}^T \int_{\R^d}\abs{u(t,x)}\dd\rho_t(x)\dd t\\
&\leq  \norm{\phi}_{W^{1,\infty}}\norm{N}_{L^\infty\pa{[0,T]}} \norm{u}_{L^1_{t} L^p_{\rho_t}\pa{[0,T]\times\R^d;\R^n}}<\infty.
\end{align*}
We conclude that the positive and negative parts of $ N(t)u(t,x)\cdot \nabla_x\phi(t,x)$ are $\rho_t(x)\dd t$ integrable and as such, by breaking the integration in two and recombining it back, we can use \eqref{eq:L_p_for_u_v1}. Consequently,
\begin{align*}\nonumber
	\int_{0}^T \int_{\R^d} N(t)u(t,x)&\cdot \nabla_x\phi(t,x)\dd\rho_t(x)\dd t  = T\int_{[0,T]\times \R^d} N(t)u(t,x)\cdot \nabla_x\phi(t,x)\dd(\bm{e}_{\sharp}\bm{\eta})(t,x)\\
	&=T\int_{[0,T]\times \Gamma_p} \int_{\bm{e}^{-1}(t,x)}N(t)\alpha_p^\ast\pa{s;\gamma(0),\gamma(T)}\cdot \nabla _x\phi\pa{t,x}d\bm{\eta}_{t,x}\pa{s,\gamma}d\bm{e}_{\sharp}\bm{\eta}(t,x)
\end{align*}
Since
$$\bm{e}^{-1}\pa{t,x}=\br{\pa{r,\gamma}\in [0,T]\times \Gamma_p\;|\; \pa{r,\gamma(r)}=\pa{t,x}} = \br{t}\times e_{t}^{-1}\pa{x}$$
we find that 
\begin{align}\label{eq:towards_continuity_eq_v1}
		\int_{0}^T \int_{\R^d} N(t)u(t,x)&\cdot \nabla_x\phi(t,x)\dd\rho_t(x)\dd t \\
		\nonumber&=T\int_{[0,T]\times \Gamma_p} \int_{\br{t}\times e_t^{-1}(x)}N(t)\alpha_p^\ast\pa{s;\gamma(0),\gamma(T)}\cdot \nabla _x\phi\pa{t,x}\dd\bm{\eta}_{t,x}\pa{s,\gamma}\dd\bm{e}_{\sharp}\bm{\eta}(t,x)\\
		\nonumber&=T\int_{[0,T]\times \Gamma_p} \int_{\br{t}\times e_t^{-1}(x)}N(s)\alpha_p^\ast\pa{s;\gamma(0),\gamma(T)}\cdot \nabla _x\phi\pa{s,\gamma(s)}\dd\bm{\eta}_{t,x}\pa{s,\gamma}\dd\bm{e}_{\sharp}\bm{\eta}(t,x)\\
		\nonumber&=T\int_{[0,T]\times \Gamma_p} \int_{\bm{e}^{-1}\pa{t,x}}N(s)\alpha_p^\ast\pa{s;\gamma(0),\gamma(T)}\cdot \nabla _x\phi\pa{s,\gamma(s)}\dd\bm{\eta}_{t,x}\pa{s,\gamma}\dd\bm{e}_{\sharp}\bm{\eta}(t,x)\\
		\nonumber&=T\int_{[0,T]\times \Gamma_p} N(s)\alpha_p^\ast\pa{s;\gamma(0),\gamma(T)}\cdot \nabla _x\phi\pa{s,\gamma(s)}\dd\bm{\eta}(s,\gamma),
\end{align}
where we have used the continuity of $N$, $\alpha_p^\ast$, and $\nabla_x \phi$.

Similarly, since 
\begin{align*}
\int_{0}^T \int_{\R^d}\abs{M(t)x\cdot \nabla_x\phi(t,x)}\dd\rho_t(x)\dd t & \leq \norm{\phi}_{W^{1,\infty}\pa{[0,T]\times\R^d}}\norm{M}_{L^\infty\pa{[0,T]}}\int_{0}^T \pa{\int_{\R^d}\abs{x}^p\dd\rho_t(x)}^{\frac{1}{p}}\dd t\\
&\leq \norm{\phi}_{W^{1,\infty}\pa{[0,T]\times\R^d}}\norm{M}_{L^\infty\pa{[0,T]}}T\sup_{t\in[0,T]} \pa{\int_{\R^d}\abs{x}^p\dd\rho_t(x)}^{\frac{1}{p}}<\infty,
\end{align*}
where we have used \eqref{eq:p_moment_of_rho_t_v1}, we see that we can apply \eqref{eq:L_p_for_u_v1} to find that 
\begin{align}\label{eq:towards_continuity_eq_v2}
		\int_{0}^T \int_{\R^d}M(t)x\cdot \nabla_x\phi(t,x)\dd\rho_t(x)\dd t  &=T \int_{[0,T]\times \R^d}M(t)x\cdot \nabla_x\phi(t,x)\dd\bm{e}_{\sharp}\bm{\eta}(t,x)\\
		\nonumber&=T \int_{[0,T]\times \R^d}M(t)\gamma(t)\cdot \nabla_x\phi(t,\gamma(t))\dd\bm{\eta}(t,\gamma).
\end{align}
Combining \eqref{eq:towards_continuity_eq_v1} and \eqref{eq:towards_continuity_eq_v2} we conclude that 
\begin{align*}\nonumber 
	\int_{0}^T \int_{\R^d} &\pa{M(t)x+N(t)u(t,x)}\cdot \nabla_x\phi(t,x)\dd\rho_t(x)\dd t\\ 
	&=T \int_{[0,T]\times \Gamma_p}\pa{ M(t)\gamma(t)+N(t)\alpha_p^{\ast}\pa{t;\gamma(0),\gamma(T)}}\cdot \nabla_x\phi(t,\gamma(t))\dd\bm{\eta}(t,\gamma)\\
	&\underset{\gamma\in \Gamma_p}{=}T \int_{[0,T]\times \Gamma_p}\gamma^\prime(t)\cdot \nabla_x\phi(t,\gamma(t))\dd\bm{\eta}(t,\gamma), 
\end{align*}
which implies that 
\begin{equation}\nonumber 
	\begin{split}
		&\int_{0}^T \int_{\R^d} \pa{\partial_t\phi\pa{t,x}+\pa{M(t)x+N(t)u(t,x)}\cdot \nabla_x\phi(t,x)}\dd\rho_t(x)\dd t \\
		&=T\int_{[0,T]\times \R^d} \partial_t\phi (t,x)\dd\bm{e}_{\sharp}\bm{\eta}(t,x) + T \int_{[0,T]\times \Gamma_p}\gamma^\prime(t)\cdot \nabla_x\phi(t,\gamma(t))\dd\bm{\eta}(t,\gamma) \\
		&=T \int_{[0,T]\times \Gamma_p}\pa{\partial_t\phi (t,\gamma(t))+\gamma^\prime(t)\cdot \nabla_x\phi(t,\gamma(t))}\dd\bm{\eta}(t,\gamma)\\
		&=\int_{\Gamma_p}\pa{\int_0^T \frac{\dd}{\dd t}\phi(t,\gamma(t)) \dd t}\dd\eta(\gamma) = \int_{\Gamma_p}\pa{\phi(T,\gamma(T)) - \phi(0,\gamma(0))}\dd\eta(\gamma)=0,
	\end{split}	
\end{equation}
since $\phi\in C_c^1\pa{(0,T)\times \R^d}$. The proof is thus complete. 
\end{proof}

\begin{remark}\label{rem:constructed_solution_for_generalised_continuity}
	Looking back at the proof of Theorem \ref{thm:not_empty} we notice that we have managed to show slightly more than just finding a pair $\pa{\rho,u}$ in $\cadm\pa{\mu,\nu}$. We showed that for any $\eta\in \A_p\pa{\mu,\nu}$ the pair
	\begin{equation}\label{eq:def_rho_t_and_u}
		\begin{split}
			&\rho_t = {e_t}_{\sharp}\eta,\\
			&u(t,x)= \int_{\bm{e}^{-1}\pa{t,x}} \alpha_p^\ast \pa{s;\gamma(0),\gamma(T)}\dd\bm{\eta}_{t,x}(s,\gamma),
		\end{split}
	\end{equation}
	where $\bm{\eta}_{t,x}(s,\gamma)$ is attained from the disintegration of $\bm{\eta}\pa{t,\gamma} = \frac{\dd t}{T}\otimes \eta(\gamma)$ 	with respect to the map $\bm{e}\pa{t,\gamma}=\pa{t,\gamma(t)}$, 
 solves the continuity equation in the sense of distribution and satisfies
	\begin{equation}\nonumber 
		\sup_{t\in[0,T]}\int_{\R^d}\abs{x}^p \dd\rho_t(x) \leq C_{\Phi,N,p}\pa{\int_{\R^d}\abs{x}^p\dd\mu(x)+\int_{\R^d}\abs{x}^p\dd\nu(x)},
	\end{equation}
	where $C_{\Phi,N,p}>0$ is a fixed constant that depends only on $\Phi$, $N$, and $p$, and
	\begin{equation}\label{eq:upper_bound_of_norm_of_u}
		\begin{split}
		& 	\int_0^T\int_{\R^d} \abs{u(t,x)}^p \dd\rho_t(x)\dd t \leq \int_{\Gamma_p}c_p\pa{\gamma(0),\gamma(T)}\dd\eta(\gamma).
	\end{split}
	\end{equation}
\end{remark}

We are now left with only one remaining goal - the proof of our generalised Benamou-Brenier formula, Theorem \ref{thm:general_bb}.

\section{A generalised Benamou--Brenier type formula}\label{sec:general_bb}

As was mentioned in \S\ref{sec:intro}, the Benamou--Brenier formula connects between a static description of an optimal transportation problem, captured by the cost function $c_p\pa{x,y}$, and a dynamic description associated to the continuity equation \eqref{eq:continuity_extended_explicit}. Each of these descriptions boasts an appropriate set on which we optimise -- $\Pi\pa{\mu,\nu}$ for the former and $\cadm\pa{\mu,\nu}$ for the latter. While seemingly unrelated, the proof of Theorem \ref{thm:not_empty} and Remark \ref{rem:constructed_solution_for_generalised_continuity} give us an inkling to an intimate connection between these sets which we will utilise in our proof of the generalised Benamou--Brenier formula. This intuition is captured in the following lemma:

\begin{lemma}\label{lem:connection_between_optimisation_sets}
	Recall the definition of $\mathcal{X}$ from \eqref{def:X} 
	and the definition of the map $\e_{0,T}:\X \to \R^d \times \R^d$ from \eqref{eq:def_of_mathfrak_e}
	and define the map $\EE_{0,T}:\PP\pa{\X}\to \PP\pa{\R^d\times \R^d}$ by
	$$\EE_{0,T}\pa{\eta} = {\e_{0,T}}_{\sharp}\eta.$$
	Then, for any $\mu,\nu\in \PP_p\pa{\R^d}$ the map $\EE_{0,T}$ is a bijection between $\A_p\pa{\mu,\nu}$ and $\Pi\pa{\mu,\nu}$.
\end{lemma}

\begin{proof}
	We start by noticing the following: for any $\eta\in\PP\pa{\X}$
	\begin{align}\label{eq:first_margina_of_E_0T}
			\EE_{0,T}\pa{\eta}\pa{A\times \R^d} &= \int_{\X} \chi_{A\times \R^d}\pa{\gamma(0),\gamma(T)}\dd\eta\pa{\gamma}\\
			\nonumber&=\int_{\X} \chi_{A }\pa{\gamma(0)}\dd\eta\pa{\gamma}=\int_{\R^d}\chi_A(x) \dd {e_0}_{\sharp}\eta(x) ={e_0}_{\sharp}\eta(A).
	\end{align}
	Similarly
	\begin{equation}\label{eq:second_margina_of_E_0T}
		\begin{split}
			&\EE_{0,T}\pa{\eta}\pa{ \R^d\times A} ={e_T}_{\sharp}\eta(A),
		\end{split}
	\end{equation}
	and as such we conclude that $\EE_{0,T}\pa{\eta}\in \Pi\pa{{e_0}_{\sharp}\eta,{e_T}_{\sharp}\eta}$. In particular, if $\eta\in \A_p\pa{\mu,\nu}$ then $\EE_{0,T}\pa{\eta}\in \Pi\pa{\mu,\nu}$.

	To show the surjectivity of $\EE_{0,T}$ we employ an approximation argument similar to that presented in the proof of Theorem \ref{thm:relaxed controllability}. Given $\pi\in \Pi\pa{\mu,\nu}$ we can find a sequence of points ${\pa{x_n,y_n}}_{n\in\N}$ in  $\R^d\times \R^d$ such that 
	$$W_p\pa{\pi_N,\pi}\underset{N\to\infty}{\longrightarrow}0,$$
	where $\pi_N= \frac{1}{N}\sum_{n=1}^N \delta_{\pa{x_n,y_n}}$. Defining 
	$$\eta_N = \frac{1}{N}\sum_{n=1}^N \delta_{\gamma^{0,x_n}_{\alpha^\ast_p\pa{\cdot;x_n,y_n}}},$$
	we find that $\eta_N$ is a probability measure that is concentrated on $\Gamma_p$ and 
	$$\EE_{0,T}\pa{\eta_N} = \frac{1}{N}\sum_{i=1}^N \EE_{0,T}\pa{\delta_{\gamma^{0,x_n}_{\alpha^\ast_p\pa{\cdot;x_n,y_n}}}} = \pi_N,$$
	where we have used the fact that for any Borel sets $A,B\subset \R^d$
	$$\EE_{0,T}\pa{\delta_{\gamma}}\pa{A\times B} = \int_{\Gamma_p} \chi_{A\times B}\pa{\xi(0),\xi(T)}\dd\delta_{\gamma}(\xi)=\chi_{A\times B}\pa{\gamma(0),\gamma(T)} = \delta_{\pa{\gamma(0),\gamma(T)}}\pa{A\times B},$$
	and the fact that $\gamma^{0,x_n}_{\alpha^\ast_p\pa{\cdot;x_n,y_n}}(0)=x_n$ and $\gamma^{0,x_n}_{\alpha^\ast_p\pa{\cdot;x_n,y_n}}(T)=y_n$.
	
	 The sequence $\pa{\eta_N}_{N\in\N}$  is a tight sequence in $\PP\pa{\Gamma_p}$. This can be shown by following a similar arguments to those presented in the proof of Theorem \ref{thm:relaxed controllability}, i.e. considering the compact sets $E_R=\e_{0,T}\vert_{\Gamma_p}^{-1}\pa{B_R}$ and showing that $\eta_N(\X\setminus E_R)$ goes to zero as $R$ goes to infinity uniformly in $N$.
	
    Invoking Prokhorov's theorem we can find a subsequence of $\pa{\eta_N}_{N\in\N}$, $\pa{\eta_{N_k}}_{k\in\N}$, that converges narrowly to some $\eta\in \PP\pa{\X}$. Much like in the proof of Theorem \ref{thm:relaxed controllability}, the fact that $\eta_N$ is concentrated on $\Gamma_p$ for any $n\in\N$ and the fact that $\Gamma_p$ is closed implies (by use of the Portmanteau theorem) that $\eta$ is also concentrated on $\Gamma_p$. Moreover, for any $\phi\in C_b\pa{\R^d\times\R^d}$
\begin{align*}	
	\int_{\R^d\times \R^d} \phi(x,y) \dd\EE_{0,T}(\eta) \pa{x,y} &= \lim_{k\to\infty}\int_{\X}\phi\pa{\e_{0,T}\pa{\gamma}}\dd\eta_{N_k}\pa{\gamma}\\
	& = \lim_{k\to\infty}\int_{\R^d\times\R^d}\phi(x,y)\dd\pi_{N_k}(x,y)=\int_{\R^d\times\R^d}\phi(x,y)\dd\pi(x,y).
\end{align*}
	As $\phi$ was arbitrary and $\R^d\times\R^d$ is Polish, we conclude that $\EE_{0,T}(\eta)=\pi$ which shows the desired surjectivity.
	
	Next we consider the injectivity of $\EE_{0,T}$. We know from Lemma \ref{e_0,T is bijective on Gamma_p} that $\e_{0,T}\vert_{\Gamma_p}^{-1}$ is continuous, and as such Borel measurable. Consequently, for any $\eta\in \A_p\pa{\mu,\nu}$ we have that $\pa{\e_{0,T}\vert_{\Gamma_p}^{-1}}_{\sharp}\EE_{0,T}(\eta)$ is a probability measure on $\Gamma_p$. Moreover, for any Borel set $A\subset \X$ we have that 
\begin{align*}	
	\pa{\e_{0,T}\vert_{\Gamma_p}^{-1}}_{\sharp}\EE_{0,T}(\eta)(A\cap \Gamma_p)&=\EE_{0,T}(\eta)(\e_{0,T}\vert_{\Gamma_p}\pa{A\cap \Gamma_p})\\
	&=\eta\pa{\e_{0,T}^{-1}\pa{\e_{0,T}\vert_{\Gamma_p}\pa{A\cap \Gamma_p}}}\\
	&=\eta\pa{\e_{0,T}^{-1}\pa{\e_{0,T}\vert_{\Gamma_p}\pa{A\cap \Gamma_p}}\cap \Gamma_p}\\
	&=\eta\pa{\e_{0,T}\vert_{\Gamma_p}^{-1}\pa{\e_{0,T}\vert_{\Gamma_p}\pa{A\cap \Gamma_p}}}=\eta\pa{A\cap \Gamma_p}=\eta(A),
\end{align*}
	where we have used the facts that $\e_{0,T}\vert_{\Gamma_p}$ is a homeomorphism and $\eta$ is concentrated on $\Gamma_p$. 
	We conclude that if $\EE_{0,T}(\eta_1)=\EE_{0,T}(\eta_2)$ for $\eta_1,\eta_2\in \A_p\pa{\mu,\nu}$ then for any Borel set $A\subset \X$
	$$\eta_1(A) = \pa{\e_{0,T}\vert_{\Gamma_p}^{-1}}_{\sharp}\EE_{0,T}(\eta_1)(A\cap \Gamma_p)=\pa{\e_{0,T}\vert_{\Gamma_p}^{-1}}_{\sharp}\EE_{0,T}(\eta_2)(A\cap \Gamma_p)=\eta_2(A),$$
	from which we conclude the injectivity, and as such the bijectivity, of $\EE_{0,T}$ from $\A_p\pa{\mu,\nu}$ to $\Pi\pa{\mu,\nu}$.
\end{proof}

\amit{\begin{remark}\label{rem:simplified_connection}
	The keen eyed reader would notice that we could have used the homeomorphism $\mathfrak{e}_{0,T}|_{\Gamma_p}$ not only show the injectivity of $\EE_{0,T}$ but also its surjectivity. Nonetheless, we have elected to provide to proof above to illustrate our ability to attain an explicit approximating sequence from which we can extract the pre-image of a given $\pi\in \Pi\pa{\mu,\nu}$.
	\end{remark}
}

With that at hand we  show the first part of our main result. 
\begin{lemma}\label{lem:half_of_BB_easy_side}
	Let $\mu,\nu\in\PP_p\pa{\R^d}$ and let $\D_{p}(\mu,\nu)$ and $\CC_{p}(\mu,\nu)$ be given in \eqref{controlled p action eq} and \eqref{intro:static}, respectively. Then there exists $\pi^\ast\in \Pi\pa{\mu,\nu}$ and $\pa{\rho^\ast,u^\ast}\in \cadm\pa{\mu,\nu}$  such that 
	\begin{equation}\label{eq:BB_easy_direction_plus_max}
		\D_p\pa{\mu,\nu} \leq \int_{0}^T \pa{\int_{\R^d}\abs{u^\ast(x,t)}^p\dd\rho^{\ast}_t(x)}\dd t \leq \int_{\R^d\times\R^d}c_p(x,y)\dd\pi^\ast(x,y) =  \CC_p\pa{\mu,\nu}.
	\end{equation}
\end{lemma}

\begin{proof}
	By Corollary \ref{cor:upper_and_lower_bounds_on_c_p} we know that the cost function $(x,y)\mapsto c_p(x,y)$ is continuous. Consequently, 
	$$\CC_p\pa{\mu,\nu}=\inf_{\pi\in \Pi\pa{\mu,\nu}}\int_{\R^d\times\R^d}c_p(x,y)\dd\pi(x,y)$$
	has a minimiser $\pi^\ast\in \Pi\pa{\mu,\nu}$ by the direct method of calculus of variations (since $\Pi(\mu,\nu)$ is compact). Using Lemma \ref{lem:connection_between_optimisation_sets} we can find $\eta\in\A_p\pa{\mu,\nu}$ such that $\EE_{0,T}(\eta^\ast)=\pi^\ast$. Following the proof and notations of Theorem \ref{thm:not_empty} and Remark \ref{rem:constructed_solution_for_generalised_continuity} we define
		\begin{align*}
			\rho^\ast_t(x) &= {e_t}_{\sharp}\eta^\ast,\\
			u^\ast(t,x) & = \int_{\bm{e}^{-1}\pa{t,x}} \alpha_p^\ast \pa{s;\gamma(0),\gamma(T)}\dd\bm{\eta^\ast}_{t,x}(s,\gamma),
	\end{align*}
	and using \eqref{eq:upper_bound_of_norm_of_u} we find that
\begin{align*}	
	\D_p\pa{\mu,\nu} &\leq \int_{0}^T \pa{\int_{\R^d}\abs{u^\ast(x,t)}^p\dd\rho^\ast_t(x)}\dd t\leq  \int_{\Gamma_p}c_p\pa{\gamma(0),\gamma(T)}\dd\eta^\ast(\gamma)\\
	&=\int_{\R^d\times\R^d}c_p\pa{x,y}\dd\EE_{0,T}(\eta^\ast)(x,y)=\int_{\R^d\times\R^d}c_p(x,y)\dd\pi^\ast(x,y) = \CC_p\pa{\mu,\nu},
\end{align*}
	from which the result follows.
\end{proof}

The last ingredient we need to show the Benamou--Brenier type formula is the \textit{superposition principle} (\cite[Theorem 8.2.1]{AGS08}) which we state here for the sake of completeness:
\begin{theorem}\label{thm:superposition}
	Let $\mu:[0,T]\to \PP\pa{\R^d}$ be a narrowly continuous family of Borel probability measures solving the continuity equation
	\begin{equation}\nonumber
	\partial_{t} \mu_t(x) + \nabla\cdot\pa{v(t,x)\mu_t(x)}=0,\qquad x\in \R^d,\; t\in (0,T)	
	\end{equation}
	in the sense of distributions for a suitable Borel vector field $v$ satisfying 
	$$\int_{0}^T \pa{\int_{\R^d}\abs{v(t,x)}^pd\mu_t(x)} \dd t <\infty$$ 
	for some $p>1$. Then there exists a probability measure $\bm{\tilde{\eta}} \in\mathcal{P}(\R^d\times \X)$ such that
	\begin{enumerate}[(i)]
		\item $\bm{\tilde{\eta}}$ is concentrated on the set of points $(x,\gamma)$ such that $\gamma\in AC^p\pa{0,T;\R^d}$ is a solution to the ODE $\gamma'(t)=v(t,\gamma(t))$ for a.e. $t\in (0,T)$ with respect to the Lebesgue measure on $(0,T)$ with $\gamma(0)=x$. 
		\item For any $\phi\in C_b\pa{\R^d}$ and $t\in[0,T]$
		\begin{equation}\nonumber
		\int_{\R^d}\phi(x)\dd\mu_t(x) = \int_{\R^d\times \X}\phi\pa{\gamma(t)}\dd\bm{\tilde{\eta}}(x,\gamma).	
		\end{equation}
	\end{enumerate}
\end{theorem}

\begin{remark}\label{rem:superposition_without_x}
	Given $\bm{\tilde{\eta}}\in \PP\pa{\R^d\times \X}$ as in Theorem \ref{thm:superposition} we can define the push-forward measure $\eta\in\PP\pa{\X}$ by $\eta = \pa{\pi_2}_{\sharp}{\bm{\tilde{\eta}}}$, where $\pi_2$ is the projection on the second component of $\R^d\times \X$. $\eta$ is a probability measure that is concentrated on some Borel set $$\Gamma_v \subseteq \br{\gamma\in\X\;|\;\gamma'(t)=v(t,\gamma(t))}.$$
	
	While this is well known for experts,
	we provide a full proof of this statement in Appendix \ref{app:additional} for the sake of completeness (see Lemma \ref{applem:pushforward_concentration}).
	 
	Note that in the above settings we find that for any $\phi\in C_b\pa{\R^d}$
	$$\int_{\R^d}\phi(x)\dd{e_t}_{\sharp}\eta(x)=\int_{\X}\phi\pa{\gamma(t)}\dd\eta\pa{\gamma} = \int_{\X}\phi\pa{\gamma(t)}\dd{\pi_2}_{\sharp}\bm{\tilde{\eta}}\pa{\gamma} =\int_{\R^d\times \X} \phi\pa{\gamma(t)}\dd\bm{\tilde{\eta}}\pa{x,\gamma} = 	\int_{\R^d}\phi(x)\dd\mu_t(x),$$
	which implies that ${e_t}_{\sharp}\eta = \mu_t$.
\end{remark}

With Theorem \ref{thm:superposition} at hand we conclude this section with the proof of the Benamou--Brenier formula:
\begin{lemma}\label{lem:half_of_BB_hard_side}
	Let $\mu,\nu\in\PP_p\pa{\R^d}$ and let $\D_{p}(\mu,\nu)$ and $\CC_{p}(\mu,\nu)$ be given in \eqref{controlled p action eq} and \eqref{intro:static}, respectively. Then $\D_p\pa{\mu,\nu} \geq \CC_p\pa{\mu,\nu}.$
\end{lemma}

\begin{proof}
	Let $\mu,\nu\in\PP_p\pa{\R^d}$ be given and consider $\pa{\rho,u}\in \cadm\pa{\mu,\nu}$. By the definition of $\cadm\pa{\mu,\nu}$ we have that $v:[0,T]\times\R^{d}\to\R^{d}$ defined as
	$$v(t,x):=M(t)x+N(t)u(t,x)$$
	is a Borel field that satisfies
	\begin{align*}
		\int_{0}^T \pa{\int_{\R^d}\abs{v(t,x)}^p\dd\rho_t}\dd t & \leq 2^{p-1}\norm{M}_{L^\infty\pa{[0,T]}}\int_{0}^T\pa{\int_{\R^d}\abs{x}^p\dd\rho_t(x)}\dd t\\
		&+ 2^{p-1}\norm{N}_{L^\infty\pa{[0,T]}}\int_{0}^T\pa{\int_{\R^d}\abs{u(t,x)}^p\dd\rho_t(x)}\dd t <\infty.	
	\end{align*}
	Consequently, Theorem \ref{thm:superposition} and Remark \ref{rem:superposition_without_x} guarantee that we can find a probability measure $\eta\in \PP\pa{\X}$ that is concentrated on a Borel set $\Gamma_u$ such that 
	$$\Gamma_u  \subseteq \br{\gamma\in\X\;|\;\gamma'(t)=M(t)\gamma(t)+N(t)u(t,\gamma(t))},$$
	and ${e_t}_{\sharp}\eta = \rho_t$.
	
	Considering the product Borel measure $\dd t \otimes \eta$ on $[0,T]\times \X$ and using Fubini's theorem we find that 
	\begin{equation}\label{eq:change_of_integration_for_second_half_of_bb}
		\begin{split}
			\int_{\X}\pa{\int_{0}^T \abs{\amit{u(t,\gamma(t))}}^p\dd t}\dd\eta(\gamma)
            &=\int_{[0,T]\times\X} \abs{u\pa{t,\gamma(t)}}^p\dd t\dd\eta(\gamma) = \int_{0}^T\pa{\int_{\X}\abs{\amit{u(t,e_t(\gamma))}}^p \dd\eta(\gamma)}\dd t\\
			&=\int_{0}^T\pa{\int_{\X}\abs{\amit{u(t,x)}}^p \dd{e_t}_{\sharp}\eta(x)}\dd t=\int_{0}^T\pa{\int_{\X}\abs{\amit{u(t,x)}}^p \dd\rho_t(x)}\dd t<\infty.
		\end{split}
	\end{equation}
	We conclude that there exists a $\eta-$null set, $\mathcal{N}$, such that for every $\gamma\not\in \mathcal{N}$, $u(\cdot,\gamma(\cdot))\in L^p\pa{[0,T];\R^n}$. As such, any $\gamma\in \Gamma_u\setminus \mathcal{N}$ satisfies
	$$\gamma^\prime(t) = M(t)\gamma(t)+N(t)u_\gamma(t)$$
	with $u_\gamma(t):=u(t,\gamma(t))$. By the definition of the cost function $c_p(x,y)$, this implies that
	$$\int_{0}^T \abs{\amit{u(t,\gamma(t))}}^p\dd t \geq c_p\pa{\gamma(0),\gamma(T)},$$
	for any $\gamma\in\Gamma_u\setminus \mathcal{N}$. 
	
	Using the above with \eqref{eq:change_of_integration_for_second_half_of_bb} we find that 
\begin{align}
	\nonumber\int_{0}^T \pa{\int_{\R^d}\abs{u(t,x)}^p\dd\rho_t(x)}\dd t 
    &= \int_{\X}\pa{\int_{0}^T \abs{\amit{u(t,\gamma(t))}}^p\dd t}\dd\eta(\gamma)=\int_{\Gamma_u\setminus \mathcal{N}}\pa{\int_{0}^T \abs{\amit{u(t,\gamma(t))}}^p\dd t}\dd\eta(\gamma)\\
	\label{eq:for_equivalence_later}&\geq \int_{\Gamma_u\setminus \mathcal{N} }c_p\pa{\gamma(0),\gamma(T)}\dd\eta(\gamma)=\int_{\X }c_p\pa{\gamma(0),\gamma(T)}\dd\eta(\gamma)\\
	\nonumber&=\int_{\R^d\times \R^d}c_p(x,y)\dd\EE_{0,T}(\eta)(x,y),
\end{align}
	where $\EE_{0,T}$ was defined in Lemma \ref{lem:connection_between_optimisation_sets}. Moreover, from \eqref{eq:first_margina_of_E_0T} and \eqref{eq:second_margina_of_E_0T} in the proof of the same lemma we know that for any Borel set $A\subset \R^n$
	$$\EE_{0,T}(\eta)\pa{A\times\R^d} ={e_0}_{\sharp}\eta(A) =\rho_0(A)=\mu(A),$$
	$$\EE_{0,T}(\eta)\pa{\R^d\times A} ={e_T}_{\sharp}\eta(A) =\rho_T(A)=\nu(A),$$	
	i.e. $\EE_{0,T}(\eta)\in \Pi\pa{\mu,\nu}$. We conclude that for any $\pa{\rho,u}\in \cadm\pa{\mu,\nu}$
	\begin{equation}\nonumber
		\int_{0}^T \pa{\int_{\R^d}\abs{u(t,x)}^p\dd\rho_t(x)}\dd t \geq  \inf_{\pi\in\Pi\pa{\mu,\nu}}\int_{\R^d\times\R^d}c_p(x,y)\dd\pi(x,y) = \CC_p\pa{\mu,\nu}.
	\end{equation} 
	Taking an infimum over $\cadm\pa{\mu,\nu}$ yield the desired result and completes the proof. 
\end{proof}

\begin{proof}[Proof of Theorem \ref{thm:general_bb}]
	Consider $\pi^\ast\in \Pi\pa{\mu,\nu}$ and $\pa{\rho^\ast,u^\ast}\in \cadm\pa{\mu,\nu}$ as prescribed by Lemma \ref{lem:half_of_BB_easy_side}. Then, with the assistance of Lemma \ref{lem:half_of_BB_hard_side} we find that
	\begin{equation}\nonumber
		\CC_p\pa{\mu,\nu}\leq \D_p\pa{\mu,\nu} \leq \int_{0}^T \pa{\int_{\R^d}\abs{u^\ast(x,t)}^p\dd\rho^{\ast}_t(x)}\dd t \leq \int_{\R^d\times\R^d}c_p(x,y)\dd\pi^\ast(x,y) =  \CC_p\pa{\mu,\nu},
	\end{equation}
	from which the result follows. By the previous chain of inequalities, we see in particular that $\pi^{*}$ is a minimiser for $\CC_p\pa{\mu,\nu}$ while $(\rho^{*},u^{*})$ is a minimiser for $\D_p(\mu,\nu)$. The result follows.
\end{proof}

\begin{remark}\label{rem:recipe}
	We would like to note that not only did we manage to show that the dynamic and static problems are minimised and equivalent -- we have actually managed to find a ``recipe'' that takes a minimiser of one problem to another. Indeed, as was seen in the proof of Lemma \ref{lem:half_of_BB_easy_side}, given a minimiser $\pi^\ast\in \Pi\pa{\mu,\nu}$ for $\CC\pa{\mu,\nu}$, the process described in Remark \ref{rem:constructed_solution_for_generalised_continuity} for the measure  $\eta^\ast = \EE_{0,T}^{-1}\pa{\pi^\ast}$, where we have used the bijectivity of $\EE_{0,T}$, gives us a minimiser $\pa{\rho^\ast,u^\ast}\in\cadm\pa{\mu,\nu}$ for $\D_{p}\pa{\mu,\nu}$. \\
	On the other hand, given a minimiser $\pa{\rho^\ast,u^\ast}\in\cadm\pa{\mu,\nu}$ for $\D_{p}\pa{\mu,\nu}$, the process described in the proof of Lemma \ref{lem:half_of_BB_hard_side} finds $\eta^\ast\in \PP\pa{\X}$ such that $\pi^\ast = \EE_{0,T}\pa{\eta^\ast}\in \Pi\pa{\mu,\nu}$ and
\begin{align*}	
	\CC_{p}\pa{\mu,\nu} & =\D_{p}\pa{\mu,\nu} =\int_{0}^T \pa{\int_{\R^d}\abs{u^\ast(t,x)}^p\dd\rho^\ast_t(x)}\dd t\\
	& \geq \int_{\R^d\times \R^d}c_p(x,y)\dd\EE_{0,T}(\eta)(x,y) = \int_{\R^d\times\R^d}c_p(x,y)\dd\pi^\ast(x,y) \geq \CC_{p}\pa{\mu,\nu},
\end{align*}	
	where we have used \eqref{eq:for_equivalence_later}, giving us a minimiser for $\CC_{p}\pa{\mu,\nu}$.
\end{remark}

\appendix

\section{Additional proofs}\label{app:additional}

In this appendix we provide additional proofs which were omitted from the main body of the work.

\begin{proof}[Proof of Lemma \ref{lem:2_parameter_semigroup}]
	\leavevmode
	\begin{enumerate}[(i)]
		\item Due to the uniqueness of solutions to \eqref{homogeneous particle system} we know that for any $x\in\R^d$
		$$\Phi\pa{s,t}x = \gamma_H^{s,x}(t) = \gamma_{H}^{\tau, \gamma_H^{s,x}(\tau)}(t)=\Phi\pa{\tau,t}\gamma_H^{s,x}(\tau)=\Phi(\tau,t)\circ \Phi\pa{s,\tau}x, $$
		which shows the desired result. 
		\item This is immediate from the definition of $\Phi(s,t)$.
		\item Since for any $x,y\in\R^d$ 
		$$\frac{\dd}{\dd t}\pa{\gamma_H^{s,x}(t)+\gamma_H^{s,y}(t)}=M(t)\pa{\gamma_H^{s,x}(t)+\gamma_H^{s,y}(t)}$$
		and since $\gamma_H^{s,x}(s)+\gamma_H^{s,y}(s)= x+y$, the uniqueness of solution to \eqref{homogeneous particle system} implies that 
		$$\Phi\pa{s,t}\pa{ x+y} = \gamma_H^{s, x+y}(t) = \gamma_H^{s,x}(t)+\gamma_H^{s,y}(t) = \Phi\pa{s,t}x+\Phi\pa{s,t}y.$$
		Similarly, for any $x\in\R^d$ and any $\alpha\in \R$
		$$\frac{\dd}{\dd t}\pa{\alpha \gamma_H^{s, x}(t)}=M(t)\pa{\alpha \gamma_H^{s,x}(t)}$$
		and since $\alpha\gamma_H^{s,x}(s)= \alpha x$ we find that
				$$\Phi\pa{s,t}\pa{\alpha x} = \gamma_H^{s, \alpha x}(t) = \alpha \gamma_H^{s,x}(t) = \alpha\Phi\pa{s,t}x.$$
		\item Any $C^1$ solution $\gamma$ to \eqref{homogeneous particle system} satisfies
		$$\abs{\gamma_H(t)} \leq \abs{\gamma_H(s)} + \int_{\min\br{s,t}}^{\max\br{s,t}} \norm{M(\tau)}\abs{\gamma_H(\tau)}\dd\tau \leq \abs{\gamma_H(s)} + M_1\int_{\min\br{s,t}}^{\max\br{s,t}} \abs{\gamma_H(\tau)}\dd\tau $$
		and consequently 
		$$\abs{\gamma_H^{s,x}(t)} \leq \abs{x} + M_1\int_{\min\br{s,t}}^{\max\br{s,t}} \abs{\gamma_H^{s,x}(\tau)}\dd\tau. $$
		Using Gr\"onwall's inequality we conclude that 
		$$\abs{\Phi(s,t)x} = \abs{\gamma_H^{s,x}(t)} \leq \abs{x}e^{M_1\abs{t-s}}.$$
		As $x\in \R^d$ was arbitrary we achieve the desired result.
		\item The continuity of $\Phi$ follows by its definition and the properties of the homogenous system \eqref{homogeneous particle system}. By definition, we have that for any $x\in\R^d$
		$$\frac{\dd}{\dd t}\Phi(s,t)x = M(t)\Phi(s,t)x.$$
		Identifying $\Phi(s,t)$ with its matrix representation we see that for any $s\in (0,T)$ the matrix $\Phi(s,\cdot)$ is differentiable and satisfies 
		$$\frac{\dd}{\dd t}\Phi(s,t) = M(t)\Phi(s,t).$$
		As inverses of differentiable matrices are differentiable, and as $\Phi(s,t) \Phi(t,s)=\mathbf{I}_{d\times d}$, we conclude that for any $s\in (0,T)$ the matrix $\Phi(t,s)$ is differentiable and
		$$\frac{\dd}{\dd t}\Phi(s,t) \; \Phi(t,s) + \Phi(s,t)\;\frac{\dd}{\dd t}\Phi(t,s)=0,$$
		which implies that for a fixed $t\in (0,T)$ 
		$$\frac{\dd}{\dd s}\Phi(s,t) = - \Phi(s,t)\pa{M(s)\Phi(t,s) }\Phi(s,t) =-\Phi(s,t)M(s).$$
		As $M\in C\pa{\rpa{0,T}}$ we find that $\partial_s\Phi(s,t),\partial_t\Phi(s,t)\in C\pa{[0,T]\times [0,T]}$. Using the fact that $M\in C^{\beta-1}\pa{\pa{t',T}} $ we can continue and differentiate $\beta-1$ times and find that all associated partial derivatives are continuous. The proof is thus complete.
	\end{enumerate}
\end{proof}

We now turn our attention to the proof of Lemma \ref{frechet differentiability of J lemma}. This will require two technical results: the first provides a useful estimate for the difference between $\abs{x}^p$ and $\abs{y}^p$ while the second is a generalisation of the standard dominated convergence theorem. 
\begin{lemma}
	\label{technical estimates on |x|^p-|y|^p - lemma AGS}
	 Let $p>1$. Then, there exist constants $0<c_{p,n},C_{p,n} <\infty$ that only depend on $p$ and $n$ such that for any $x,y\in \R^n$ we have that 
	 \begin{equation}\label{eq:technical_p_leg_2}
	 	0\leq \frac{p-1}{2}\abs{x-y}^2 \min(\abs{x},\abs{y})^{p-2}
	 	\leq \frac{1}{p}\abs{y}^p - \frac{1}{p}\abs{x}^p - \jj_p(x)^\top \pa{y-x}
	 	\leq C_{p,n}\abs{x-y}^p,
	 \end{equation}
	 when $1<p\leq 2$ and 
	 	 \begin{equation}\label{eq:technical_p_geg_2}
	 	0\leq c_{p,n}\abs{x-y}^p
	 	\leq \frac{1}{p}\abs{y}^p - \frac{1}{p}\abs{x}^p - \jj_p(x)^\top \pa{y-x}
	 	\leq \frac{p-1}{2}\abs{x-y}^2 \max(\abs{x},\abs{y})^{p-2},
	 \end{equation}
	 when $p>2$, where we define $\jj_{p}:\R^{d}\to\R^{d}$ as
	 $$\jj_p(x) := \begin{cases}
	 	\abs{x}^{p-2}x,& x\neq 0,\\
	 0,& x=0. 
	 \end{cases}$$
\end{lemma}

\begin{proof}
	A more general version of this lemma and its proof can be found in \cite[Lemma 10.2.1]{AGS08}. 
\end{proof}

\begin{lemma}
	\label{generalised dct lemma}
	Let $(E,\mu)$ be a measure space where $\mu$ is a Borel measure. Let $(f_n)_{n\in\N}$ and $(g_n)_{n\in\N}$, $f_{n},g_{n}: E\to\R$ be two sequences of measurable functions that converge pointwise $\mu-$a.e. to $f:E\to\R$ and $g:E\to\R$, respectively. Assume in addition that 
	$$\abs{f_n} \leq g_n,\qquad \int_{E}g_n\dd\mu \underset{n\to\infty}{\longrightarrow}\int_E g \dd\mu <\infty.$$
	Then 
	$$\int_{E}f_n\dd\mu \underset{n\to\infty}{\longrightarrow}\int_E f \dd\mu.$$
\end{lemma}

\begin{proof}
	The proof is a straightforward application of Fatou's lemma for $g_n-f_n$ and $g_n+f_n$.
\end{proof}

\begin{proof}[Proof of Lemma \ref{frechet differentiability of J lemma}]
	As a first step, we will show the Fr\'echet differentiability of $J$. For any $\alpha,u\in L^p\pa{0,T;\R^n}$ we have that  
\begin{align*}
	\Bigg{|}J\pa{\alpha+u} - J(\alpha) &- p \int_0^T \mathfrak{j}_p(\alpha(t))^\top u(t)\dd t\Bigg{|} \leq \int_{0}^T\abs{\abs{\alpha(t)+u(t)}^p - \abs{\alpha(t)}^p -p \mathfrak{j}_p(\alpha(t))\cdot u(t) }\dd t\\
	&\leq \begin{cases}
		pC_{p,n}\int_{0}^T \abs{u(t)}^p \dd t, & 1<p\leq 2,\\
		\frac{p(p-1)}{2}\int_{0}^T \abs{u(t)}^2\max\br{\abs{\alpha(t)},\abs{\alpha(t)+u(t)}}^{p-2}\dd t,& p>2.
	\end{cases}\\
	&\leq \begin{cases}
		pC_{p,n}\norm{u}^p_{L^p\pa{0,T;\R^n}}, & 1<p\leq 2,\\
		\frac{p(p-1)}{2}\;2^{p-2}\pa{\norm{u}^p_{L^p\pa{0,T;\R^n}}+\int_{0}^T \abs{u(t)}^2\abs{\alpha(t)}^{p-2}},& p>2,
	\end{cases}.
\end{align*}
	where we have used Lemma \ref{technical estimates on |x|^p-|y|^p - lemma AGS}. We conclude that 
	\begin{align*}
			&\displaystyle\frac{\abs{J\pa{\alpha+u} - J(\alpha) -\displaystyle p \int_0^T \mathfrak{j}_p(\alpha(t))^T u(t)\dd t}}{\norm{u}_{L^p\pa{0,T;\R^n}}}\\
			&\leq \begin{cases}
				pC_{p,n}\norm{u}^{p-1}_{L^p\pa{0,T;\R^n}}, & 1<p\leq 2,\\[5pt]
				2^{p-3}p(p-1)\pa{\norm{u}^{p-1}_{L^p\pa{0,T;\R^n}}+\norm{u}_{L^p\pa{0,T;\R^n}}\norm{\alpha}^{p-2}_{L^p\pa{0,T;\R^n}}},& p>2.
			\end{cases}
	\end{align*}
	where we have used the fact that the H\"older conjugate of $\frac{p}{2}$ is $\frac{p}{p-2}$ when $p>2$. Consequently 
	 $$\frac{\abs{J\pa{\alpha+u} - J(\alpha) - D_\alpha J(u)}}{\norm{u}_{L^p\pa{0,T;\R^n}}}\underset{u\to 0}{\longrightarrow}0,$$
	 which shows the differentiability of $J$ at any $\alpha\in L^p\pa{0,T;\R^n}$.
	 
	To show that the Fr\'echet derivative of $J$ is continuous, it is enough to show that the map $\jj_p$ is continuous. Indeed, since
	$$\abs{D_{\alpha}J\pa{u}-D_{\beta}J(u)}\leq p\norm{\jj_p(\alpha)-\jj_p(\beta)}_{L^q\pa{0,T;\R^n}}\norm{u}_{L^p\pa{0,T;\R^n}}$$
	we have that 
	$$\norm{D_{\alpha}J-D_{\beta}J}\leq p\norm{\jj_p(\alpha)-\jj_p(\beta)}_{L^q\pa{0,T;\R^n}}$$
	which shows that the modulus of continuity of the Fr\'echet derivative is controlled by that of $\jj_p$. We thus focus on showing the continuity of $\jj_p$.
	
	Since $L^q\pa{0,T;\R^n}$ is a normed space, and as such a metric space, to show that $\jj_p$ is continuous it is enough to show that if $(\alpha_m)_{m\in\N}$ converges to $\alpha$ in $L^p\pa{0,T;\R^n}$, then for any subsequence of $\br{\alpha_m}_{m\in\N}$, $(\alpha_{m_k})_{k\in\N}$ there exists a subsequence, $(\alpha_{m_{k_j}})_{j\in\N}$, on which 
	$$\norm{\jj_p\pa{\alpha_{m_{k_j}}}-\jj_p\pa{\alpha}}_{L^q\pa{0,T;\R^n}}\underset{j\to\infty}{\longrightarrow}0.$$ 
	
	Given a subsequence of $(\alpha_m)_{m\in\N}$, $(\alpha_{m_k})_{k\in\N}$, we can extract a subsequence of it, $\left(\alpha_{m_{k_j}}\right)_{j\in\N}$, that converges pointwise a.e. to $\alpha$. Denoting by 
	$$\sign(x):=\begin{cases}
		\frac{x}{\abs{x}}, & x\neq 0,\\
		0, & x=0,
	\end{cases}$$
	we find that $\jj_p\pa{\alpha} = \abs{\alpha}^{p-1}\sign(\alpha)$ and
	$$\abs{\jj_p\pa{\alpha_{m_{k_j}}(t)}-\jj_p\pa{\alpha(t)}} \leq \abs{\abs{\alpha_{m_{k_j}}(t)}^{p-1}-\abs{\alpha(t)}^{p-1}} +\abs{\alpha(t)}^{p-1}\abs{\sign\pa{\alpha_{m_{k_j}}(t)}-\sign\pa{\alpha(t)}}. $$
	Denoting by $f_j(t)=\abs{\jj_p\pa{\alpha_{m_{k_j}}(t)}-\jj_p\pa{\alpha(t)}}^q $ we get that 
	$$f_j(t)\leq 2^q\pa{\abs{\abs{\alpha_{m_{k_j}}(t)}^{p-1}-\abs{\alpha(t)}^{p-1}}^q +\abs{\alpha(t)}^{p}\abs{\sign\pa{\alpha_{m_{k_j}}(t)}-\sign\pa{\alpha(t)}}^q}. $$
	Furthermore, if we denote by $\mathcal{Z}:=\br{t\in [0,T]\;|\;\alpha(t)=0}$, we see that 
	$$f_j(t)=f_j(t)\chi_{\mathcal{Z}}(t)+f_j(t)\chi_{\mathcal{Z}^c}(t) \underset{j\to\infty}{\longrightarrow}0$$
	pointwise a.e., where we have used the fact that 
	$$ 0\leq f_j(t)\chi_{\mathcal{Z}}(t) \leq 2^q\pa{\abs{\abs{\alpha_{m_{k_j}}(t)}^{p-1}-\abs{\alpha(t)}^{p-1}}^q}\underset{j\to\infty}{\longrightarrow}0$$
	pointwise a.e., and the fact that on $\mathcal{Z}^c$ 
	$$\sign\pa{\alpha_{m_{k_j}}(t)}\underset{j\to\infty}{\longrightarrow}\sign\pa{\alpha(t)},$$
	pointwise a.e.
	
	Moreover, denoting by
	$$g_j(t) = 4^q\pa{\abs{\alpha_{m_{k_j}}(t)}^{p}+2\abs{\alpha(t)}^{p}}, \quad g(t)=3\cdot 4^q\abs{\alpha(t)}^{p},$$
	  we find that
	  $$\abs{f_j(t)}=f_j(t) \leq g_j(t),\quad g_j(t)\underset{j\to\infty}{\longrightarrow} g(t)\text{ pointwise a.e. and}\quad \int_{0}^T g_j(t)\dd t\underset{j\to\infty}{\longrightarrow}\int_0^T g(t)\dd t, $$
	  where we have used the fact that $\left(\alpha_{m_{k_j}}\right)_{j\in\N}$ converges to $\alpha$ in $L^p\pa{0,T;\R^n}$. Using the generalised dominated convergence theorem from Lemma \ref{generalised dct lemma} we find that 
	  $$\norm{\jj_p\pa{\alpha_{m_{k_j}}}-\jj_p\pa{\alpha}}_{L^q\pa{0,T;\R^n}}^q \leq \int_{0}^T f_j(t)\dd t\underset{j\to\infty}{\longrightarrow} \int_{0}^T 0 \dd t=0,$$
	  which concludes the proof.  
\end{proof}

\begin{lemma}\label{applem:missing_measurability}
	Under the same notations as in the proof of Theorem \ref{thm:not_empty} we have that for any non-negative Borel function $\Psi : [0,T]\times \R^d\to [0,+\infty)$, the function $[0,T]\ni s\mapsto \int_{\R^d}\Psi(s,x)\dd{e_s}_{\sharp}\eta(x)$ is Borel measurable and  
	\begin{equation}\label{eqapp:integration_of_e_s_vs_e_sharp_eta}
		\frac{1}{T}\int_{0}^T \pa{\int_{\R^d}\Psi(s,x)\dd{e_s}_{\sharp}\eta(x)}\dd s =\int_{[0,T]\times \R^d}\Psi(s,x) \dd\pa{\bm{e}_{\sharp}\bm{\eta}}(s,x).
	\end{equation} 
\end{lemma}

\begin{proof}
	We have seen in the proof of Theorem \ref{thm:not_empty} in the main text that the statement of this lemma holds when $\Psi$ is a bounded and continuous function,  so it remains to consider the general case, when $\Psi$ is non-negative and measurable. Let $K\subset [0,T]\times \R^d$ be compact. Defining
	$$U_n:=\br{\pa{s,x}\in [0,T]\times \R^d\;|\; {\rm{dist}}\pa{\pa{s,x},K} <\frac{1}{n}}$$
	we find that $\pa{U_n}_{n\in\N}$ is a sequence of nested (i.e. $U_{n+1}\subseteq U_{n},\ \forall n\in\N$) open sets such that 
	$$K=\cap_{n\in\N}U_n.$$
	We recall that in any metric space the function 
	$$f_{A,B}(x) := \frac{{\rm{dist}}(x,B^c)}{{\rm{dist}}(x,A)+{\rm{dist}}(x,B^c)}$$
	is continuous whenever $A$ is a closed set, $B$ is an open set, and $A \subset B$. Moreover,
	$$0\leq f_{A,B}\leq 1,\qquad f_{A,B}\vert_{A}=1,\qquad f_{A,B}\vert_{B^c}=0.$$
	For the metric space $[0,T]\times\R^{d}$, whose elements we denote as $(s,x)$, defining $f_n:[0,T]\times\R^{d}\to[0,1]$ as
	$$f_n\pa{s,x} := f_{K,U_n}\pa{s,x}$$
	gives us a sequence of bounded continuous functions that satisfy $f_n\vert_{K}=1$ and $f_n\vert_{U_n^c}=0$. In addition, as $\pa{U_n}_{n\in\N}$ is decreasing, so is $\pa{f_n}_{n\in\N}$ (i.e. $f_{n+1}(s,x)\le f_{n}(s,x)$ for all $n\in\N$ and for all $(s,x)\in[0,T]\times\R^{d}$). Since 
	$$\lim_{n\to\infty}f_n\pa{s,x} = \chi_{K}\pa{s,x}$$
	we find, using the monotone convergence theorem for finite measures, that 
	$$\int_{\R^d}\chi_K(s,x)\dd{e_s}_{\sharp}\eta(x) = \lim_{n\to\infty}\int_{\R^d}f_n(s,x)\dd{e_s}_{\sharp}\eta(x),$$
	which implies that $[0,T]\ni s\mapsto \int_{\R^d}\chi_K(s,x)\dd{e_s}_{\sharp}\eta(x)$ is Borel measurable as a limit of such functions. Moreover, using the monotone convergence theorem again we see that 
\begin{align*}	
	\frac{1}{T}\int_{0}^T \pa{\int_{\R^d}\chi_K(s,x)\dd{e_s}_{\sharp}\eta(x)}\dd s & =\frac{1}{T}\int_{0}^T\pa{\lim_{n\to\infty}\int_{\R^d}f_n(s,x)\dd{e_s}_{\sharp}\eta(x)}\dd s\\
	&=\lim_{n\to\infty}\frac{1}{T}\int_{0}^T\pa{\int_{\R^d}f_n(s,x)\dd{e_s}_{\sharp}\eta(x)}\dd s\\
	&=\lim_{n\to\infty}\int_{[0,T]\times \R^d}f_n(s,x) \dd\bm{e}_{\sharp}\bm{\eta}(s,x)=\int_{[0,T]\times \R^d}\chi_K(s,t) \dd\bm{e}_{\sharp}\bm{\eta}(s,x).
\end{align*}
	Similarly, for any bounded open set $U\subset [0,T]\times \R^d$ we can find a sequence of increasing compact sets $\pa{K_n}_{n\in\N}$ such that $\cup_{n\in\N}K_n=U$\footnote{Indeed, the sets $K_n = \br{\pa{s,x}\in U\;|\; {\rm{dist}}\pa{\pa{s,x},\partial U} \geq \frac{1}{n}}$ are closed and bounded sets that satisfy 
	$$K_{n}\subseteq K_{n+1},\qquad \text{and}\qquad \cup_{n\in N}K_n=U.$$}. Defining $f_n(s,x)=\chi_{K_n}(s,x)$ gives us an increasing sequence of bounded functions that converges pointwise to $\chi_U(s,x)$.  Using the monotone convergence theorem we conclude that 
	$$[0,T]\ni s\mapsto\int_{\R^d}\chi_U(s,x)\dd{e_s}_{\sharp}\eta(x) = \lim_{n\to\infty}\int_{\R^d}f_n(s,x)\dd{e_s}_{\sharp}\eta(x),$$
	is Borel measurable with respect to $s$ and 
\begin{align*}	
	\frac{1}{T}\int_{0}^T \pa{\int_{\R^d}\chi_U(s,x)\dd{e_s}_{\sharp}\eta(x)}\dd s &=\frac{1}{T}\int_{0}^T\pa{\lim_{n\to\infty}\int_{\R^d}f_n(s,x)\dd{e_s}_{\sharp}\eta(x)}\dd s\\
	&=\lim_{n\to\infty}\frac{1}{T}\int_{0}^T\pa{\int_{\R^d}f_n(s,x)\dd{e_s}_{\sharp}\eta(x)}\dd s\\
	&=\lim_{n\to\infty}\int_{[0,T]\times \R^d}f_n(s,x) \dd\bm{e}_{\sharp}\bm{\eta}(s,x)=\int_{[0,T]\times \R^d}\chi_U(s,t) \dd\bm{e}_{\sharp}\bm{\eta}(s,x).
\end{align*}
	If $U$ is an unbounded open set then we can build on the above by considering the increasing sequence of open sets $U_n=U\cap B_n(0)$ and repeating this process. 
	
	Next we consider a Borel set $B\subset [0,T]\times \R^d$. Due to the regularity of $\bm{e}_{\sharp}\bm{\eta}$ we can find a sequence of compact sets $\pa{K_n}_{n\in\N}\subset [0,T]\times \R^d$ and open sets $\pa{U_n}_{n\in\N}$ in $[0,T]\times \R^d$ such that 
	$$K_n\subset B \subset U_n,\;\; \forall n\in\N,\qquad \bm{e}_{\sharp}\bm{\eta}\pa{U_n\setminus K_n}\underset{n\to\infty}{\longrightarrow}0.$$
	Moreover, we can assume without loss of generality that the sequence $\pa{K_n}_{n\in\N}$ is increasing while the sequence $\pa{U_n}_{n\in\N}$ is decreasing. We conclude that the functions $g_{1},g_{2}:[0,T]\to [0,\infty)$ defined as
	$$g_1(s) := \lim_{n\to\infty}\int_{\R^d}\chi_{K_n}(s,x)\dd{e_s}_{\sharp}\eta(x) = \int_{\R^d}\chi_{\cup_{n\in\N}K_n}(s,x)\dd{e_s}_{\sharp}\eta(x)$$
	and
	$$g_2(s) :=\lim_{n\to\infty}\int_{\R^d}\chi_{U_n}(s,x)\dd{e_s}_{\sharp}\eta(x)=\int_{\R^d}\chi_{\cap_{n\in\N}U_n}(s,x)\dd{e_s}_{\sharp}\eta(x)$$
	are well defined and Borel measurable. Moreover, $g_1\leq g_2$ and using the monotone convergence theorem and what we have shown so far we see that 
\begin{align*}	
	0&\leq \frac{1}{T}\int_{0}^T \pa{g_2(s)-g_1(s)}\dd s = \lim_{n\to\infty}\frac{1}{T}\int_{0}^T\pa{\int_{\R^d}\pa{\chi_{U_n}(s,x)-\chi_{K_n}(s,x)}\dd{e_s}_{\sharp}{\eta}(x)} \dd s\\
	&=\lim_{n\to\infty}\int_{[0,T]\times \R^d}\pa{\chi_{U_n}(s,x)-\chi_{K_n}(s,x)}\dd\pa{\bm{e}_{\sharp}\bm{\eta}}(s,x) = \lim_{n\to\infty}(\bm{e}_{\sharp}\bm{\eta})\pa{U_n\setminus K_n}=0.
\end{align*}
	Consequently there exists a Borel null set $N$ such that $g_1(s)=g_2(s)=:g(s)$ for all $s\not\in N$.
	
	Since for any $n\in\N$\footnote{note that for a fixed $s$ the $s-$section of $B$ is Borel measurable so $\int_{\R^d}\chi_{B}(s,x)\dd{e_s}_{\sharp}\eta(x)$ is well defined.}
	$$ \int_{\R^d}\chi_{K_n}(s,x)\dd{e_s}_{\sharp}\eta(x)\leq \int_{\R^d}\chi_{B}(s,x)\dd{e_s}_{\sharp}\eta(x)\leq \int_{\R^d}\chi_{U_n}(s,x)\dd{e_s}_{\sharp}\eta(x)$$
	we conclude that for $s\not\in N$
	$$g(s)=\int_{\R^d}\chi_{B}(s,x)\dd{e_s}_{\sharp}\eta(x)$$
	which shows that (since $N$ is a null Borel set) $[0,T]\ni s\mapsto\int_{\R^d}\chi_{B}(s,x)\dd{e_s}_{\sharp}\eta(x)$ is Borel measurable. Moreover, using the monotone convergence theorem again
\begin{align*}
	\frac{1}{T}\int_{0}^T\pa{\int_{\R^d}\chi_{B}(s,x)\dd{e_s}_{\sharp}\eta(x)}\dd s &= \frac{1}{T}\int_0^T\pa{\lim_{n\to\infty}\int_{\R^d}\chi_{K_n}(s,x)\dd{e_s}_{\sharp}\eta(x)}\dd s\\
	&=\lim_{n\to\infty}\frac{1}{T}\int_{0}^T\pa{\int_{\R^d}\chi_{K_n}(s,x)\dd{e_s}_{\sharp}\eta(x)}\dd s\\
	&=\lim_{n\to\infty}\int_{[0,T]\times \R^d}\chi_{K_n}(s,x)\dd\bm{e}_{\sharp}\bm{\eta} = \int_{[0,T]\times \R^d}\chi_{B}(s,x)\dd\bm{e}_{\sharp}\bm{\eta},
\end{align*}
	since $\pa{\chi_{K_n}}_{n\in\N}$ converges to $\chi_B$ in $L^1\pa{\dd\bm{e}_{\sharp}\bm{\eta}}$.
	
	We have shown at this point, that for any Borel set $B\subset [0,T]\times \R^d$ the function $[0,T]\ni s\mapsto\int_{\R^d}\chi_{B}(s,x)\dd{e_s}_{\sharp}\eta(x)$ is Borel measurable and 
	$$\frac{1}{T}\int_0^T \pa{\int_{\R^d}\chi_{B}(s,x)\dd{e_s}_{\sharp}\eta(x)}\dd s = \int_{[0,T]\times \R^d}\chi_{B}(s,x)\dd\bm{e}_{\sharp}\bm{\eta}.$$
	Consequently the above holds for any simple functions and since for any non-negative Borel function $\Psi$ we can find an increasing sequence of non-negative simple functions $\pa{h_n}_{n\in\N}$ that converges to $\Psi$ pointwise, using the monotone convergence theorem yet again we conclude that 
	$$[0,T]\ni s \mapsto\int_{\R^d}\Psi(s,x)\dd{e_s}_{\sharp}\eta(x) = \lim_{n\to\infty}\int_{\R^d}h_n(s,x)\dd{e_s}_{\sharp}\eta(x)$$
	is Borel measurable with respect and 
\begin{align*}	
	\frac{1}{T}\int_{0}^T \pa{\int_{\R^d}\Psi(s,x)\dd{e_s}_{\sharp}\eta(x) }\dd s & = \lim_{n\to\infty}\frac{1}{T}\int_0^T \pa{\int_{\R^d}h_n(s,x)\dd{e_s}_{\sharp}\eta(x) }\dd s\\
	&=\lim_{n\to\infty}\int_{[0,T]\times \R^d}h_n(s,x)\dd\bm{e}_{\sharp}\bm{\eta}(s,x)=\int_{[0,T]\times \R^d}\Psi(s,x)\dd\pa{\bm{e}_{\sharp}\bm{\eta}}(s,x).
\end{align*}
	The proof is thus complete.
\end{proof}

\begin{lemma}\label{applem:pushforward_concentration}
	Let $X$ and $Y$ be Polish spaces such that $X=\cup_{n\in\N}K_n$ with $\pa{K_n}_{n\in\N}$ being an increasing sequence of compact sets. Let $\eta\in \PP\pa{X\times Y}$ be given and define $\mu\in\PP\pa{Y}$ by 
	$$\mu=\pa{\pi_2}_{\sharp}\eta,$$
	where $\pi_2:X\times Y\to Y$ is given by $\pi_2\pa{x,y}=y$. Then there exists a Borel set $S$ on which $\mu$ is concentrated such that 
	$$S\subseteq \pi_2\pa{\mathrm{spt}\pa{\eta}}.$$
\end{lemma}

\begin{proof}
	We start by mentioning that if $Z$ is a Polish space and $\mu_1,\mu_2\in \PP\pa{Z}$ satisfy $\mu_1\leq \mu_2$  then $${\rm{spt}}\pa{\mu_1}\subseteq {\rm{spt}}\pa{\mu_2}\footnote{Indeed, for any $x\in {\rm{spt}}\pa{\mu_1}$ and any open set containing $x$, $U$, we have that $\mu_2\pa{U} \geq \mu_1\pa{U}>0$, which shows that $x\in {\rm{spt}}\pa{\mu_2}$.}.$$
	
	For any $n\in\N$ we define the Borel measure $\eta_n$ by 
	$$\eta_n(\A) = \eta\pa{\pa{K_n\times Y}\cap \A}$$
	for any Borel set $\A$ in $X\times Y$ and let $\mu_n:=\pa{\pi_2}_{\sharp}\eta_n$. We note the following:
	\begin{equation}\label{eqapp:monotonicity_of_eta_n}
		\eta_{n}\leq \eta_{n+1},\qquad \eta_n \leq \eta,\quad \forall n\in\N,	
	\end{equation}
	which implies that 
	\begin{equation}\label{eqapp:monotonicity_of_mu_n}
		\mu_n\leq \mu_{n+1},\qquad \mu_n\leq \mu,\quad \forall n\in\N.
	\end{equation}
	As a first step we claim that ${\rm{spt}}\pa{\mu_n}\subseteq \pi_2\pa{{\rm{spt}}\pa{\eta_n}}$. Indeed, assume that there exists $y\in {\rm{spt}}\pa{\mu_n}$ such that $y\not\in \pi_2\pa{{\rm{spt}}\pa{\eta_n}}$. Then, for any $x\in X$ we have that there exists an open set $U_{x,y}$ in $X\times Y$ with 
	$$\eta_n\pa{U_{x,y}}=0.$$
	This implies that we can find $r(x),\rho(x)>0$ such that $B_{r(x)}(x)\times B_{\rho(x)}(y)\subseteq U_{x,y}$ and consequently. 
	$$\eta_n\pa{B_{r(x)}(x)\times B_{\rho(x)}(y)}=0.$$
	The collection $\pa{B_{r(x)}(x)\times B_{\rho(x)}(y)}_{x\in K_n}$ is an open cover for $K_n\times \br{y}$ which is compact in $X\times Y$. Consequently, we can find $\br{x_j}_{j=1,\dots,m}\subset K_n$ and $\br{r_i,\rho_i}_{i=1,\dots,m}\in (0,\infty)$ such that 
	$$K_n\times \br{y}\subset \cup_{i=1}^m B_{r_i}(x_i)\times B_{\rho_i}(y).$$
	Denoting by $\rho=\min_{i=1,\dots,m}\rho_i>0$ we see that 
	$$K_n\times \br{y}\subset  \pa{\cup_{i=1}^m B_{r_i}(x_i)}\times B_{\rho}(y)\subset \cup_{i=1}^m B_{r_i}(x_i)\times B_{\rho_i}(y)$$
	and conclude that 
\begin{align*}	
	\mu_n\pa{B_{\rho}(y)} &= \eta_n\pa{\pi_2^{-1}\pa{B_{\rho}(y)}}=\eta_n\pa{X\times B_{\rho}(y)}\\
		&=\eta\pa{\pa{K_n\times Y}\cap \pa{X\times B_{\rho}(y)}}=\eta\pa{K_n\times B_{\rho}(y)}=\eta_n\pa{K_n\times B_{\rho}(y)}\\
		&\leq \eta_n\pa{\cup_{i=1}^m B_{r_i}(x_i)\times B_{\rho_i}(y)} \leq \sum_{i=1}^m \eta_n\pa{B_{r_i}(x_i)\times B_{\rho_i}(y)}=0,
\end{align*}
	which contradicts the fact that $y\in {\rm{spt}}\pa{\mu_n}$. 
	
	From the above discussion and \eqref{eqapp:monotonicity_of_eta_n} we find that
	$${\rm{spt}}\pa{\mu_n}\subseteq \pi_2\pa{{\rm{spt}}\pa{\eta_n}} \subseteq \pi_2\pa{{\rm{spt}}\pa{\eta}}.$$
	From \eqref{eqapp:monotonicity_of_mu_n} we see that $\pa{{\rm{spt}}\pa{\mu_n}}_{n\in\N}$ is an non-decreasing sequence of closed sets and as such $S=\cup_{n\in\N}{\rm{spt}}\pa{\mu_n}$ is a Borel set which satisfies
	$$S\subseteq \pi_2\pa{{\rm{spt}}\pa{\eta}}.$$
	Lastly, we notice that
	$$\mu(S) \geq \mu_n(S)\geq \mu_n({\rm{spt}}\pa{\mu_n})=\mu_n\pa{Y} = \eta_n\pa{X\times Y}=\eta\pa{K_n\times Y}.$$
	Since $\lim_{n\to\infty}\eta\pa{K_n\times Y} = \eta\pa{X\times Y}=1$ we find that $\mu(S)=1$ which concludes the proof. 
\end{proof}

\bibliographystyle{alpha}
\bibliography{reference}

\begin{thebibliography}{AGPM25}

\bibitem[AGPM25]{AmbGriMes:25}
D.M. Ambrose, M.~Griffin-Pickering, and A.R. M\'esz\'aros.
\newblock Kinetic-type mean field games with non-separable local hamiltonians.
\newblock {\em J. Lond. Math. Soc. (2)}, to appear, 2025.

\bibitem[AGS08]{AGS08}
L.~Ambrosio, N.~Gigli, and G.~Savar\'e.
\newblock {\em Gradient flows in metric spaces and in the space of probability
  measures}.
\newblock Lectures in Mathematics ETH Z\"urich. Birkh\"auser Verlag, Basel,
  second edition, 2008.

\bibitem[AL09]{AgrLee:09}
A.~Agrachev and P.~Lee.
\newblock Optimal transportation under nonholonomic constraints.
\newblock {\em Trans. Amer. Math. Soc.}, 361(11):6019--6047, 2009.

\bibitem[AR04]{Ambrosio04}
L.~Ambrosio and S.~Rigot.
\newblock Optimal mass transportation in the {H}eisenberg group.
\newblock {\em J. Funct. Anal.}, 208(2):261--301, 2004.

\bibitem[BB00]{BB00}
J.-D. Benamou and Y.~Brenier.
\newblock A computational fluid mechanics solution to the {M}onge-{K}antorovich
  mass transfer problem.
\newblock {\em Numer. Math.}, 84(3):375--393, 2000.

\bibitem[BB06]{Bernard06}
P.~Bernard and B.~Buffoni.
\newblock The {M}onge problem for supercritical {M}a\~n\'e{} potentials on
  compact manifolds.
\newblock {\em Adv. Math.}, 207(2):691--706, 2006.

\bibitem[BB07]{Bernard07}
P.~Bernard and B.~Buffoni.
\newblock Optimal mass transportation and {M}ather theory.
\newblock {\em J. Eur. Math. Soc. (JEMS)}, 9(1):85--121, 2007.

\bibitem[Ber08]{Bernard08}
P.~Bernard.
\newblock Young measures, superposition and transport.
\newblock {\em Indiana Univ. Math. J.}, 57(1):247--275, 2008.

\bibitem[BF21]{BonFra:21}
B.~Bonnet and H.~Frankowska.
\newblock Necessary optimality conditions for optimal control problems in
  {W}asserstein spaces.
\newblock {\em Appl. Math. Optim.}, 84:S1281--S1330, 2021.

\bibitem[BF22]{BonFra:22}
B.~Bonnet and H.~Frankowska.
\newblock Semiconcavity and sensitivity analysis in mean-field optimal control
  and applications.
\newblock {\em J. Math. Pures Appl. (9)}, 157:282--345, 2022.

\bibitem[BMQ25]{BriMaaQua}
G.~Brigati, J.~Maas, and F.~Quattrocchi.
\newblock Kinetic optimal transport ({OTIKIN}) -- {P}art 1: Second-order
  discrepancies between probability measures.
\newblock {\em arXiv:2502.15665}, 2025.

\bibitem[BQ20]{BivQui:20}
M.~Bivas and M.~Quincampoix.
\newblock Optimal control for the evolution of deterministic multi-agent
  systems.
\newblock {\em J. Differential Equations}, 269(3):2228--2263, 2020.

\bibitem[BR21]{BonRos:21}
B.~Bonnet and F.~Rossi.
\newblock Intrinsic {L}ipschitz regularity of mean-field optimal controls.
\newblock {\em SIAM J. Control Optim.}, 59(3):2011--2046, 2021.

\bibitem[CD18]{CDF2018}
R.~Carmona and F.~Delarue.
\newblock {\em Probabilistic theory of mean field games with applications.
  {I}}, volume~83 of {\em Probability Theory and Stochastic Modelling}.
\newblock Springer, Cham, 2018.
\newblock Mean field FBSDEs, control, and games.

\bibitem[CEL25]{CraElaLee:25}
K.~Craig, K.~Elamvazhuthi, and H.~Lee.
\newblock A blob method for mean field control with terminal constraints.
\newblock {\em ESAIM Control Optim. Calc. Var.}, 31:Paper No. 20, 46, 2025.

\bibitem[CGP17]{CGP17}
Y.~Chen, T.T. Georgiou, and M.~Pavon.
\newblock Optimal transport over a linear dynamical system.
\newblock {\em IEEE Trans. Automat. Control}, 62(5):2137--2152, 2017.

\bibitem[CLOS22]{CLOS22}
G.~Cavagnari, S.~Lisini, C.~Orrieri, and G.~Savar\'e.
\newblock Lagrangian, {E}ulerian and {K}antorovich formulations of multi-agent
  optimal control problems: equivalence and gamma-convergence.
\newblock {\em J. Differential Equations}, 322:268--364, 2022.

\bibitem[CMNP18]{CM18}
G.~Cavagnari, A.~Marigonda, K.T. Nguyen, and F.S. Priuli.
\newblock Generalized control systems in the space of probability measures.
\newblock {\em Set-Valued Var. Anal.}, 26(3):663--691, 2018.

\bibitem[DPGG06]{DePascale06}
L.~De~Pascale, M.~S. Gelli, and L.~Granieri.
\newblock Minimal measures, one-dimensional currents and the
  {M}onge-{K}antorovich problem.
\newblock {\em Calc. Var. Partial Differential Equations}, 27(1):1--23, 2006.

\bibitem[EJ25]{ElaJac:25}
K.~Elamvazhuthi and M.~Jacobs.
\newblock Optimal transport of linear systems over equilibrium measures.
\newblock {\em Automatica J. IFAC}, 175:Paper No. 112222, 7, 2025.

\bibitem[Ela25]{EL24}
K.~Elamvazhuthi.
\newblock {B}enamou--{B}renier formulation of optimal transport for nonlinear
  control systems on $\mathbb{R}^{d}$.
\newblock {\em arXiv:2407.16088v4}, 2025.

\bibitem[ELLO23]{ELO23}
K.~Elamvazhuthi, S.~Liu, W.~Li, and S.~Osher.
\newblock Dynamical optimal transport of nonlinear control-affine systems.
\newblock {\em J. Comput. Dyn.}, 10(4):425--449, 2023.

\bibitem[FJ08]{Figalli08}
A.~Figalli and N.~Juillet.
\newblock Absolute continuity of {W}asserstein geodesics in the {H}eisenberg
  group.
\newblock {\em J. Funct. Anal.}, 255(1):133--141, 2008.

\bibitem[FLOS19]{FLOS2019}
M.~Fornasier, S.~Lisini, C.~Orrieri, and G.~Savar\'e.
\newblock Mean-field optimal control as gamma-limit of finite agent controls.
\newblock {\em European J. Appl. Math.}, 30(6):1153--1186, 2019.

\bibitem[FR10]{FigRif:10}
A.~Figalli and L.~Rifford.
\newblock Mass transportation on sub-{R}iemannian manifolds.
\newblock {\em Geom. Funct. Anal.}, 20(1):124--159, 2010.

\bibitem[GPM22]{GriMes:22}
M.~Griffin-Pickering and A.R. M\'esz\'aros.
\newblock A variational approach to first order kinetic mean field games with
  local couplings.
\newblock {\em Comm. Partial Differential Equations}, 47(10):1945--2022, 2022.

\bibitem[HPR11]{HPR11}
A.~Hindawi, J.-B. Pomet, and L.~Rifford.
\newblock Mass transportation with {LQ} cost functions.
\newblock {\em Acta Appl. Math.}, 113(2):215--229, 2011.

\bibitem[Iac22]{Iac}
M.~Iacobelli.
\newblock A new perspective on {W}asserstein distances for kinetic problems.
\newblock {\em Arch. Ration. Mech. Anal.}, 244(1):27--50, 2022.

\bibitem[IJ24]{IacJun}
M.~Iacobelli and J.~Junn\'e.
\newblock Stability estimates for the {V}lasov-{P}oisson system in
  {$p$}-kinetic {W}asserstein distances.
\newblock {\em Bull. Lond. Math. Soc.}, 56(7):2250--2267, 2024.

\bibitem[Jim08]{Jimenez08}
C.~Jimenez.
\newblock Dynamic formulation of optimal transport problems.
\newblock {\em J. Convex Anal.}, 15(3):593--622, 2008.

\bibitem[JMQ20]{JimMarQui:20}
C.~Jimenez, A.~Marigonda, and M.~Quincampoix.
\newblock Optimal control of multiagent systems in the {W}asserstein space.
\newblock {\em Calc. Var. Partial Differential Equations}, 59(2):Paper No. 58,
  45, 2020.

\bibitem[Kle20]{Klenke2020}
A.~Klenke.
\newblock {\em Probability theory---a comprehensive course}.
\newblock Universitext. Springer, Cham, third edition, [2020] \copyright 2020.

\bibitem[Par25]{Par}
S.~Park.
\newblock A variational perspective on the dissipative {H}amiltonian structure
  of the {V}lasov--{F}okker--{P}lanck equation.
\newblock {\em arXiv:2406.13682v2}, 2025.

\bibitem[Pra05]{Pratelli05}
A.~Pratelli.
\newblock Equivalence between some definitions for the optimal mass transport
  problem and for the transport density on manifolds.
\newblock {\em Ann. Mat. Pura Appl. (4)}, 184(2):215--238, 2005.

\bibitem[Rif14]{Rifford14}
L.~Rifford.
\newblock {\em Sub-{R}iemannian geometry and optimal transport}.
\newblock SpringerBriefs in Mathematics. Springer, Cham, 2014.

\bibitem[San15]{Filippo15}
F.~Santambrogio.
\newblock {\em Optimal transport for applied mathematicians}, volume~87 of {\em
  Progress in Nonlinear Differential Equations and their Applications}.
\newblock Birkh\"auser/Springer, Cham, 2015.
\newblock Calculus of variations, PDEs, and modeling.

\bibitem[Son98]{Sontag98}
E.D. Sontag.
\newblock {\em Mathematical control theory}, volume~6 of {\em Texts in Applied
  Mathematics}.
\newblock Springer-Verlag, New York, second edition, 1998.
\newblock Deterministic finite-dimensional systems.

\bibitem[Vil03]{Villani03}
C.~Villani.
\newblock {\em Topics in optimal transportation}, volume~58 of {\em Graduate
  Studies in Mathematics}.
\newblock American Mathematical Society, Providence, RI, 2003.

\bibitem[Vil09]{Villani08}
C.~Villani.
\newblock {\em Optimal transport}, volume 338 of {\em Grundlehren der
  mathematischen Wissenschaften [Fundamental Principles of Mathematical
  Sciences]}.
\newblock Springer-Verlag, Berlin, 2009.
\newblock Old and new.

\end{thebibliography}

\end{document}